\documentclass{article}
\usepackage[utf8]{inputenc}
\usepackage{amsmath,amsfonts,amssymb,amsthm}
\usepackage{bbm}
\usepackage[margin=1in]{geometry}
\usepackage{color}
\usepackage{tikz}
\usepackage{tikz-cd}
\usepackage{authblk}
\usepackage{hyperref}

\usepackage{csquotes}
\usepackage[
  hyperref=auto,
  mincrossrefs=999,
  backend=biber,
  style=alphabetic,
  sorting=nyt
]{biblatex}
\makeatletter
\newtheorem*{rep@theorem}{\rep@title}
\newcommand{\newreptheorem}[2]{%
\newenvironment{rep#1}[1]{%
 \def\rep@title{#2 \ref{##1}}%
 \begin{rep@theorem}}%
 {\end{rep@theorem}}}
\makeatother

\newcommand{\norm}[1]{\left\lVert #1 \right\rVert}
\newcommand{\mor}{\text{Mor}}

\DeclareMathOperator{\onb}{onb}
\DeclareMathOperator{\mult}{\mathfrak{m}}
\DeclareMathOperator{\irr}{irr}

\DeclareMathOperator{\id}{Id}
\DeclareMathOperator{\op}{op}

\DeclareMathOperator{\tr}{Tr}
\DeclareMathOperator{\del}{\partial}

\DeclareMathOperator{\qaut}{QAut}
\DeclareMathOperator{\mtimes}{\otimes_{(M,\del)}}
\DeclareMathOperator{\End}{End}
\DeclareMathOperator{\bpi}{bpi}
\DeclareMathOperator{\adj}{Adj}

\DeclareMathOperator{\ovtimes}{\overline{\otimes}}

\newcommand{\vect}[1]{\mathfrak{v} \left( #1 \right)}
\newcommand{\Gelt}[3]{G_{#1}\left(\substack{ #2 \\ #3} \right)}
\newcommand{\Aelt}[3]{A_{#1}\left(\substack{ #2 \\ #3} \right)}

\newcommand{\Uelt}[2]{U\left(\substack{ #1 \\ #2} \right)}

\theoremstyle{definition}
\newtheorem{definition}[subsubsection]{Definition}
\newtheorem{definition/proposition}[subsubsection]{Definition/Proposition}
\newtheorem{theorem}[subsubsection]{Theorem}
\newreptheorem{theorem}{Theorem}
\newtheorem{lemma}[subsubsection]{Lemma}
\newtheorem{proposition}[subsubsection]{Proposition}
\newtheorem{corollary}[subsubsection]{Corollary}
\newtheorem{remark}[subsubsection]{Remark}
\newtheorem{notation}[subsubsection]{Notation}
\newtheorem{example}[subsubsection]{Example}

\title{Equivariant Tannaka-Krein reconstruction and quantum automorphism groups of discrete structures}

\author[1,2]{Lukas Rollier \thanks{KU Leuven, Department of Mathematics, Leuven (Belgium), email: lukas.rollier@kuleuven.be} \thanks{Supported by FWO research project G090420N of the Research Foundation Flanders and by FWO-PAS
research project VS02619N}}

\date{}

\begin{document}

\maketitle

\begin{abstract}
    We define quantum automorphism groups of a wide range of discrete structures. The central tool for their construction is a generalisation of the Tannaka-Krein reconstruction theorem. For any direct sum of matrix algebras $M$, and any concrete unitary $2$-category of finite type Hilbert-$M$-bimodules $\mathcal{C}$, under reasonable conditions, we construct an algebraic quantum group $\mathbb{G}$ which acts on $M$ by $\alpha$, such that the category of $\alpha$-equivariant corepresentations of $\mathbb{G}$ on finite type Hilbert-$M$-bimodules is equivalent to $\mathcal{C}$. Moreover, we explicitly describe how to get such categories from connected locally finite discrete structures. As an example, we define the quantum automorphism group of a quantum Cayley graph.
\end{abstract}

\section{Introduction} \label{sc: introduction}
Given any mathematical structure $S$, one may hope to define the quantum automorphism group $\qaut(S)$ of $S$. If it exists, it is the unique quantum group which admits an action on $S$ such that any other action of any other quantum group factors through a homomorphism to $\qaut(S)$. If we are only looking at classical symmetries, it is very easy to define the classical automorphism group. We do not have such luxury in the quantum setting, since it is difficult to speak rigorously about individual quantum automorphisms, which are usually defined to be irreducible representations of some $C^*$- or von Neumann algebra with extra structure. The problem then shifts to finding the right algebra and the right extra structure.

In the theory of compact groups, the Tannaka-Krein duality is a reconstruction theorem which, in essence, states that the concrete $C^*$-tensor category of finite dimensional representations of a compact group remembers everything about that compact group. In \cite{Woronowicz88}, this duality was extended by S.L. Woronowicz to compact quantum groups. As a result, when one is in need of a compact quantum group, e.g. the quantum automorphism group of a finite graph, it suffices to describe a desired representation category, and by the magic of this duality, the desired quantum group pops out. When working in the compact setting, very often the general theory is strong enough such that the Tannaka-Krein duality need only work in the background. All one needs is a unital $C^*$-algebra with the right kind of comultiplication, and these are easy to define. One should note, however, that in very many examples of compact quantum groups\footnote{For example all easy quantum groups, or more generally graph theoretic quantum groups as in \cite{MancinskaRoberson19, MancinskaRoberson2020published}.}, this $C^*$-algebra and comultiplication are defined by specifying a generating corepresentation along with some intertwiners between its tensor powers. It is clear that the entire representation category is implicitly present. When one extends the theory to more general locally compact quantum groups, as defined by Kustermans and Vaes \cite{KustermansVaes99, KustermansVaes03}, the Tannaka-Krein duality fails, and hence defining these is a far more difficult task. Unlike in the classical setting, it no longer suffices to find a $C^*$-algebra with the right kind of comultiplication: the existence of Haar integrals is not automatic for locally compact quantum groups, but an axiom, and this is often the hurdle when defining new examples.

By losing a bit of generality, the problem of finding new quantum groups may be made a great deal easier. In \cite{VanDaele98}, Van Daele introduced the class of algebraic quantum groups. This class contains all compact and all discrete quantum groups, as well as all classical groups which contain a compact open subgroup by \cite{LandstadVanDaele2008}. Through work of Kustermans and Van Daele \cite{KustermansVanDaele97}, all algebraic quantum groups are special cases of locally compact quantum groups in the framework of Kustermans and Vaes \cite{KustermansVaes99, KustermansVaes03}. The advantage of algebraic quantum groups over the more general locally compact ones is that, as the name suggests, one may work entirely on the algebraic level, with $*$-algebras without taking a completion to make them either $C^*$- or von Neumann algebras, and with algebraic tensor products. The existence of Haar integrals is still an axiom, and it is through the GNS-construction that one gets either the $C^*$- or von Neumann algebraic picture.

Using this framework, Vaes and the author, in \cite{RollierVaes2024}, defined the quantum automorphism group of any connected locally finite graph $\Pi$ with vertex set $I$ and edge set $E \subset I \times I$ as an algebraic quantum group. In order to overcome the hurdle of the Haar integrals, first a category $\mathcal{C}(\Pi)$ was introduced. This is a unitary $2$-category of finite type Hilbert-$\ell^{\infty}(I)$-bimodules which contains a lot of combinatorial information about the graph $\Pi$. Finite type here means that the finitely supported elements of $\ell^{\infty}(I)$ act by finite rank operators.  From this category $\mathcal{C}(\Pi)$, a reconstruction very reminiscent of the Tannaka-Krein reconstruction yielded an algebraic quantum group, which was shown to be the universal quantum group acting on $I$ while preserving the graph structure. It is this reconstruction which is generalised and solidified in this article. This happens mainly through the lens of two questions.

On the one hand, a central part of the reconstruction of $\qaut(\Pi)$ from $\mathcal{C}(\Pi)$ passes through the construction of a $*$-algebra $\mathcal{B}$. In fact, this is the $*$-algebra underlying a partial compact quantum group as in \cite{DeCommerTimmermann2015}. In that article, De Commer and Timmerman defined for any set $I$, from any unitary $2$-category of Hilbert-$\ell^{\infty}(I)$-bimodules $\mathcal{C}$, a partial compact quantum group, which may be seen as a quantum groupoid whose base space is $I$. This then raises the question: under which circumstances can we cut down from the partial compact quantum group to an action of a locally compact quantum group? This article gives at least a sufficient condition for this to work.

On the other hand, through work of many hands (see e.g. \cite{DuanSeveriniWinter13, MustoReutterVerdon, Weaver2012}), a notion of quantum graph was introduced as a generalisation of classical graphs. The idea of these is to replace the classical set of vertices of a graph with a discrete quantum space, i.e. a direct sum of matrix algebras equipped with a delta-form (see definition \ref{def: discrete quantum space}). Such quantum graphs appear naturally in quantum information theory (see references above). Moreover, in \cite{Wasilewski2023}, the different frameworks regarding quantum graphs were shown to be equivalent, and quantum Cayley graphs of discrete quantum groups were introduced, opening up the strong theorems of geometric group theory to the realm of discrete quantum groups. For these reasons, it is a natural question to wonder whether the definition of quantum automorphism group of a (connected locally finite) graph may be generalised to the setting of quantum graphs. More generally, one need not restrict oneself to (quantum) graphs. Given any discrete quantum space and some additional structure, it is a natural question to ask whether there exists some universal quantum group acting on this discrete quantum space in such a way that the structure is preserved. This quantum group could then rightfully be called the quantum automorphism group of this discrete structure. This article will provide an explicit construction of the quantum automorphism group of any discrete structure which is, in a suitable sense to be made precise below, connected and locally finite.

The structure of the article is as follows. We start by recalling some preliminary notions in section \ref{sc: preliminaries}: discrete quantum spaces on the one hand, and some very mild equivariant representation theory on the other. Next we move on to the central results.

Section \ref{sc: reconstruction theorem} contains the bulk of the work: a simultaneous generalisation and solidification of the Tannaka-Krein like reconstruction found in \cite{RollierVaes2024}. This is the content of theorem \ref{thm: unique quantum group action for equivariant category}, stated below. We start from any von Neumann algebra $M$, which is a direct sum of matrix algebras, and a concrete unitary $2$-category of finite type Hilbert-$M$-bimodules, $\mathcal{C}$, which contains enough morphisms in a sense to be made precise below. From this category $\mathcal{C}$, we then construct concretely an algebraic quantum group $\mathbb{G}$ with an action $\alpha$ on $M$ such that $\mathcal{C}$ is equivalent to the category of $\alpha$-equivariant corepresentations of $\mathbb{G}$ on finite type Hilbert-$M$-bimodules. Two points are important to mention here. Firstly, the extra structure we impose on $\mathcal{C}$, which restricts us from working with arbitrary concrete unitary $2$-categories of finite type Hilbert-$M$-bimodules, is very natural, and appears automatically for many actions of locally compact quantum groups on $M$. The author intends to make this statement more precise in a following paper. Secondly, if one takes $M$ to be $\mathbb{C}$, the field of complex numbers, theorem \ref{thm: unique quantum group action for equivariant category} gives precisely the Tannaka-Krein reconstruction from \cite{Woronowicz88}.

Before being able to state the theorem, we should briefly explain some concepts. By $M_0$, we mean the finitely supported elements of the algebra $M$, i.e. those which are contained in the algebraic direct sum of matrix algebras. Again, finite type Hilbert-$M$-bimodules are those for which $M_0$ acts by finite rank operators. Also, we say a general (not necessarily finite type) Hilbert-$M$-bimodule $H$ is covered by some category $\mathcal{C}$ if it can be written as an (infinite) direct sum of Hilbert-$M$-subbimodules which are contained in the category. Given $H,K$ two Hilbert-$M$-bimodules which are covered by some category $\mathcal{C}$, we say a densely defined $M$-bimodular linear (possibly unbounded) map $T: H \to K$ is covered by $\mathcal{C}$ if it restricts nicely to subbimodules of $H,K$ which are contained in $\mathcal{C}$, and these restrictions are morphisms in $\mathcal{C}$. Let us now state the theorem.

\begin{reptheorem}{thm: unique quantum group action for equivariant category}
    Given a concrete unitary $2$-category $\mathcal{C}$ of finite type Hilbert-$M$-bimodules, in the sense of definition \ref{def: concrete unitary 2-category of M-bimodules}, which covers $L^2(M^2)$ in the sense of definition \ref{def: Category covers bimodule}, and covers the morphisms $\mult: L^2(M^2) \to L^2(M)$ and $(\id \otimes \del \otimes \id): L^2(M^3) \to L^2(M^2)$ in the sense of definition \ref{def: covering morphisms}. There exists an algebraic quantum group $\mathbb{G} := (\mathcal{A},\Delta)$ and a right coaction $\alpha: M_0 \to \mathcal{M}(M_0 \otimes \mathcal{A})$ such that the category of finite type $\alpha$-equivariant corepresentations $\mathcal{C}(M \curvearrowleft \mathbb{G})$, as defined in \ref{def: category of equivariant corepresentations}, is equivalent to $\mathcal{C}$.
\end{reptheorem}

In order to use theorem \ref{thm: unique quantum group action for equivariant category} as a tool to define new examples of algebraic quantum groups, the focus shifts to finding good categories of Hilbert-$M$-bimodules. This problem is tackled in section \ref{sc: qaut of CLF discrete structures}. In \cite{RollierVaes2024} such categories were constructed from connected locally finite graphs, and also this construction is generalised in the current article. Keeping still a direct sum of matrix algebras $M$ fixed, we denote by $M_0$ the $*$-algebra of finitely supported elements in $M$. We define what it means for $M$-bimodular linear maps between tensor powers of $M_0$ to be locally finite, as well as what it means for sets of locally finite maps to be connected. From any connected set of locally finite maps, we can then define a concrete unitary $2$-category of finite type Hilbert-$M$-bimodules, which abides by the conditions of theorem \ref{thm: unique quantum group action for equivariant category}. This is the content of theorem \ref{thm: CLF discrete structures give good categories}, stated below this paragraph, and proven in section \ref{sbsc: building categories from discrete structures}. A prime example of a connected set of locally finite morphisms is the singleton $\{P\}$, where $P: \ell^2(I \times I) \to \ell^2(I \times I)$ is the projection onto the edges of some connected locally finite graph $\Pi = (I,E)$. This then completely recovers the construction of $\qaut(\Pi)$ from \cite{RollierVaes2024}. The theorem in this paper is, however, far more widely applicable, generalising immediately to (connected locally finite) weighted graphs, coloured graphs, directed graphs, simplicial complexes etc... As an example, in section \ref{sbsc: example quantum cayley graph}, we show that theorem \ref{thm: CLF discrete structures give good categories} applies to connected locally finite quantum graphs, and in particular to quantum Cayley graphs as defined in \cite{Wasilewski2023}. This is the content of theorem \ref{thm: quantum cayley graphs give quantum automorphism groups}, also stated below this paragraph.

\begin{reptheorem}{thm: CLF discrete structures give good categories}
    Denote by $\mathcal{F}_n := M_0^{\otimes n+1}$. Let $\mathcal{T}_{n,m}$ be a set of $M$-bimodular linear maps from $\mathcal{F}_m$ to $\mathcal{F}_n$ for every $n,m \in \mathbb{N}$ which are all locally finite in the sense of definition \ref{def: locally finite map}, and such that $\mathcal{T} = \bigcup_{n,m \geq 0} \mathcal{T}_{n,m}$ is connected in the sense of definition \ref{def: connected set of maps}. Then there exists a smallest concrete unitary $2$-category of Hilbert-$M$-bimodules $\mathcal{C}$, as in definition \ref{def: concrete unitary 2-category of M-bimodules}, which covers $L^2(M^n)$ for every $n \in \mathbb{N}$ in the sense of definition \ref{def: Category covers bimodule}, and which covers every $T \in \mathcal{T}$, as well as $(\id \otimes \del \otimes \id)$ and $\mult$ in the sense of definition \ref{def: covering morphisms}. In particular, the category $\mathcal{C}$ satisfies the conditions of theorem \ref{thm: unique quantum group action for equivariant category}. The resulting algebraic quantum group acts faithfully on $M$.
\end{reptheorem}

\begin{reptheorem}{thm: quantum cayley graphs give quantum automorphism groups}
    Let $\Pi$ be any connected locally finite quantum graph in the sense of definition \ref{def: CLF quantum graph} with vertex space $M$ and adjacency matrix $\adj$. There is a universal quantum group acting on $M$ by $\alpha$ in such a way that
    \[
    \alpha \circ \adj = (\adj \otimes \id) \circ \alpha.
    \]
    In particular, one may take $\Pi$ to be the quantum Cayley graph associated by \cite[Definition 5.1]{Wasilewski2023} to any $\Gamma$, discrete quantum group, and $S \subset \irr(\widehat{\Gamma})$ a finite symmetric generating set of irreducible corepresentations of $\widehat{\Gamma}$.
\end{reptheorem}

\subsection*{Acknowledgement}
The author would like to thank Stefaan Vaes for many fruitful discussions, important input, and the proofreading of this paper. We also thank Kenny De Commer for asking some pertinent questions leading to the inception of this undertaking.

\section{Preliminaries} \label{sc: preliminaries}
The reader is expected to be familiar with algebraic quantum groups (see \cite{VanDaele94, VanDaele98}) and Tannaka-Krein duality (see \cite{Woronowicz88, NeshveyevTuset2013}). Some familiarity with actions of locally compact quantum groups (see \cite{DeCommer16, Vaes01}) is advisable, though it will not be necessary to understand the paper. 

\subsection{Discrete quantum spaces, bimodules, and categorical data}

For the entire paper, $M$ will always denote a direct sum of matrix algebras
\[
M = \ell^{\infty}\bigoplus_{i \in I} M_i
\]
where each $M_i = \mathbb{C}^{d_i \times d_i}$ for some $d_i \in \mathbb{N}$. We will denote by $M_0$ its algebraic part, i.e. the finitely supported elements of $M$. We denote by $E^i_{k,l}$ the minimal partial isometry at position $k,l$ in block $i$, and by $1_i$ the unit of $M_i$. Then $M_0 = \text{span} \{E^i_{k,l} | i \in I, \text{ and } 1\leq k,l \leq d_i \}$. Recall that the multiplier algebra $\mathcal{M}(M_0)$ consists of (unbounded) elements $T = \sum_{i \in I} T_i$, where each $T_i \in M_i$. 

\begin{definition} \label{def: delta-form}
    A functional $\del: M_0 \to \mathbb{C}$ is called a delta-form if there exist a positive invertible element $\sigma = \sum_{i \in I}\sigma_i \in \mathcal{M}(M_0)$, with $\sigma_i \in M_i$, such that $\tr((\sigma_i)^{-1}) = d_i$ and $\del(x) = \tr_M(x \sigma)$, where $\tr_M$ is the Markov trace $\tr_M(E^i_{k,l}) = d_i \tr(E^i_{k,l}) = \delta_{k,l} d_i$. 
\end{definition}

\begin{example} \label{ex: markov trace is delta-form}
    Given $M$ a direct sum of matrix algebras, the markov trace $\tr_M$ as defined in definition \ref{def: delta-form} is the unique tracial delta-form
\end{example}

\begin{definition} \label{def: discrete quantum space}
    A pair $(M,\del)$ consisting of a direct sum of matrix algebras $M$, and a delta-form $\del$ is called a discrete quantum space. For a discrete quantum space $(M,\del)$, we denote by $\mathcal{F}_n$ the space of finitely supported elements in $M_0^{\otimes (n+1)}$. To clarify that we mean these as elements in a vector space more than elements in an algebra, we write $\vect{x_0 \otimes \ldots x_n} \in \mathcal{F}_n$. The spaces $\mathcal{F}_n$ come equipped with an inner product given by
    \[
    \langle \vect{x_0 \otimes \cdots \otimes x_n}, \vect{y_0 \otimes \cdots \otimes y_n } \rangle = \del(y_0^*x_0) \cdots \del(y_n^*x_n)
    \]
    We denote the Hilbert space closure of $\mathcal{F}_n$ with respect to this inner product as\footnote{Note that these Hilbert spaces very much depend on the delta-form $\del$ as well. We have chosen to suppress this in the already heavy notation.} $L^2(M^{n+1})$.
\end{definition}

We fix now for once and for all a discrete quantum space $(M,\del)$. This comes equipped with the (unbounded) element $\sigma \in \mathcal{M}(M_0)$ as in definition \ref{def: delta-form}. We denote for any $t \in \mathbb{R}$ by $\mu^t: M_0 \to M_0$ the homomorphism $x \mapsto \sigma^{-t} x \sigma^t$, and note that then $\del(xy) = \del(\mu(y)x)$ for all $x,y \in M_0$, where $\mu = \mu^1$. Recall that $M_i = \mathbb{C}^{d_i \times d_i}$, and note that an orthonormal basis of $\mathcal{F}_0$ is given by 
\begin{equation}
    \left\{ \frac{1}{d_i^{1/2}}\vect{E^i_{k,l} \sigma^{-1/2}} | i \in I \text{ and } 1 \leq k,l \leq d_i \right\}. \label{eq: onb for F_0}
\end{equation}
In many calculations we will be summing over any orthonormal basis of $\mathcal{F}_0$, denoted $\onb(\mathcal{F}_0)$, and we will make sure that it never matters over which one we sum. When in doubt, however, one can always choose this one.

\begin{proposition} \label{prop: multiplication has isometric adjoint}
    Consider the map $\mult: \mathcal{F}_1 \to \mathcal{F}_0: \vect{x \otimes y} \mapsto \vect{xy}$. This is a bounded linear map satisfying $\mult \mult^* = 1$, and hence it extends to a bounded linear map $L^2(M^2) \to L^2(M)$.
\end{proposition}
\begin{proof}
    Take $x \in M_0$ arbitrarily. Then 
    \begin{equation}
        \mult^*(x) = \sum_{y,z \in \onb(\mathcal{F}_0)} \langle \mult^*(x), \vect{y \otimes z} \rangle \vect{y \otimes z} = \sum_{y,z \in \onb(\mathcal{F}_0)} \del(z^*y^*x) \vect{y \otimes z} = \sum_{y \in \onb(\mathcal{F}_0)} \vect{y \otimes y^*x}. \label{eq: adjoint of multiplication}
    \end{equation}
    Applying $\mult$ to this expression, and filling in the orthonormal basis \eqref{eq: onb for F_0}, we find that the result equals
    \[
    \sum_{\substack{i \in I \\ 1 \leq  k,l \leq d_i }} \frac{1}{d_i} E^i_{k,l} \sigma^{-1} E^i_{l,k} x = \sum_{i \in I} \frac{\tr(\sigma_i^{-1})}{d_i} 1_i x = x
    \]
\end{proof}

\begin{definition} \label{def: Hilbert-M-bimodule + finite type}
    A Hilbert-$M$-bimodule $H$ is a Hilbert space equipped with $\lambda_H: M \to B(H)$ and $\rho_H: M^{\op} \to B(H)$, two normal unital $*$-homomorphisms whose ranges commute. For $\xi \in H$ and $x,y \in M$, we will often use the following notation.
    \[
    x \cdot \xi \cdot y :=  \lambda_H(x) \rho_H(y^{\op}) (\xi)
    \]
    We say $H$ is of finite type if $\lambda_H$ and $\rho_H$ map $M_0$ resp. $M_0^{\op}$ into the space of finite rank operators. Equivalently, for every $i,j \in I$, we must then have that $1_i \cdot H$ and $H \cdot 1_j$ are finite-dimensional Hilbert spaces. We will always denote $H_0 := M_0 \cdot H \cdot M_0$, which is a dense linear subspace of $H$.
\end{definition}

\begin{example} \label{ex: F_n as bimodules}
    Consider again the pre-Hilbert spaces $\mathcal{F}_n$ from definition \ref{def: discrete quantum space}. We equip these with a left- and right $M$-module structure by defining
    \begin{equation}
        a \cdot \vect{x_0 \otimes \cdots \otimes x_n} \cdot b := \vect{ax_0 \otimes \cdots \otimes x_n \mu^{-1/2}(b)}. \label{eq: bimodule structure on F_n}
    \end{equation}
    One checks that this gives a well-defined bimodule structure, which extends to $L^2(M^{n+1})$, thus making $L^2(M^{n+1})$ a Hilbert-$M$-bimodule in the sense of definition \ref{def: Hilbert-M-bimodule + finite type}.
\end{example}

\begin{definition} \label{def: dual Hilbert-M-bimodule}
    Given $H$, a Hilbert-$M$-bimodule, we can endow the dual Hilbert space $\overline{H}$ with a Hilbert-$M$-bimodule structure by defining
    \[
    x \cdot \overline{\xi} \cdot y := \overline{y^* \cdot \xi \cdot x^*}
    \]
    for any $\xi \in H$ and $x,y \in M$.
\end{definition}

\begin{example} \label{ex: F_n is its own dual}
    There is a natural isomorphism $\overline{\mathcal{F}_n} \cong \mathcal{F}_n$ given by
    \[
    \overline{\vect{x_0 \otimes \cdots \otimes x_n}} \mapsto \vect{\mu^{1/2}(x_n)^* \otimes \cdots \otimes \mu^{1/2}(x_0)^*}.
    \]
    As such we identify $L^2(M^{n+1}) \cong \overline{L^2(M^{n+1})}$.
\end{example}

\begin{definition} \label{def: M-bimodular map}
    Given $H,K$ two Hilbert-$M$-bimodules, we say a bounded linear map $T: H \to K$ is $M$-bimodular if $T \circ \lambda_H(x) \circ \rho_H(y^{\op}) = \lambda_K(x) \circ \rho_K(y^{\op}) \circ T$ for every $x,y \in M$. This may be written more concisely as
    \[
    T(x \cdot \xi \cdot y) = x \cdot T(\xi) \cdot y \text{ for all } x,y  \in M, \xi \in H.
    \]
    One may note that the $M$-bimodular morphisms in $B(L^2(M))$ are given precisely by multiplication with central elements of $M$. 
\end{definition}

For more general bimodules over von Neumann algebras, Alain Connes introduced the notion of their relative tensor product in \cite{Connes1980}. We recall this definition here, specified to our context, for later reference.

\begin{definition} \label{def: relative tensor product}
    Given $H,K$ two Hilbert-$M$-bimodules for some discrete quantum space $(M,\del)$, we define their $(M,\del)$-relative tensor product as follows.
    \begin{align*}
        H \mtimes K := \{\xi \in H \otimes K | \forall x \in M: (\rho_H(x^{\op}) \otimes 1) \xi = (1 \otimes \lambda_K(\mu^{1/2}(x))) \xi \}
    \end{align*}
    We explicitly view this as a closed subspace $H \mtimes K \subset H \otimes K$. This ordinary tensor product is naturally equipped with a Hilbert-$M$-bimodule structure by letting
    \[
    \lambda_{H \otimes K} = \lambda_H \otimes \id \text{ and } \rho_{H \mtimes K} = \id \otimes \rho_K
    \]
    and we define a Hilbert-$M$-bimodule structure on $H \mtimes K$ as the restriction of the one on $H \otimes K$. One easily checks that $H \mtimes K$ is finite type if $H$ and $K$ are.
    
    We denote by $P_{H \mtimes K} \in B(H \otimes K)$ the projection onto $H \mtimes K$, and for any $\xi \in H$ and $\eta \in K$, we denote $\xi \mtimes \eta := P_{H \mtimes K}(\xi \otimes \eta)$. For any Hilbert-$M$-bimodules $H_1,H_2,K_1,K_2$ and any two bounded linear maps $T: H_1 \to K_1$ and $S: H_2 \to K_2$, we define their relative tensor product as follows.
    \begin{equation}
        T \mtimes S := P_{K_1 \mtimes K_2} \circ (T \otimes S) \circ P_{H_1 \mtimes H_2} \label{eq: relative tensor product of linear maps}
    \end{equation}
    One may note that when $T,S$ happen to be $M$-bimodular, this is simply the restriction of $T \otimes S$ to $H_1 \mtimes H_2$.
\end{definition}

\begin{proposition} \label{prop: formula for projection onto relative tensor product}
    For $H,K$ two Hilbert-$M$-bimodules, the projection $P_{H \mtimes K}$ from definition \ref{def: relative tensor product} is given by
    \begin{equation}
    P_{H \mtimes K} := \sum_{y \in \onb(\mathcal{F}_0)} \rho_H(y^{\op}) \otimes \lambda_K(\mu^{1/2}(y^*)) \label{eq: projection onto relative tensor product}
    \end{equation}
    with convergence in the strong operator topology.
\end{proposition}
\begin{proof}
    Restricting the sum in \eqref{eq: projection onto relative tensor product} to $y  \in \onb(1_X \cdot \mathcal{F}_0)$ for some finite $X \subset I$, one easily checks that the result gives a projection in $B(H \otimes K)$. As the SOT-limit of increasing projections, $P_{H \mtimes K}$ is therefore well-defined. Take now $x \in M$ arbitrarily, and calculate as follows, using again the explicit orthonormal basis \eqref{eq: onb for F_0}.
    \begin{align*}
        (\rho_H(x^{\op}) \otimes 1) \circ P_{H \mtimes K} =& \; \sum_{\substack{i \in I \\ 1 \leq k,l \leq d_i }} \frac{1}{d_i} \rho_H((E^i_{k,l} \sigma^{-1/2} x)) \otimes \lambda_K(E^i_{l,k} \sigma^{-1/2}) \\
        =& \; \sum_{\substack{i \in I \\ 1 \leq k,l \leq d_i }} \frac{1}{d_i} \rho_H((E^i_{k,l} \sigma^{-1/2} )) \otimes \lambda_K(\sigma^{-1/2}x \sigma^{1/2} E^i_{l,k} \sigma^{-1/2}) \\
        =& \; (1 \otimes \lambda_K(\mu^{1/2}(x))) \circ P_{H \mtimes K}\\
    \end{align*}
    Here, we have used the fact that
    \[
    \sum_{1 \leq l \leq d_i} E^i_{k,l}x \otimes E^i_{l,j} = \sum_{1 \leq l \leq d_i} E^i_{k,l} \otimes xE^i_{l,j} \text{ for any } i \in I, 1 \leq k,j \leq d_i, x \in M 
    \]
    which is very easy to check.
\end{proof}

\begin{proposition} \label{prop: F_0 is monoidal unit}
    For any Hilbert $M$-bimodule $H$, we get that there are unitary $M$-bimodular isomorphisms
    \begin{align*}
        H \otimes_{(M,\del)} L^2(M) \cong H \cong L^2(M) \otimes_{(M,\del)} H
    \end{align*}
    given by the linear continuous extension of the maps
    \begin{align*}
        \phi_R&: H \otimes \mathcal{F}_0 \to H: \xi \otimes \vect{x} \mapsto \xi \cdot \mu^{1/2}(x) \\
        \phi_L&: \mathcal{F}_0 \otimes H \to H: \vect{x} \otimes \xi \mapsto x \cdot \xi.
    \end{align*}
\end{proposition}
\begin{proof}
    By direct calculation, one may check that the maps $\phi_R$ and $\phi_L$    are bounded linear maps, and hence extend to $H \otimes L^2(M)$ and $L^2(M) \otimes H$ respectively, with respective adjoints
    \[
    \phi_R^*: H \to H \otimes L^2(M): \xi \mapsto \sum_{x \in \onb(\mathcal{F}_0)} \xi \cdot x \otimes \vect{\mu^{1/2}(x)^*}
    \]
    and 
    \[
    \phi_L^*: H \to L^2(M) \otimes H: \xi \mapsto \sum_{x \in \onb(\mathcal{F}_0)} \vect{x} \otimes x^* \cdot \xi.
    \]
    By direct computation, one checks that $\phi_R \phi_R^* = \id_H$ while $\phi_R^* \phi_R = P_{H \mtimes L^2(M)}$, the projection from definition \ref{def: relative tensor product}. Similarly, $\phi_L \phi_L^* = \id_H$ and $\phi_L^* \phi_L = P_{L^2(M) \mtimes H}$. This establishes the required isomorphisms.
\end{proof}
Since one sees straightforwardly that $L^2(M^n) \cong L^2(M)^{\otimes n}$, proposition \ref{prop: F_0 is monoidal unit} yields isomorphisms
\begin{equation}
    L^2(M^{n+1}) \mtimes L^2(M^{m+1}) \cong L^2(M^{n+m+1}). \label{eq: isomorphisms tensor powers of M}
\end{equation}
Moreover, one may note that for $H = L^2(M)$, the maps $\phi_L$ and $\phi_R$ from the proof of proposition \ref{prop: F_0 is monoidal unit} are both equal to $\mult$.

\begin{definition} \label{def: concrete unitary 2-category of M-bimodules}
    Recall that we have fixed a von Neumann algebra $M = \ell^{\infty} \bigoplus_{i \in I} M_i$, which is a direct sum of matrix algebras indexed by some set $I$, and a delta form $\del$ on $M$ as in definition \ref{def: delta-form}. A concrete unitary $2$-category of finite type Hilbert-$M$-bimodules is a $2$-category $\mathcal{C}$ whose 1-cells, which we call objects, are finite type Hilbert-$M$-bimodules, such that the following hold. 
    \begin{enumerate}
        \item For any $H,K$ objects in $\mathcal{C}$, the $2$-cells connecting $H$ to $K$, which we call morphisms, are given by the vector space $\mor_{\mathcal{C}}(H,K)$, which consists of bounded linear $M$-bimodular maps, and is closed in the norm topology.
        
        \item The $0$-cells of the category are given by a partition $\mathcal{E}$ of $I$, i.e. $I = \cup_{a \in \mathcal{E}} I_a$. We denote $M_a := \bigoplus_{i \in I_a} M_i$. For every $a \in \mathcal{E}$, there is an irreducible object $L^2(M_a)$. We denote by $z_a: L^2(M) \to L^2(M_a)$ the projection onto this subspace.

        \item For any $H,K$ objects in $\mathcal{C}$, we have $\mor_{\mathcal{C}}(H,K)^* = \mor_{\mathcal{C}}(K,H)$.

        \item $\mathcal{C}$ is closed under the relative tensor product $\mtimes$ from definition \ref{def: relative tensor product}, both on objects and morphisms.

        \item $\mathcal{C}$ is closed under finite direct sums.\footnote{One endows the direct sum of two Hilbert-$M$-bimodules with the obvious Hilbert-$M$-bimodule structure. If the summands are finite type, so is the sum.}

        \item For any object $H$ in $\mathcal{C}$, the morphisms $\phi_L$ and $\phi_R$ from the proof of proposition \ref{prop: F_0 is monoidal unit} are morphisms in $\mathcal{C}$, i.e. $H \mtimes L^2(M_a)$ and $L^2(M_b) \mtimes H$ are $\{0\}$ for all but finitely many $a,b \in \mathcal{E}$, and we view $\phi_L, \phi_R$ as the appropriate restriction.
        
        \item For any object $H$ in $\mathcal{C}$, the Hilbert-$M$-bimodule $\overline{H}$ is also an object in $\mathcal{C}$, and there are morphisms $s_H \in \mor_{\mathcal{C}}(L^2(M), H \mtimes \overline{H})$ and $t_H \in \mor_{\mathcal{C}}(L^2(M), \overline{H} \mtimes H)$ satisfying the conjugate equations
        \begin{align}
            \phi_R& \circ (s_H^* \mtimes \id) \circ (\id \mtimes t_H) \circ \phi_R^* = \id_H \text{ and } \nonumber \\
            \phi_L& \circ (t_H^* \mtimes \id) \circ (\id \mtimes s_H) \circ \phi_L^* = \id_{\overline{H}}. \label{eq: conjugate equations}
        \end{align}
    \end{enumerate}
\end{definition}
\begin{remark} \label{rem: alternative definition unitary 2-category}
    Unitary $2$-categories, as originally defined in \cite{LongoRoberts97}, and elsewhere in the literature, typically do not allow for direct sums to be taken of objects which are supported on different $0$-cells. This is of course only a convention, and for our purpose it seemed more natural to allow for arbitrary finite direct sums, as this is of course possible in the equivariant corepresentation category of a coaction $M \curvearrowleft \mathbb{G}$ as defined in definition \ref{def: category of equivariant corepresentations}.
\end{remark}
\begin{remark} \label{rem: category depends on delta form}
    We will be speaking about concrete unitary $2$-categories of Hilbert-$M$-bimodules in the sense of definition \ref{def: concrete unitary 2-category of M-bimodules}. One may note that this terminology bears no reference to the delta-form $\del$, though it is important since the relative tensor product $\mtimes$ and the maps $\phi_L,\phi_R$ depend on it. This choice was made in order not to overburden the already heavy notation.
\end{remark}

\begin{lemma} \label{lem: projections on 0-cells are morphisms in category}
    Let $H$ be any object in a concrete unitary $2$-category $\mathcal{C}$ of Hilbert-$M$-bimodules. Recall from definition \ref{def: concrete unitary 2-category of M-bimodules}.2 the definition of $0$-cells in our category. Take $a,b \in \mathcal{E}$ arbitrarily, and denote by $z_a, z_b$ the corresponding minimal projections in $\End_{\mathcal{C}}(L^2(M)) \subset M$. Then $\lambda_H(z_a) \rho_H(z_b^{op}) \in \End_{\mathcal{C}}(H)$.
\end{lemma}
\begin{proof}
    Consider the morphisms $\phi_L$ and $\phi_R$ from the proof of proposition \ref{prop: F_0 is monoidal unit}, which are morphisms in $\mathcal{C}$ by definition. One checks by direct computation that 
    \[
    \lambda_H(z_a) \rho_H(z_b^{\op}) = \phi_R \circ (\phi_L \mtimes 1) \circ (z_a \mtimes \id_H \mtimes z_b) \circ (\phi_L^* \mtimes 1) \circ \phi_r.
    \]
    This proves the claim.
\end{proof}
In light of lemma \ref{lem: projections on 0-cells are morphisms in category} and remark \ref{rem: alternative definition unitary 2-category}, one could define a smaller unitary $2$-category whose objects are Hilbert-$M$-bimodules in $\mathcal{C}$ for which there exist some $a,b \in \mathcal{E}$ such that $\lambda_H(z_a) = \rho_H(z_b^{\op}) = \id_H$. This would then really be a unitary $2$-category in the sense of \cite{LongoRoberts97}. The picture is, however, equivalent.

\begin{lemma} \label{lem: homomorphism to finite dimensional C*-algebra}
    Let $H,K$ be objects in a concrete unitary $2$-category $\mathcal{C}$ of finite type Hilbert-$M$-bimodules as in definition \ref{def: concrete unitary 2-category of M-bimodules}. Suppose there is a $0$-cell $a \in \mathcal{E}$ of $\mathcal{C}$ such that $\lambda_H(z_a) = \id_H$ and $\lambda_K(z_a) = \id_K$. Then for any $i \in I_a$, we get that
    \begin{equation}
        \mor_{\mathcal{C}}(H,K) \to B(1_i \cdot H,1_i \cdot K): T \mapsto \lambda_K(1_i) \circ T \label{eq: homomorphism to finite dimensional C*-algebra}
    \end{equation}
    is an injective linear map. Moreover, when $H=K$, it is a homomorphism of $C^*$-algebras. Likewise, if there is some $0$-cell $b \in \mathcal{E}$ with $\rho_H(z_b^{\op}) = \id_H$ and $\rho_K(z_b^{\op}) = \id_K$, then for any $j \in I_b$
    \begin{equation*}
        \mor_{\mathcal{C}}(H,K) \to B(H \cdot 1_j,K \cdot 1_j): T \mapsto \rho_K(1_j^{\op}) \circ T 
    \end{equation*}
    is an injective linear map, and when $H=K$, it is a homomorphism of $C^*$-algebras.
\end{lemma}
\begin{proof}
    We show the first claim. The second may be handled analogously. Clearly, when $H=K$, the map in \eqref{eq: homomorphism to finite dimensional C*-algebra} is a homomorphism of $C^*$-algebras. We show that it is in general injective. Suppose therefore that we have $T \in \mor_{\mathcal{C}}(H,K)$ such that $\lambda_K(1_i) \circ T = 0$. Let $s_H, t_H$ be as in \eqref{eq: conjugate equations}, and note that we then have
    \[
    0 = s_H^* \circ (T^*\lambda_K(1_i)T \mtimes 1) s_H = \lambda_{L^2(M_a)}(1_i) s_H^* \circ (T^*T \mtimes 1) s_H.
    \]
    Since $s_H^* \circ (T^*T \mtimes 1) s_H \in \End_{\mathcal{C}}(L^2(M_a)) \cong \mathbb{C}$, we must then have that $s_H^* \circ (T^*T \mtimes 1) s_H = 0$. By positivity, this implies $(T \mtimes 1)s_H = 0$. Now, we can calculate, using \eqref{eq: conjugate equations}.
    \begin{align*}
        0 =& \; \phi^K_R \circ (T \otimes t_H^*) \circ (s_H \mtimes \id_H) \circ \phi^H_L = \phi^K_R (T \otimes 1) (\phi^H_R)^* = T.
    \end{align*}
    Here we write $\phi^K_R$ and $\phi^H_R$ to denote the morphisms from proposition \ref{prop: F_0 is monoidal unit} applied to $K$ and $H$ respectively. This proves the claim.
\end{proof}
\begin{corollary} \label{cor: finite dimensional intertwiner spaces}
    For any objects $H,K$ in a concrete unitary $2$-category of finite type Hilbert-$M$-bimodules, the morphism space $\mor_{\mathcal{C}}(H,K)$ is finite dimensional.
\end{corollary}
\begin{proof}
    By lemma \ref{lem: projections on 0-cells are morphisms in category}, we may assume that there are some $a,b \in \mathcal{E}$ such that $z_a \cdot H \cdot z_b = H$ and $z_a \cdot K \cdot z_b = K$. For any $i \in I_a$, the space $B(1_i \cdot H,1_i \cdot K)$ is finite dimensional, so by injectivity of the map \eqref{eq: homomorphism to finite dimensional C*-algebra}, we may conclude.
\end{proof}

\begin{definition} \label{def: Category covers bimodule}
    Given $\mathcal{C}$, a concrete unitary $2$-category of finite type Hilbert-$M$-bimodules in the sense of definition \ref{def: concrete unitary 2-category of M-bimodules}, and a Hilbert-$M$-bimodule $H$ which is not necessarily finite type. A covering of $H$ by $\mathcal{C}$ consists of a set $\{H_a | a \in A\}$ of objects in $\mathcal{C}$ and isometries $\iota_a: H_a \to H$ with mutually orthogonal ranges, for $a \in A$ some index set, such that $\bigoplus_{a \in A} \iota(H_a)$ is dense in $H$. We will very often suppress the isometries $\iota_a$ in the notation, and simply write $\bigoplus_{a \in A} H_a$ dense in $H$.
\end{definition}

\begin{remark} \label{rem: category covers L^2(M)}
    Note that by definition \ref{def: concrete unitary 2-category of M-bimodules}.2, every concrete unitary $2$-category of Hilbert-$M$-bimodules covers $L^2(M)$.
\end{remark}

\begin{proposition} \label{prop: covering bimodules closed under relative tensor product}
    Let $\mathcal{C}$ be a concrete unitary $2$-category of Hilbert-$M$-bimodules as defined in \ref{def: concrete unitary 2-category of M-bimodules}, and suppose we have $H,K$ two Hilbert-$M$-bimodules which are covered by $\mathcal{C}$, then also $H \mtimes K$ is covered by $\mathcal{C}$.
\end{proposition}
\begin{proof}
    Let $(H_a, \iota_a)_{a \in A}$ and $(K_b, \iota_b)_{b \in B}$ be the objects in $\mathcal{C}$ which realise the covering of $H$ and $K$ respectively. Since $\mathcal{C}$ is closed under the relative tensor product, we have that $(H_a \mtimes K_b)_{(a,b) \in A \times B}$ are objects in $\mathcal{C}$. Let now $\iota_{(a,b)} := \iota_a \mtimes \iota_b: H_a \mtimes K_b \to H \mtimes K$. We claim they realise the covering of $H \mtimes K$. Take $\xi \in H$ and $\eta \in K$ arbitrarily. Denote, for any finite $A_0 \subset A$, by $\xi_{A_0}$ the orthogonal projection of $\xi$ on $\bigoplus_{a \in A_0} \iota_a(H_a)$. Similarly denote $\eta_{B_0}$ for any finite $B_0 \subset B$. By the explicit description of the projection $H \mtimes K$ from proposition \ref{prop: formula for projection onto relative tensor product}, we find that $P_{\iota_a(H_a) \mtimes \iota_b(K_b)}$ is the restriction of $P_{H \mtimes K}$ to $\iota(H_a) \otimes \iota(K_b)$. It follows that $\xi_{A_0} \mtimes \eta_{B_0}$ converges to $\xi \mtimes \eta$ for $A_0$ and $B_0$ increasing finite sets in $A$ resp. $B$. It follows that indeed
    $\bigoplus_{(a,b) \in A \times B} \iota_{(a,b)}(H_a \mtimes K_b)$ is dense in $H \mtimes K$.
\end{proof}

\begin{definition} \label{def: basis of partial isometries}
    Given two objects $H,K$ in a concrete unitary $2$-category of finite type Hilbert-$M$-bimodules $\mathcal{C}$, with $K$ irreducible, we denote by $\bpi(H,K)$ any maximal set of partial isometries $V \in \mor_{\mathcal{C}}(H,K)$ such that
    \[
    VW^* = \left\{ \begin{array}{ll}
        \id_K & \text{ if } V=W \in \bpi(H,K) \\
        0 & \text{ otherwise.}
    \end{array} \right.
    \]
    When $H,K$ are not irreducible, we fix $\irr$, some maximal set of pairwise nonisomorphic irreducible objects of $\mathcal{C}$, and denote
    \[
    \bpi(H,K) = \bigcup_{L \in \irr} \{ V^*W | V \in \bpi(H,L), W \in \bpi(K,L) \}.
    \]
    When $H,K$ are not necessarily objects in $\mathcal{C}$, but are covered by it (definition \ref{def: Category covers bimodule}), we fix some $(H_a)_{a \in A}$ and $(K_b)_{b \in B}$ such that $\bigoplus_{a \in A} H_a$ is dense in $H$ and $\bigoplus_{b \in B} K_b$ is dense in $K$, and denote
    \[
    \bpi(H,K) = \bigcup_{(a,b) \in A \times B} \bpi(H_a, K_b).
    \]
    We read $\bpi$ as `basis of partial isometries'.
\end{definition}

\begin{definition} \label{def: covering morphisms}
    Let $\mathcal{C}$ be a concrete unitary $2$-category of finite type Hilbert-$M$-bimodules in the sense of definition \ref{def: concrete unitary 2-category of M-bimodules}, which covers two Hilbert-$M$-bimodules $H,K$ in the sense of definition \ref{def: Category covers bimodule}. Let $T: H \to K$ be any densely defined $M$-bimodular unbounded linear operator. We say $\mathcal{C}$ covers $T$ if there exist $(H_a,\iota_a)_{a \in A}$ and $(K_b,\iota_b)_{b \in B}$ objects in $\mathcal{C}$ which realise the covering such that for every $a \in A$, $\iota_a(H_a)$ lies in the domain of $T$ and for every $a \in A,b \in B$ we have $\iota_b^* \circ T \circ \iota_a \in \mor_{\mathcal{C}}(H_a,K_b)$.
\end{definition}

\begin{remark} \label{rem: category covers multiplication}
    Note that by definition \ref{def: concrete unitary 2-category of M-bimodules}.6, any concrete unitary $2$-category of Hilbert-$M$-bimodules covers $\mult: L^2(M) \mtimes L^2(M) \to L^2(M)$.
\end{remark}

\subsection{Equivariant corepresentation theory}

\begin{definition} \label{def: action on discrete quantum space}
    Let $\alpha: M \curvearrowleft \mathbb{G}$ be a right coaction of a locally compact quantum group on a direct sum of matrix algebras $M$. Let $\del$ be a delta-form on $M$ as in definition \ref{def: delta-form}, i.e. $(M,\del)$ is a discrete quantum space in the sense of definition \ref{def: discrete quantum space}. We say $\mathbb{G}$ acts by $\alpha$ on $(M,\del)$ if for every $x \in M_0$, and every $\omega \in L^1(\mathbb(G)) = L^{\infty}(\mathbb{G})_*$ we get that 
    \[
    (\del \otimes \omega) \alpha(x) = \del(x) \omega(1).
    \]
\end{definition}

\begin{definition} \label{def: equivariant bimodule corepresentation}
    Let $\mathbb{G}$ be a locally compact quantum group admitting a right coaction $\alpha: M \to M \ovtimes L^{\infty}(\mathbb{G})$. Let $H$ be a Hilbert-$M$-bimodule as in definition \ref{def: Hilbert-M-bimodule + finite type}. We say a corepresentation $U \in B(H) \ovtimes L^{\infty}(\mathbb{G})$ is equivariant with $\alpha$ if
    \begin{equation}
        U (\lambda_H(x) \otimes 1) = [(\lambda_H \otimes \id)\alpha(x)]U \text{ and } (\rho_H(x^{\op}) \otimes 1)U = [((\rho_H \circ \op) \otimes R)\alpha(x)]U \text{ for all } x \in M \label{eq: equivariance of corepresentation}
    \end{equation}
    where $\op: M \to M^{\op}$ denotes the obvious anti-$*$-homomorphism, and $R$ denotes the unitary antipode of $\mathbb{G}$.
\end{definition}

\begin{definition} \label{def: intertwiner of bimodule corepresentations}
    Given a locally compact quantum group $\mathbb{G}$ acting by $\alpha$ on $M$ as in definition \ref{def: equivariant bimodule corepresentation}, let $U,V$ be two $\alpha$-equivariant corepresentations of $\mathbb{G}$ on Hilbert-$M$-bimodules $H,K$ respectively. We say an $M$-bimodular linear map $T: H \to K$ intertwines the corepresentations if
    \[
    (T \otimes 1)U = V(T \otimes 1).
    \]
\end{definition}

\begin{definition} \label{def: category of equivariant corepresentations}
    Let $\alpha: M \curvearrowleft \mathbb{G}$ be a right coaction of a locally compact quantum group on a discrete quantum space $(M,\del)$, as in definition \ref{def: action on discrete quantum space}. We denote by $\mathcal{C}(M \curvearrowleft \mathbb{G})$ the category whose objects are all unitary $\alpha$-equivariant corepresentations of $\mathbb{G}$ on finite type Hilbert-$M$-bimodules (definition \ref{def: equivariant bimodule corepresentation}), and whose morphisms are all $M$-bimodular intertwiners of these corepresentations (definition \ref{def: intertwiner of bimodule corepresentations}).
\end{definition}

\section{The reconstruction theorem} \label{sc: reconstruction theorem}

This section will be concerned with the reconstruction theorem \ref{thm: unique quantum group action for equivariant category}. For completeness, we state it again in slightly more detail.
\begin{theorem} \label{thm: unique quantum group action for equivariant category}
    Given a concrete unitary $2$-category $\mathcal{C}$ of finite type Hilbert-$M$-bimodules, in the sense of definition \ref{def: concrete unitary 2-category of M-bimodules}, which covers $L^2(M^2)$ in the sense of definition \ref{def: Category covers bimodule}, and covers the morphisms $\mult: L^2(M^2) \to L^2(M)$ and $(\id \otimes \del \otimes \id): M_0^{ \otimes 3} \to M_0^{ \otimes 2}$, viewed as an unbounded linear map $L^2(M^3) \to L^2(M^2)$, in the sense of definition \ref{def: covering morphisms}. There exists an algebraic quantum group $(\mathcal{A},\Delta)$ and a right coaction $\alpha: M_0 \to \mathcal{M}(M_0 \otimes \mathcal{A})$ such that the following hold.
    \begin{enumerate}
        \item For every object $H$ in $\mathcal{C}$, there is a unitary $\alpha$-equivariant corepresentation $U_H$ of $(\mathcal{A},\Delta)$ on $H$. $U_H$ is irreducible as corepresentation when $H$ is irreducible in $\mathcal{C}$.
        \item For every $H,K$ objects in $\mathcal{C}$, the morphisms $\mor_{\mathcal{C}}(H,K)$ intertwine the corepresentations $U_H$ and $U_K$.
        \item For every equivariant corepresentation $U$ of $(\mathcal{A}, \Delta)$ on some finite type Hilbert-$M$-bimodule $H$, there exists an object $H'$ in $\mathcal{C}$ and a unitary $M$-bimodular linear map $T: H \to H'$ which intertwines $U$ and $U_{H'}$.
    \end{enumerate}
    That is, the category $\mathcal{C}$ is equivalent to the category of unitary $\alpha$-equivariant corepresentations of $(\mathcal{A},\Delta)$ on Hilbert-$M$-bimodules.
\end{theorem}
Hence, fix for the entire section a concrete unitary $2$-category of finite type Hilbert-$M$-bimodules, $\mathcal{C}$ (definition \ref{def: concrete unitary 2-category of M-bimodules}), which covers $L^2(M^2)$ (definition \ref{def: Category covers bimodule}). Note that by a successive application of proposition \ref{prop: covering bimodules closed under relative tensor product} in combination with \eqref{eq: isomorphisms tensor powers of M}, we then have that, up to unitary $M$-bimodular isomorphism, $\mathcal{C}$ covers $L^2(M^n)$ for every $n \geq 1$. Hence it makes sense to also assume that $\mathcal{C}$ covers $(1 \otimes \del \otimes 1): \mathcal{F}_2 \to \mathcal{F}_1: \vect{x \otimes y \otimes z} \mapsto \del(y) \vect{x \otimes z}$, viewed as an unbounded operator $L^2(M^3) \to L^2(M^2)$ (definition \ref{def: covering morphisms}).

\subsection{Properties of the unitary $2$-category $\mathcal{C}$} \label{sbsc: properties of the category}

In this subsection, we start by maximally exploiting the important two facts that
\begin{enumerate}
    \item $\mathcal{C}$ covers $L^2(M^2)$, and
    \item $\mathcal{C}$ covers $(1 \otimes \del \otimes 1): M_0^{\otimes 3} \to M_0^{\otimes 2}$, viewed as an unbounded operator $L^2(M^3) \to L^2(M^2)$.
\end{enumerate}
Moreover, by definition \ref{def: concrete unitary 2-category of M-bimodules}.6, the map $\mult: L^2(M^2) \to L^2(M)$ is also covered by $\mathcal{C}$.

\begin{definition} \label{def: algebra N, endomorphisms of L^2(M^2)}
    Let $(H_a)_{a \in A}$ be the objects in $\mathcal{C}$ which realise the covering of $L^2(M^2)$ (definition \ref{def: Category covers bimodule}), and view, for every $a,b \in A$, $\mor_{\mathcal{C}}(H_a,H_b) \subset B(L^2(M^2))$. We define the von Neumann algebra $\mathcal{N}$ consisting of all bounded linear operators $T: L^2(M^2) \to L^2(M^2)$ such that for every $a,b \in A$, we have $P_{a} T P_{b} \in \mor_{\mathcal{C}}(H_b,H_a)$, where $P_a,P_b$ denote the projections onto $H_a$ and $H_b$ respectively. Alternatively, one could say that $\mathcal{N}$ is the von Neumann algebra generated by the morphisms in $\mor_{\mathcal{C}}(H_a,H_b)$.
    \begin{equation}
        \mathcal{N} := \left( \bigcup_{a,b \in A} \mor_{\mathcal{C}}(H_a,H_b) \right)'' \subset B(L^2(M^2)) \label{eq: defining algebra N}
    \end{equation}
    Note that the set of all bounded $M$-bimodular linear maps in $B(L^2(M^2))$ may be identified with $M^{\op} \otimes M$ by acting by $\rho_{L^2(M)} \otimes \lambda_{L^2(M)}$ on the isomorphic space $L^2(M) \otimes L^2(M) \cong L^2(M^2)$, so $\mathcal{N}$ is isomorphic to a subalgebra of $M^{\op} \otimes M$.
\end{definition}

It is clear that for any projection in $\mathcal{N}$, whenever its range is a finite type Hilbert-$M$-bimodule, it must be an object in $\mathcal{C}$ We denote by $\mathcal{P} \subset \mathcal{N}$ the set of these finite type projections. Additionally, when $P,Q \in \mathcal{P}$ are such projections in $\mathcal{N}$, and $H_P, H_Q$ are the corresponding objects in $\mathcal{C}$, the morphism space $\mor_{\mathcal{C}}(H_P, H_Q)$ is given by
\begin{equation}
\mor_{\mathcal{C}}(H_P, H_Q) = Q \mathcal{N} P. \label{eq: morphism spaces in terms of N}
\end{equation}
Moreover, from lemma \ref{lem: projections on 0-cells are morphisms in category}, one straightforwardly gets that $z_a \otimes z_b \in \mathcal{N}$ for any $a,b \in \mathcal{E}$, the $0$-cells of our category, as defined in \ref{def: concrete unitary 2-category of M-bimodules}.2.

For any $S,T \in \mathcal{N}$, we will often use the identification $(S \otimes 1)(1 \otimes T) = S \mtimes T \in B(L^2(M^3))$, which makes sense considering $L^2(M^3) \cong L^2(M^2) \mtimes L^2(M^2)$ as in \eqref{eq: isomorphisms tensor powers of M}.

Lemma \ref{lem: central projections in N} below gives a characterisation of all irreducible objects in $\mathcal{C}$ which are contained in $L^2(M^2)$ up to isomorphism.

\begin{lemma} \label{lem: central projections in N}
    We denote by $\mathcal{J} \subset \mathcal{N}$ the set of minimal central projections in the von Neumann algebra $N$. The following hold.
    \begin{enumerate}
        \item For any $P \in \mathcal{J}$, $P \mathcal{N}$ is a type I factor, and for any $Q,R \in \mathcal{P}$ minimal projections such that $PQ =Q$ and $PR = R$, we have $R \mathcal{N} Q \neq \{0\}$, i.e. the ranges of $R$ and $Q$ are isomorphic objects in $\mathcal{C}$. 
        
        \item $\mathcal{J} \subset \mathcal{P} \subset \mathcal{N}$, i.e. minimal central projections have finite type ranges.
    \end{enumerate}
\end{lemma}
\begin{proof}
    \begin{enumerate}
        \item Since $P$ is a minimal central projection in the von Neumann algebra $\mathcal{N}$, we get that $P \mathcal{N}$ is a factor. Since $\mathcal{N}$ is isomorphic to a subalgebra of $M^{\op} \otimes M$, this means that $P \mathcal{N}$ is a type I factor. All minimal projections in a type I factor are equivalent, so there is a partial isometry $V \in P\mathcal{N}$ such that $VV^* = R$ and $V^*V = Q$. Hence $V \in R \mathcal{N} Q$.

        \item Take $P \in \mathcal{J}$ arbitrarily, and suppose by contradiction that $P$ is not finite type. Since all minimal projections in $\mathcal{N}$ are finite type, this means by the previous point that we can find an infinite family $\{Q_x | x \in X\} \subset \mathcal{N}$ which are all equivalent. Take now $i,j \in I$ such that $Q_x \vect{M_i \otimes M_j} \neq \{0\}$ for some $x \in X$. By bimodularity and equivalence of the projections $(Q_x)_{x \in X}$, we must then have that $Q_y(\vect{M_i \otimes M_j}) \neq \{0\}$ for all $y \in X$. However, this is in contradiction with the fact that $\vect{M_i \otimes M_j}$ is a finite dimensional space. It follows that $P$ had to have been finite type all along.
    \end{enumerate}
\end{proof}

\begin{notation} \label{not: morphism spaces of ind objects}
    We will slightly abuse notation, and define morphism spaces of bimodules which are not elements of our category. For $H,K$ Hilbert-$M$-bimodules which are covered by $\mathcal{C}$, let $(H_a)_{a \in A}$ and $(K_b)_{b \in B}$ realise the covering (definition \ref{def: Category covers bimodule}). We denote 
    \begin{equation}
        \mor_{\mathcal{C}}(H,K) := \text{span} \{ \mor_{\mathcal{C}}(H_a,K_b) | a \in A, b \in B\} \subset B(H,K). \label{eq: morphism space of ind object}
    \end{equation}
    Similarly, we take $\End_{\mathcal{C}}(H) := \mor_{\mathcal{C}}(H,H) $. Note that $\mathcal{N} = \End_{\mathcal{C}}(L^2(M^2))''$.
\end{notation}

Now, by assumption, the following are morphisms in $\mathcal{C}$ for any $P,Q \in \mathcal{P}$, $a \in \mathcal{E}$.
\begin{align*}
    \left\{ \begin{array}{c}
          \mult \circ (z_a \otimes 1), \\
          \mult \circ (1 \otimes z_a)
    \end{array} \right\}
    \subset& \; \mor_{\mathcal{C}}(L^2(M^2),L^2(M)) \\
    \left\{ \begin{array}{c}
         (1 \otimes \del \otimes 1) \circ  (P \mtimes Q),  \\
         P \circ (1 \otimes \del \otimes 1) \circ (Q \mtimes 1), \\
         P \circ (1 \otimes \del \otimes 1) \circ (1 \mtimes Q)
    \end{array} \right\}
    \subset& \; \mor_{\mathcal{C}}(L^2(M^3),L^2(M^2))
\end{align*}

\begin{proposition} \label{prop: frobenius reciprocity for finite type morphisms}
    For any $X,Y \in \End_{\mathcal{C}}(L^2(M^2))$, we also have
    \begin{equation}
        \left\{ \begin{array}{c}
            S_X := (X \mtimes 1) \circ (1 \otimes \vect{1} \otimes 1) \circ \mult^*,  \\
            T_Y := (1 \mtimes Y) \circ (1 \otimes \vect{1} \otimes 1) \circ \mult^* 
        \end{array} \right\}
        \subset \mor_{\mathcal{C}}(L^2(M),L^2(M^3))   \label{eq: frobenius reciprocity for finite type morphisms} 
    \end{equation}
    where $\vect{1}$ denotes $\sum_{i \in I} \vect{1_i}$. We claim therefore that $S_X$ and $T_Y$ become finite sums when applied to any $\vect{x}$, $x \in M_0$ making them well-defined linear maps $\mathcal{F}_0 \to \mathcal{F}_2$, and moreover, we claim that these linear maps are bounded.
\end{proposition}
\begin{proof}
    Take $x \in M_0$ arbitrarily. Then there are finitely many $i,j$ such that $1_i x 1_j$ is nonzero. Since $X$ and $Y$ are finite type, it follows directly that the sums defining $S_X(\vect{x})$ and $T_Y(\vect{x})$ have only finitely many nonzero terms. We now show boundedness of $S_X$, since $T_Y$ may be handled analogously.
    
    By considering $X$ as the composition of a bounded operator with a projection in $\mathcal{P}$, we may assume that $X$ is such a projection, since the composition of bounded operators is bounded. By considering finite linear combinations, and using corollary \ref{cor: finite dimensional intertwiner spaces}, we may assume that $X$ is a minimal projection. In particular, by lemma \ref{lem: projections on 0-cells are morphisms in category}, there are $a,b \in \mathcal{E}$ such that $X = X \circ (z_a \otimes z_b)$. Then, we note that
    \[
    S_X^* = \mult \circ (z_a \otimes \del \otimes 1) \circ (X \mtimes 1)
    \]
    Precomposing this with $\mult^*$, we obtain a map which is a bounded element of $\mor_{\mathcal{C}}(L^2(M^3),L^2(M^2))$ by assumption, so precomposing again with the bounded map $\mult \circ (z_a \otimes 1) \in \mor_{\mathcal{C}}(L^2(M^2),L^2(M))$, we get a bounded operator in $\mor_{\mathcal{C}}(L^2(M^3),L^2(M))$.
\end{proof}

Consider now for any $x,y \in M_0$, $i \in I$ and $1 \leq k,l \leq d_i$, the equality 
\begin{align}
    (1 \otimes \rho(x^{\op})& \otimes \lambda(y)) \circ (1 \otimes \vect{1} \otimes 1) \circ \mult^*(E^i_{k,l}) \nonumber \\
    =& \; (1 \otimes \vect{\mu^{-1/2}(x)} \otimes 1) \sum_{1 \leq t \leq d_i} d_i^{-1} E^i_{k,t} \sigma^{-1/2} \otimes y \sigma^{-1/2} E^i_{t,l} \nonumber \\
    =& \; (1 \otimes \vect{\mu^{-1/2}(x)} \otimes 1) \sum_{1 \leq t \leq d_i} d_i^{-1} E^i_{k,t} \sigma^{1/2} y \sigma^{-1} \otimes \sigma^{-1/2} E^i_{t,l}  \nonumber \\
    =& \; (\rho(\mu^{-1/2}(y)^{\op}) \otimes \lambda(\mu^{-1/2}(x)) \otimes 1) \circ (1 \otimes \vect{1} \otimes 1) \circ \mult^*(E^i_{k,l}) \nonumber
\end{align}
Where we used $\sigma$ and $\mu$ as defined in definition \ref{def: discrete quantum space} and the discussion below. Hence, defining
\begin{equation}
    \theta: M_0^{\op} \otimes M_0 \to M_0^{\op} \otimes M_0: x^{\op} \otimes y \mapsto \mu^{-1/2}(y)^{\op} \otimes \mu^{-1/2}(x) \label{eq: defining automorphism theta}
\end{equation}
yields an anti-automorphism\footnote{Note that this is not a $*$-anti-automorphism.} of $M_0^{\op} \otimes M_0$ which extends to the multiplier algebra. Identifying $\mathcal{N}$ with a subalgebra of $M^{\op} \otimes M$ as in definition \ref{def: algebra N, endomorphisms of L^2(M^2)}, we find that $T_X = S_{\theta(X)}$, as defined in proposition \ref{prop: frobenius reciprocity for finite type morphisms}.

Recall from lemma \ref{lem: central projections in N} the definition of $\mathcal{J} \subset \mathcal{P} \subset \mathcal{N}$. Since the elements of $\mathcal{J}$ are in one to one correspondence with the isomorphism classes of irreducible subobjects of $L^2(M^2)$ in $\mathcal{C}$, it follows that for any $P \in \mathcal{J}$, there is a unique $Q \in \mathcal{P}$ such that
\[
(P \mtimes Q) \circ (1 \otimes \vect{1} \otimes 1) \circ \mult^* \neq 0
\]
In particular, $\theta$ as defined in \eqref{eq: defining automorphism theta} permutes the central projections $P \in \mathcal{J}$. In particular, $\sigma^{-1} \otimes \sigma$ commutes with every $P \in \mathcal{J}$.

Take now $P \in \mathcal{J}$ arbitrarily, and let $H_P$ denote its range, which is an object in $\mathcal{C}$. We get that $(S_P,T_P)$ solve the conjugate equations \eqref{eq: conjugate equations} for $H_P$.

\begin{lemma} \label{lem: categorical trace}
    Formally denote $X = (1 \otimes \vect{1} \otimes 1) \circ \mult^*$ with adjoint $X^* = \mult \circ (1 \otimes \del \otimes 1)$, and $Y = (1 \otimes X \otimes 1) \circ X$. Then for $P \in \End_{\mathcal{C}}(L^2(M^2))$ and $Q \in \mor_{\mathcal{C}}(L^2(M^3))$, we get
    \begin{align}
        X^*(1 \otimes P)X(\vect{E^i_{k,l}}) =& \; d_i^{-2} (\tr \otimes \tr)(P \circ (\rho(\sigma^{\op}) \otimes \lambda(\sigma_i^{-1}))) \vect{E^i_{k,l}} \label{eq: categorical trace L^2(M^2)}\\
        Y^*(1 \otimes 1 \otimes  Q)Y(\vect{E^i_{k,l}}) =& \; d_i^{-2} (\tr \otimes \tr \otimes \tr) (Q \circ (\rho(\sigma^{\op}) \otimes \lambda(\sigma^{-1}) \rho(\sigma^{\op}) \otimes \lambda(\sigma_i^{-1})) \vect{E^i_{k,l}} \label{eq: categorical trace L^2(M^3)}
    \end{align}
    where $\tr$ denotes the operator trace.
\end{lemma}
\begin{proof}
    For the first equation \eqref{eq: categorical trace L^2(M^2)}, we calculate, for any $x,y \in M_0$ and $E^i_{k,l} \in M_0$.
    \begin{align*}
        X^*(1 \otimes \rho(x^{\op}) \otimes \lambda(y))X(\vect{E^i_{k,l}}) =& \; X \sum_{1 \leq s \leq d_i} d_i^{-1} \left( \vect{E^i_{k,s} \sigma^{-1/2}} \otimes \vect{\mu^{-1/2}(x)} \otimes \vect{y \sigma^{-1/2} E^i_{s,l}} \right) \\
        =& \; \del(x) \sum_{1 \leq s \leq d_i} d_i^{-1} \vect{E^i_{k,s} \sigma^{-1/2} y \sigma^{-1/2} E^i_{s,l}} \\
        =& d_i^{-2} \tr_M(x \sigma) \tr_M(y \sigma_i^{-1}) \vect{E^i_{k,l}}
    \end{align*}
    Here, $\tr_M$ denotes the markov trace from definition \ref{def: delta-form}. One easily checks that $\tr(\lambda(x)) = \tr_M(x) = \tr(\rho(x^{\op}))$, proving the first equation \eqref{eq: categorical trace L^2(M^2)}.
    
    Similarly, we calculate as follows for the second equation \eqref{eq: categorical trace L^2(M^3)}, for any $x,y,z,t \in M_0$ and $E^i_{k,l} \in M_0$, where $E_{z,t} \in B(L^2(M))$ denotes the rank one operator $\vect{a} \mapsto \del(t^*a) \vect{z}$ for any $a \in M_0$.
    \begin{align*}
        Y^*(1& \otimes 1 \otimes  \rho(x^{\op}) \otimes E_{z,t} \otimes \lambda(y))Y(\vect{E^i_{k,l}}) \\
        =& \; Y^* \sum_{\substack{1 \leq s \leq d_i \\ j \in I, 1 \leq p,q \leq d_j}} d_i^{-1}d_j^{-1} \vect{E^i_{k,s} \sigma^{-1/2}} \otimes \vect{E^j_{p,q} \sigma^{-1/2}} \\
        &\otimes \vect{\mu^{-1/2}(x)} \otimes \del(t^* \sigma^{-1/2}E^j_{q,p}) \vect{z} \otimes \vect{y \sigma^{-1/2} E^i_{s,l}} \\
        =& \; \del(x)d_i^{-1} \sum_{\substack{1 \leq s \leq d_i \\ j \in I, 1 \leq p,q \leq d_j}} d_j^{-1} \del(t^* \sigma^{-1/2} E^j_{q,p}) \del(E^j_{p,q} \sigma^{-1/2} z) \vect{E^i_{k,s} \sigma^{-1/2} y \sigma^{-1/2} E^i_{s,l}} \\
        =& \; d_i^{-2} \tr_M(x \sigma) \tr_M(y \sigma_i^{-1}) \sum_{r \in \onb(\mathcal{F}_0)} \langle \vect{\mu^{-1}(t^*)}, \vect{r} \rangle \langle \vect{z}, \vect{r^*} \rangle \\
        =& \; d_i^{-2} \tr_M(x \sigma) \tr_M(y \sigma_i^{-1}) \langle \vect{z}, \vect{\mu(t)} \rangle \\
        =& \; d_i^{-2} \tr_M(x \sigma) \tr_M(y \sigma_i^{-1}) \del(\mu(t)^* z)
    \end{align*}
    Since one now also readily checks that $\tr(E_{z,t}) = \del(t^*z)$, the second formula is also proven.
\end{proof}

\begin{definition} \label{def: element gamma}
    One checks that for any $P \in \mathcal{J}$, as defined in lemma \ref{lem: central projections in N}, the maps $(S_P,T_P)$ from proposition \ref{prop: frobenius reciprocity for finite type morphisms} solve the conjugate equations for the range $H_P$ of $P$. Define the unique positive invertible $\gamma_P \in \End_{\mathcal{C}}(H_P)$ such that the pair
    \[
     \left( (\gamma_P^{-1/2} \otimes 1)S_P, (1 \otimes \gamma_P^{1/2})T_P \right)
    \]
    is a standard solution to the conjugate equations. Recall that $\mathcal{N}$, the von Neumann algebra generated by $\mor_{\mathcal{C}}(L^2(M^2))$, is a direct sum of matrix algebras. Hence, $\mathcal{M}(\mor_{\mathcal{C}}(L^2(M^2)))$ is the direct product of these matrix algebras. Then let $\gamma \in \mathcal{M}(\End_{\mathcal{C}}(L^2(M^2)))$ be the unique element such that $\gamma P = P \gamma = \gamma_P$ for any $P \in \mathcal{J}$. We can identify $\gamma$ with an unbounded element of $\mathcal{M}(M_0^{\op} \otimes M_0)$, and with an unbounded operator on $L^2(M^2)$.
\end{definition}

\begin{theorem} \label{thm: explicit description of gamma}
    Let $\gamma$ be as defined in definition \ref{def: element gamma}. Then there exists a positive invertible element $\delta \in \mathcal{M}(M_0)$, unique up to a scalar multiple, which commutes with $\sigma$ and satisfies
    \[
    \gamma = \rho(\delta^{\op}) \otimes \lambda(\delta^{-1}) \in \mathcal{M}(\End_{\mathcal{C}}(L^2(M^2))).
    \]
\end{theorem}
\begin{proof}
    Take $a,b,c \in \mathcal{E}$ arbitrarily, and $P_1,P_2,P \in \mathcal{J}$ such that $P_1 = P_1 \circ (z_a \otimes z_b)$, $P_2 = P_2(z_b \otimes z_c)$ and $P = P \circ (z_a \otimes z_c)$. Denote $H_{1}, H_{2}, H$ their ranges. Recall that the pair
    \[
    \left( (\gamma_P^{-1/2} \otimes 1)S_P, (1 \otimes \gamma_P^{1/2})T_P \right)
    \]
    defined in proposition \ref{prop: frobenius reciprocity for finite type morphisms} and definition \ref{def: element gamma}, is a standard solution to the conjugate equations for $H$, and similarly for $H_{1},H_{2}$. It follows that we can calculate the categorical trace on $\End_{\mathcal{C}}((H_1 \mtimes H_2) \oplus H)$ as
    \begin{align}
        \begin{pmatrix}
            A & V \\
            W^* & B
        \end{pmatrix} \mapsto \begin{array}{rl}
             &Y^*(1 \otimes 1 \otimes  (\gamma_{P_1} \mtimes \gamma_{P_2})^{1/2}A(\gamma_{P_1} \mtimes \gamma_{P_2})^{1/2})Y  \\
             +& X^*(1 \otimes \gamma_P^{1/2}B\gamma_P^{1/2})X
        \end{array} \in \End_{\mathcal{C}}(L^2(M_c)) \label{eq: categorical trace involving gamma 1}
    \end{align}
    where $X,Y$ denote the maps from lemma \ref{lem: categorical trace}. Denote $\sigma_1 := \rho(\sigma^{\op}) \otimes \lambda(\sigma^{-1})$ and $\sigma_2 := \rho(\sigma^{\op}) \otimes \lambda(\sigma^{-1}) \rho(\sigma^{\op}) \otimes \lambda(\sigma^{-1})$. Then, using lemma \ref{lem: categorical trace}, we may compute \eqref{eq: categorical trace involving gamma 1} more explicitly, and find that the functional 
    \begin{align*}
        \begin{pmatrix}
            A & V \\
            W^* & B
        \end{pmatrix} \mapsto \begin{array}{rl}
             &d_i^{-2} (\tr \otimes \tr \otimes \tr) (A (\gamma_{P_1} \mtimes \gamma_{P_2})^{1/2}\sigma_2(\gamma_{P_1} \mtimes \gamma_{P_2} \rho(1_i^{\op}))^{1/2})   \\
             +& d_i^{-2} (\tr \otimes \tr)(B \circ \gamma_P^{1/2} \sigma_1 \gamma_P^{1/2} \rho(1_i^{\op}))
        \end{array} \text{ for any } i \in I_c 
    \end{align*}
    is tracial, faithful, and does not depend on the choice of $i \in I_c$. Consider this weight
    \begin{equation}
        \Gamma := \begin{pmatrix}
            (\gamma_{P_1} \mtimes \gamma_{P_2})^{1/2} \sigma_2 (\gamma_{P_1} \mtimes \gamma_{P_2})^{1/2} & 0 \\
            0 & \gamma_P^{1/2} \sigma_1 \gamma_P^{1/2}
        \end{pmatrix} \label{eq: weight for categorical space in proof of form gamma}
    \end{equation}
    which we will call $\Gamma$ for the time being. Recall the definition of $\theta$ from \eqref{eq: defining automorphism theta}, and note that since $\theta^2$ is an automorphism of $\End_{\mathcal{C}}(L^2(M^2))$, we must have that $\sigma_1 \End_{\mathcal{C}}(L^2(M^2)) \sigma_1^{-1} = \End_{\mathcal{C}}(L^2(M^2))$. Similarly, we get that $\sigma_2 \End_{\mathcal{C}}(L^2(M^3)) \sigma_2^{-1} = \End_{\mathcal{C}}(L^2(M^3))$. Now, since $(1 \otimes \del \otimes 1): L^2(M^3) \to L^2(M^2)$ is a surjective unbounded operator which is covered by $\mathcal{C}$ (definition \ref{def: covering morphisms}), we get that
    \[
    \mor_{\mathcal{C}}(L^2(M^3),L^2(M^2)) = (1 \otimes \del \otimes 1) \circ \End_{\mathcal{C}}(L^2(M^3))
    \]
    Since $\sigma_1 \circ (1 \otimes \del \otimes 1) \sigma_2^{-1} = (1 \otimes \del \otimes 1)$, it follows that we also get
    \[
    \sigma_1 \mor_{\mathcal{C}}(L^2(M^3),L^2(M^2)) \sigma_2^{-1} = \mor_{\mathcal{C}}(L^2(M^3),L^2(M^2)).
    \]
    Therefore, the weight $\Gamma$ from \eqref{eq: weight for categorical space in proof of form gamma} normalises $\End_{\mathcal{C}}((H_1 \mtimes H_2) \oplus H)$. Since the functional in \eqref{eq: categorical trace involving gamma 1} must be tracial and faithful, we get that $\Gamma$ commutes with $\End_{\mathcal{C}}((H_1 \mtimes H_2) \oplus H)$. In particular, since $\gamma_P \in \End_{\mathcal{C}}(H)$, we get that $\gamma_P^{-1/2}$ commutes with $\gamma_P^{1/2} \sigma_1 \gamma_P^{1/2}$, so $\sigma_1$ commutes with all powers of $\gamma_P$. Likewise, $\sigma_2$ commutes with all powers of $\gamma_{P_1} \mtimes \gamma_{P_2}$.

    Since $\Gamma$ commutes with $\End_{\mathcal{C}}((H_1 \mtimes H_2) \oplus H)$, we get as well that for any $V \in \mor_{\mathcal{C}}(H,H_1 \mtimes H_2)$, 
    \[
    V \circ \gamma_P \circ \sigma_1  = (\gamma_{P_1} \mtimes \gamma_{P_2}) \circ \sigma_2 \circ V
    \]
    In particular, this holds for $V = (P_1 \mtimes P_2) \circ (1 \otimes \vect{1} \otimes 1) \circ P $. Summing the resulting equality over all $P_1,P_2,P \in \mathcal{J}$, we get that for any $i,j,k \in I$, there is an equality of matrices
    \begin{align}
        (1_i \otimes 1_j \otimes 1_k) \circ (\gamma \sigma_1)_{(13)} = (1_i \otimes 1_j \otimes 1_k) \circ (\gamma \mtimes \gamma) \circ \sigma_2. \label{eq: gamma sigma is tensor square}
    \end{align}
    Here, we denote $\gamma = \sum_{P \in \mathcal{J}} \gamma_P \in \mathcal{M}(\End(L^2(M^2)))$ as in definition \ref{def: element gamma}. Moreover, we use leg numbering notation $(a \otimes b)_{(13)} = a \otimes 1 \otimes b$.

    Write now, for any $i,j \in I$, $(1_i \otimes 1_j) \circ \gamma \sigma_1 = \sum_{s \in S_{i,j}} 1_i A_s^{\op} \otimes 1_j B_s$, for some $A_s,B_s \in \mathcal{M}(M_0)$, and $S_{i,j}$ some finite index set. Note that we have once more identified $\End_{\mathcal{C}}(L^2(M^2))$ with a subalgebra of $M^{\op} \ovtimes M$. Then we get by \eqref{eq: gamma sigma is tensor square}, remembering that $\sigma_2 = \sigma_1 \mtimes \sigma_1$, that for any $i,j,k \in I$
    \begin{align}
        \sum_{r \in S_{i,k}} 1_i \mu^{-1/2}(A_r) \otimes \vect{1_j} \otimes  1_k B_r = \sum_{s \in S_{i,j}, t \in S_{j,k}} 1_i \mu^{-1/2}(A_s) \otimes B_s \mu^{-1/2}(A_t) 1_j \otimes B_t 1_k \label{eq: formula for decomposition of gamma sigma}
    \end{align}
    Denote now, for any $i,j \in I$, $X_{i,j} := \sum_{s \in S_{i,j}} 1_i A_s \otimes 1_j B_s = (\op \otimes \id) [(1_i \otimes 1_j) \circ \gamma \sigma_1]$. By invertibility of $\gamma \sigma_1$, $X_{i,j} \neq 0$ for any $i,j \in I$. The formula \eqref{eq: formula for decomposition of gamma sigma} now says that
    \begin{equation}
    (X_{i,j} \otimes 1)(1 \otimes \sigma_j^{1/2} \otimes 1)(1 \otimes X_{j,k}) (1 \otimes \sigma_j^{-1/2} \otimes 1) = (1 \otimes 1_j \otimes 1) \circ (X_{i,k})_{1,3} \text{ for all } i,j,k \in I. \label{eq: formula for X_i,j in proof of form gamma}
    \end{equation}
    Fix $i_0,k_0 \in I$ and $\omega_2 \in M_{k_0}^*$ such that $(\id \otimes \omega_2) X_{i_0,k_0} \neq 0$. Then fix $\omega_1 \in M_{i_0}^*$ such that $(\omega_1 \otimes \omega_2) X_{i_0,k_0} = 1$. It then follows from \eqref{eq: formula for X_i,j in proof of form gamma} that 
    \[
    ((\omega_1 \otimes \id)X_{i_0,j}) \sigma_j^{1/2} ((\id \otimes \omega_2)X_{j,k_0}) \sigma^{-1/2} = 1_j \text{ for any } j \in I,
    \]
    so that $(\omega_1 \otimes \id)X_{i_0,j}$ and $(\id \otimes \omega_2)X_{j,k_0}$ are invertible for every $j \in I$. Then \eqref{eq: formula for X_i,j in proof of form gamma} says that
    \[
    X_{i,j} (1 \otimes \sigma_j^{1/2}) (1 \otimes (1 \otimes \omega_2)(X_{j,k_0})) (1 \otimes \sigma_j^{-1/2}) = (\id \otimes \omega_2)(X_{i,k_0}) \otimes 1 \text{ for any } i,j \in I.
    \]
    Define $\alpha \in \mathcal{M}(M_0)$ such that $\alpha 1_i = (\id \otimes \omega_2)(X_{i,k_0})$ for any $i \in I$, and remember that $\alpha$ is invertible. We have now shown that
    \[
    X_{i,j} = \alpha 1_i \otimes \sigma^{1/2} \alpha^{-1} \sigma^{-1/2} 1_j \text{ for any } i,j \in I.
    \]
    This then means that
    \[
    \gamma \sigma_1 = \alpha^{\op} \otimes \sigma^{1/2} \alpha^{-1} \sigma^{-1/2}.
    \]
    Now write $\gamma = \beta_1^{\op} \otimes \beta_2$ for some $\beta_1,\beta_2 \in \mathcal{M}(M_0)$. We know that $\gamma$ is positive and invertible. Hence, for every $i,j  \in I$, we get
    \[
    0 < (\tr \otimes \tr) (\gamma(1_i \otimes 1_j)) = \tr(\beta_1 1_i) \tr(\beta_2 1_j)
    \]
    so $\tr(\beta_1 1_i) \neq 0$ and $\tr(\beta_2 1_j) \neq 0$ for every $i,j \in I$. Fix $j_0 \in I$, and define
    \[
    \gamma_1 := \tr(\beta_2 1_{j_0}) \beta_1 \text{ and } \gamma_2 := \tr(\beta_2 1_{j_0})^{-1} \beta_2 
    \]
    Then $\gamma = \gamma_1^{\op} \otimes \gamma_2$. We will show that $\gamma_1$ and $\gamma_2$ are positive and invertible.

    For every $i \in I$, $(\id \otimes \tr)(\gamma(1_i \otimes 1_{j_0}))$ is positive and invertible. Since $\tr(\gamma_2 1_{j_0}) = 1$, it follows that $\gamma_1 1_i$ is positive and invertible for all $i \in I$. Hence, $\gamma_1$ is positive invertible. Then it follows that $\tr(\gamma_1 1_i) > 0$ for all $i \in I$. Then since $(\tr \otimes \id)(\gamma(1_i \otimes 1_j))$ is positive invertible for every $i,j \in I$, we must also get that $\gamma_2 1_j$ is positive invertible for every $j \in I$.  It follows that $\gamma_2$ is positive invertible.

    Since $\gamma$ commutes with $\sigma_1$, we must get that $\gamma_1^{\op} \otimes \gamma_2$ commutes with $\sigma^{\op} \otimes \sigma^{-1}$.  Therefore, there exists some nonzero scalar $c \in \mathbb{C}$ such that
    \begin{align*}
        \sigma \gamma_1 =& \; c \gamma_1 \sigma \\
        \sigma \gamma_2 =& \; c^{-1} \gamma_2 \sigma.
    \end{align*}
    Now, $\tr(\sigma \gamma_1 1_i) = \tr(\sigma^{1/2} \gamma_1 \sigma^{1/2} 1_i) > 0$, and also $\tr(\sigma \gamma_1 1_i) = \tr(\gamma_1 \sigma 1_i) = c^{-1} \tr(\sigma \gamma_1 1_i)$. It follows that $c = 1$, so that $\gamma_1$ and $\gamma_2$ both commute with $\sigma$.

    Now recall that $X_{i,j} = \sigma \gamma_1 1_i \otimes \sigma^{-1} \gamma_2 1_j$, and then use \eqref{eq: formula for X_i,j in proof of form gamma} to find that
    \[
    \sigma \gamma_1 1_i \otimes \sigma^{-1} \gamma_2 \sigma_j^{1/2} \sigma \gamma_1 \sigma_j^{-1/2} \otimes \sigma^{-1} \gamma_2 = \sigma \gamma_1 1_i \otimes 1_j \otimes \sigma^{-1} \gamma_2 1_k \text{ for any } i,j,k \in I.
    \]
    It follows that $\gamma_2 = \gamma_1^{-1}$. Take now $\delta = \gamma_1$.
\end{proof}

\begin{definition} \label{def: Omega}
    Let $H$ be any object in $\mathcal{C}$, and let $(s_H,t_H)$ be a standard solution to the conjugate equations. Recall from definition \ref{def: Hilbert-M-bimodule + finite type} that we write $H_0 := M_0 \cdot H \cdot M_0$. We define the (possibly unbounded) operator $\Omega_H: H_0 \to H_0$ given by
    \begin{equation}
    \Omega_H(\xi) = (\del \circ t_H^* \otimes \id)(\id \otimes \xi \otimes \id) t_H(\vect{1}) \label{eq: defining Omega_H}
    \end{equation}
    where we interpret $\vect{1}$ as an infinite sum over $(\vect{1_i})_{i \in I}$, and filling this sum into the expression \eqref{eq: defining Omega_H}, only finitely many nonzero terms remain for each fixed $\xi \in H_0$.
\end{definition}
\begin{proposition} \label{prop: properties of the operators Omega}
    The operators $\Omega_H$ from definition \ref{def: Omega} are well-defined, do not depend on the choice of standard solution to the conjugate equations, and satisfy the following for any objects $H,K$ in $\mathcal{C}$.
    \begin{enumerate}
        \item For any $x,y \in M_0$ and $\xi \in H_0$, we get that $\Omega_H(x \cdot \xi \cdot y) = \mu(x) \cdot \Omega_H(\xi) \cdot \mu(y)$.
        
        \item $\Omega_H$ is invertible with inverse
        \[
        \Omega_H^{-1}: H_0 \to H_0: \xi \mapsto (\id \otimes \del \circ s_H^*)(\id \otimes \xi \otimes \id)s_H(\vect{1})
        \]
        which is also well-defined.

        \item $\Omega_H \otimes \Omega_K$ restricts to $H \mtimes K$, and this restriction equals $\Omega_{H \mtimes K}$.
        
        \item For any morphism $V \in \mor_{\mathcal{C}}(H,K)$, we have $V \circ \Omega_H = \Omega_K \circ V$.

        \item Recall the definition of $\mathcal{J}$ from lemma \ref{lem: central projections in N}, and $\gamma$ from definition \ref{def: element gamma}. Take $P \in \mathcal{J}$ arbitrarily, and denote $H_P$ its range. Then for any $\vect{x \otimes y} \in (H_P)_0$, we get $\Omega_{H_P}(\vect{x \otimes y}) = \gamma_P \vect{\mu(x) \otimes \mu(y)}$.
    \end{enumerate}
\end{proposition}
\begin{proof}
    The well-definedness of $\Omega_H$ will follow from $1.$ since for any central projections $p,q \in M_0$, $\Omega_H$ restricts to a linear map on the finite dimensional Hilbert space $p \cdot H \cdot q$. For now, fix for every $H$ a solution to the conjugate equations $(s_H,t_H)$. From point 4. it will follow that the choice is of no consequence, since the operator $\Omega_H'$, defined by any other choice $(s_H',t_H')$, would satisfy $\Omega_H' = U^* \circ \Omega_H \circ U$ for some unitary $U \in \End_{\mathcal{C}}(H)$.
    \begin{enumerate}
        \item Recall definition \ref{def: relative tensor product}. Since $t_H$ maps $L^2(M)$ into $\overline{H} \mtimes H$, we get that
        \[
        t_H^* \circ (1 \otimes \lambda_H(x)) = [(1 \otimes \lambda_H(x^*))t_H]^* = [(\rho_{\overline{H}}(\mu^{-1/2}(x^*)) \otimes 1)t_H]^* = t_H^* (\rho_{\overline{H}}(\mu^{1/2}(x)) \otimes 1)
        \]
        Hence, calculating similarly on the other side, we do indeed find that $\Omega_H(x \cdot \xi) = \mu(x) \cdot \Omega_H(\xi)$. On the other hand, by bimodularity of $t_H$, we can calculate
        \begin{align*}
            \Omega_H(\xi \cdot y) =& \; (\del \circ \rho(y^{\op}) \circ t_H^* \otimes \id) \circ (\id \otimes \xi \otimes \id) \circ t_H(\vect{1}) \\
            =& \; (\del \circ \lambda(\mu^{1/2}(y)) \circ t_H^* \otimes \id) \circ (\id \otimes \xi \otimes \id) \circ t_H(\vect{1}) \\
            =& \; (\del \circ t_H^* \otimes \id) \circ (\id \otimes \xi \otimes \id) \circ t_H \left(\vect{\mu^{1/2}(y)} \right) \\
            =& \; \Omega_H(\xi) \cdot \mu(y)
        \end{align*}

        \item The well-definedness of $\Omega_H^{-1}$ can be handled completely analogously to that of $\Omega_H$. They are each others inverse since $(s_H,t_H)$ are a solution to the conjugate equations.

        \item Let $\phi_L: L^2(M) \mtimes H \to H$ and $\phi_R: K \mtimes L^2(M) \to K$ denote the maps as in the proof of proposition \ref{prop: F_0 is monoidal unit}. Then the pair
        \[
        \left( (\id \otimes s_K \otimes \id) \circ (\id \otimes \phi_L^*) \circ s_H , (\id \otimes t_H \otimes \id) \circ (\phi_R^* \otimes \id) \circ t_H  \right)
        \]
        is a standard solution to the conjugate equations for $H \mtimes K$. Using these to define $\Omega_{H \mtimes K}$, we find that indeed $\Omega_{H \mtimes K} = \Omega_H \otimes \Omega_K$.

        \item By the general theory of unitary tensor categories, and more generally unitary $2$-categories, there is a $*$-preserving linear isomorphism $\theta_{H,K}: \mor_{\mathcal{C}}(H,K) \to \mor_{\mathcal{C}}(\overline{K}, \overline{H})$ which satisfies
        \[
        (V \otimes \id) \circ s_H = (\id \otimes \theta_{H,K}) \circ s_K \text{ and } (\id \otimes V) \circ t_H = (\theta_{H,K}(V) \otimes \id) t_K.
        \]
        Using this, the result is immediate.

        \item This follows from direct computation, using the fact that the pair
        \[
        \left( (\gamma_P^{-1/2} \otimes 1)S_P, (1 \otimes \gamma_P^{1/2})T_P \right)
        \]
        is a standard solution to the conjugate equations for $H_P$.
    \end{enumerate}
\end{proof}

\subsection{A $*$-algebra obtained from the unitary 2-category $\mathcal{C}$} \label{sbsc: *-algebra B obtained from category}

In \cite{DeCommerTimmermann2015}, De Commer and Timmermann defined partial compact quantum groups, and showed a version of Tannaka-Krein-Woronowicz duality for these objects. The reconstruction in that setting starts from a unitary $2$-category of finite type Hilbert-$\ell^{\infty}(I)$-bimodules for some set $I$. When we restrict our attention in this paper to the setting where $M$ is abelian, i.e. $M = \bigoplus_{i \in I}M_i \cong \ell^{\infty}(I)$, and $\del$ is simply the functional implemented by the counting measure, the following section recovers part of the construction of \cite[section 3]{DeCommerTimmermann2015}. We will however not show that the $*$-algebra $\mathcal{B}$ obtained here admits a bialgebra structure, as we will move directly to the quantum group setting in the next section \ref{sbsc: the algebraic quantum group}.

\begin{definition} \label{def: algebra B}
    Fix some maximal set $\irr$ of mutually inequivalent irreducible objects $H$ in $\mathcal{C}$. We define a $*$-algebra $\mathcal{B}$, which as a vector space is isomorphic to $\bigoplus_{H \in \irr} \overline{H}_0 \otimes_{\text{alg}} H_0$, where we use the notation $H_0 = M_0 \cdot H \cdot M_0$ introduced in definition \ref{def: Hilbert-M-bimodule + finite type}. Denote $\Gelt{H}{\xi}{\eta}$ for the element $\overline{\xi} \otimes \eta \in \overline{H}_0 \otimes_{\text{alg}} H_0$. Recall the notion of a basis of partial isometries from definition \ref{def: basis of partial isometries}. The multiplication is defined as follows.
    \begin{equation}
        \Gelt{H}{\xi}{\eta} \Gelt{K}{\xi'}{\eta'} = \sum_{\substack{L \in \irr \\ V \in \bpi(H \mtimes K, L)}} \Gelt{L}{V(\xi \mtimes \xi')}{V(\eta \mtimes \eta')} \label{eq: multiplication in B}
    \end{equation}
    For any object $H$ in $\mathcal{C}$, let $s_H,t_H$ denote a standard solution to the conjugate equations. Then the adjoint is given by
    \begin{equation}
        \Gelt{H}{\xi}{\eta}^* = \Gelt{\overline{H}}{(\id \otimes \overline{\xi})t_H(\vect{1})}{(\overline{\eta} \otimes \id)s_H(\vect{1})}. \label{eq: adjoint in B}
    \end{equation}
    Here, as in proposition \ref{prop: frobenius reciprocity for finite type morphisms}, we write $\vect{1}$, where we mean that we are taking a sum $\vect{1} = \sum_{i \in I} \vect{1_i}$, and the resulting sum in \eqref{eq: adjoint in B} has only finitely many nonzero terms.
\end{definition}
\begin{proof}[Proof that definition \ref{def: algebra B} yields a well-defined $*$-algebra]
    Firstly note that all but finitely many of the terms in \eqref{eq: multiplication in B} are zero, since $H \mtimes K$ has finitely many subobjects each with finite multiplicity. Also, the sum in \eqref{eq: multiplication in B} does not depend on the choice of $\bpi$, since any two such bases are unitarily conjugate to one another. The antilinearity-linearity of $G_L$ then ensures this non-dependence.

    To show associativity of the multiplication, it suffices to note that for any $H_1,H_2,H_3,L$ irreducible objects in $\mathcal{C}$, the following are both $\bpi(H_1 \mtimes H_2 \mtimes H_3, L)$
    \[
    \bigcup_{K \in \irr} \{ V(W \mtimes 1) | V \in \bpi(K \mtimes H_3,L), W \in \bpi(H_1 \mtimes H_2,K) \}
    \]
    and
    \[
    \bigcup_{K \in \irr} \{ V(1 \mtimes W) | V \in \bpi(H_1 \mtimes K, L), W \in \bpi(H_2 \mtimes H_3, K) \}.
    \]

    The adjoint is well-defined since for any $\xi \in H_0$, we can find some central projection $p \in M_0$ such that $p \cdot \xi = \xi$. Then for any central projection $q \in M_0$ with $q \leq 1-p$, we get that $(\id \otimes \overline{\xi})t_H(\vect{q}) = 0$ by $M$-bimodularity of $t_H$. Hence, $(\id \otimes \overline{\xi})t_H(\vect{1}) = (\id \otimes \overline{\xi})t_H(\vect{p})$. One finds a similar statement for $\eta$.
    
    The adjoint is involutive by definition, and the observation that $(t_H,s_H)$ form a standard solution to the conjugate equations for $\overline{H}$ when $(s_H,t_H)$ do for $H$. To show antimultiplicativity, we first make the following observations for some fixed $H_1,H_2,K \in \irr$, with respective standard solutions to the conjugate equations $(s_{H_i},t_{H_i})_{i \in \{1,2\}}$ and $(s_K,t_K)$. All three of these observations follow from the general theory of unitary $2$-categories, and the fact that we take the solutions to the conjugate equations to be standard.
    \begin{itemize}
        \item Denote $\phi_L: L^2(M) \mtimes H_2 \to H_2$ the morphism from the proof of proposition \ref{prop: F_0 is monoidal unit}. For every $W \in \mor_{\mathcal{C}}(\overline{H_2} \mtimes \overline{H_1}, \overline{K})$, we have that
        \[
        \norm{t_{K}}^{-1} t_{H_2}^* (\id_{\overline{H_2}} \otimes \phi_L )(\id_{\overline{H_2}} \otimes t_{H_1}^* \otimes \id_{H_2})(W^* \otimes \id_{H_1 \mtimes H_2})
        \]
        is an element of $\mor_{\mathcal{C}}(\overline{K} \mtimes H_1 \mtimes H_2, L^2(M))$, and taking the set of these elements for $W$ ranging over some $\bpi(\overline{H_2} \mtimes \overline{H_1}, \overline{K})$, we obtain a $\bpi(\overline{K} \mtimes H_1 \mtimes H_2, L^2(M))$.

        \item For any $V \in \mor_{\mathcal{C}}(H_1 \mtimes H_2, K)$, we have that
        \[
        \norm{s_{K}}^{-1} s_{K}^*(V \otimes \id_{\overline{K}})
        \]
        is an element of $\mor_{\mathcal{C}}(H_1 \mtimes H_2 \mtimes \overline{K}, L^2(M))$, and taking the set of these elements for $V$ ranging over some $\bpi(H_1 \mtimes H_2, K)$ yields a $\bpi(H_1 \mtimes H_2 \mtimes \overline{K}, L^2(M))$.

        \item Denote again $\phi_L: L^2(M) \mtimes H_2 \to H_2$ the morphism from the proof of proposition \ref{prop: F_0 is monoidal unit}, as well as $\phi_R: H_1 \mtimes L^2(M) \to H_1$. Note that the ranges of these morphisms are not the same, in contrast to what was the case in proposition \ref{prop: F_0 is monoidal unit}. For any $V \in \mor_{\mathcal{C}}(H_1 \mtimes H_2, K)$ and $W \in \mor_{\mathcal{C}}(\overline{H_2} \mtimes \overline{H_1}, \overline{K})$, we have the following equality.
        \begin{align*}
            &\norm{t_{K}}^{-2} \overline{\norm{ t_{H_2}^*(\id_{\overline{H_2}} \otimes \phi_L)(\id_{\overline{H_2}} \otimes t_{H_1}^* \otimes \id_{H_2})(W^* \otimes V^*) t_{K} }} \\
            &= \norm{s_{K}}^{-2} \norm{s_{H_1}^*(\phi_R \otimes \id_{\overline{H_1}}) (\id_{H_1} \otimes s_{H_2}^* \otimes \id_{\overline{H_1}}) (V^* \otimes W^*) s_{K} }
        \end{align*}
    \end{itemize}

    Now take $H_1,H_2 \in \irr$ arbitrarily, as well as $\xi,\eta \in (H_1)_0$ and $\xi',\eta' \in (H_2)_0$. We calculate as follows.
    \begin{align*}
        &\left( \Gelt{H_1}{\xi}{\eta} \Gelt{H_2}{\xi'}{\eta'} \right)^* \\        
        =& \; \left( \sum_{\substack{K \in \irr \\ V \in \bpi(H_1 \mtimes H_2,K)}} \Gelt{K}{V(\xi \mtimes \xi')}{V(\eta \mtimes \eta')} \right)^* \\
        =& \; \sum_{\substack{K \in \irr \\ V \in \bpi(H_1 \mtimes H_2,K)}} \Gelt{\overline{K}}{(\id \otimes \overline{(\xi \mtimes \xi')}V^*) t_{K}(\vect{1})}{(\overline{(\eta \mtimes \eta')}V^* \otimes \id) s_K(\vect{1})} \\        
        =& \; \sum_{\substack{K \in \irr \\ V \in \bpi(H_1 \mtimes H_2,K) \\ W \in \bpi(\overline{H_2} \mtimes \overline{H_1}, \overline{K})}} \norm{t_{K}}^{-2} \overline{\norm{ t_{H_2}^*(\id_{\overline{H_2}} \otimes \phi_L)(\id_{\overline{H_2}} \otimes t_{H_1}^* \otimes \id_{H_2})(W^* \otimes V^*) t_{K} }} \\
        & \Gelt{\overline{K}}{ (\id \otimes \overline{(\xi \mtimes \xi')})(W \otimes \id_{H_1 \otimes H_2}) (\id_{\overline{H_2}} \otimes t_{H_1} \otimes \id_{H_2})(\id_{\overline{H_2}} \otimes \phi_L^*)t_{H_2}(\vect{1}) }{ (\overline{(\eta \mtimes \eta')}V^* \otimes \id) s_K(\vect{1})} \\        
        =& \; \sum_{\substack{K \in \irr \\ V \in \bpi(H_1 \mtimes H_2,K) \\ W \in \bpi(\overline{H_2} \mtimes \overline{H_1}, \overline{K})}} \norm{s_{K}}^{-2} \norm{s_{H_1}^*(\phi_R \otimes \id_{\overline{H_1}}) (\id_{H_1} \otimes s_{H_2}^* \otimes \id_{\overline{H_1}}) (V^* \otimes W^*) s_{K} } \\
        & \Gelt{\overline{K}}{ (\id \otimes \overline{(\xi \mtimes \xi')})(W \otimes \id_{H_1 \otimes H_2}) (\id_{\overline{H_2}} \otimes t_{H_1} \otimes \id_{H_2})(\id_{\overline{H_2}} \otimes \phi_L^*)t_{H_2}(\vect{1}) }{ (\overline{(\eta \mtimes \eta')}V^* \otimes \id) s_K(\vect{1})} \\        
        =& \; \sum_{\substack{K \in \irr \\ W \in \bpi(\overline{H_2} \mtimes \overline{H_1}, \overline{K})}} \Gelt{\overline{K}}{ (\id \otimes \overline{(\xi \mtimes \xi')})(W \otimes \id_{H_1 \otimes H_2}) (\id_{\overline{H_2}} \otimes t_{H_1} \otimes \id_{H_2})(\id_{\overline{H_2}} \otimes \phi_L^*)t_{H_2}(\vect{1}) }{ (\overline{(\eta \mtimes \eta')} \otimes \id) (\id_{H_1 \mtimes H_2} \otimes W) (\id_{H_1} \otimes s_{H_2} \otimes \id_{\overline{H_1}})(\phi_R^* \otimes \id_{\overline{H_1}})s_{H_1}(\vect{1}) } \\        
        =& \; \sum_{\substack{K \in \irr \\ W \in \bpi(\overline{H_2} \mtimes \overline{H_1}, \overline{K})}} \Gelt{\overline{K}}{W \left[ (\id \otimes \overline{\xi'})t_{H_2}(\vect{1}) \mtimes (\id \otimes \overline{xi})t_{H_1}(\vect{1}) \right]}{W \left[ (\overline{\eta'} \otimes \id)s_{H_2}(\vect{1}) \mtimes (\overline{\eta} \otimes \id)s_{H_1}(\vect{1}) \right]} \\
        =& \; \Gelt{H_2}{\xi'}{\eta'}^* \Gelt{H_1}{\xi}{\eta}^*
    \end{align*}
\end{proof}

To alleviate some of the notation, we introduce the following.
\begin{notation} \label{not: G_H for reducibles}
    For any object $H$ in $\mathcal{C}$, and any $\xi,\eta \in H_0$, we denote
    \[
    \Gelt{H}{\xi}{\eta} := \sum_{\substack{K \in \irr \\ V \in \bpi(H,K)}} \Gelt{K}{V(\xi)}{V(\eta)}.
    \]
    When $H$ is a Hilbert-$M$-bimodule which is not necessarily an object of $\mathcal{C}$, but it is covered by $\mathcal{C}$ (definition \ref{def: Category covers bimodule}), we will also write
    \[
    \Gelt{H}{\xi}{\eta} := \sum_{\substack{K \in \irr \\ V \in \bpi(H,K)}} \Gelt{K}{V(\xi)}{V(\eta)}
    \]
    but only when we make sure that $\xi,\eta$ are contained in some algebraic direct sum $\bigoplus_{a \in A} (H_a)_0$, where $(H_a)_{a \in A}$ realises the covering of $H$ by $\mathcal{C}$. In particular, for $\vect{x_0 \otimes \cdots \otimes x_n}, \vect{y_0 \otimes \cdots \otimes y_n} \in \mathcal{F}_n$, we will write
    \begin{equation}
    \Gelt{n}{x_0 \otimes \cdots \otimes x_n}{y_0 \otimes \cdots \otimes y_n} := \Gelt{L^2(M^{n+1})}{\vect{x_0 \otimes \cdots \otimes x_n}}{\vect{y_0 \otimes \cdots \otimes y_n}} = \sum_{\substack{K \in \irr \\ V \in \bpi(L^2(M^{n+1}),K)}} \Gelt{K}{V(\vect{x_0 \otimes \cdots \otimes x_n})}{V(\vect{y_0 \otimes \cdots \otimes y_n})}. \label{eq: defining G_n}
    \end{equation}
\end{notation}

\begin{remark} \label{rem: calculation in B for reducibles}
    With the notation \ref{not: G_H for reducibles}, we get that for all objects $H,K$ in $\mathcal{C}$
    \[
    \Gelt{H}{\xi}{\eta} \Gelt{K}{\xi'}{\eta'} = \Gelt{H \mtimes K}{\xi \mtimes \xi'}{\eta \mtimes \eta'} \text{ for any } \xi,\eta \in H_0, \xi',\eta' \in K_0
    \]
    and
    \[
    \Gelt{H}{\xi}{\eta}^* = \Gelt{\overline{H}}{(\id \otimes \overline{\xi})t_H(\vect{1})}{(\overline{\eta} \otimes \id)s_H(\vect{1})} \text{ for any } \xi,\eta \in H_0.
    \]
    Note that in particular,
    \[
    \Gelt{n}{x_0 \otimes \cdots \otimes x_n}{y_0 \otimes \cdots \otimes y_n} \Gelt{m}{z_0 \otimes \cdots \otimes z_m}{t_0 \otimes \cdots \otimes t_m} = \Gelt{n+m}{x_0 \otimes \cdots \otimes x_nz_0 \otimes \cdots \otimes z_m}{y_0 \otimes \cdots \otimes y_nt_0 \otimes \cdots \otimes t_m}.
    \]
\end{remark}

\begin{lemma} \label{lem: calculation in B}
    The following rules hold in the $*$-algebra $\mathcal{B}$, for any objects $K,H$ in $\mathcal{C}$, any $\xi,\eta \in H_0$, $\xi',\eta' \in K_0$, $V \in \mor_{\mathcal{C}}(H,K)$, and $n \in \mathbb{N}$.
    \begin{enumerate}
        \item for any $x,y \in M_0$, we have $\Gelt{H}{\xi \cdot x}{\eta \cdot y} \Gelt{K}{\xi'}{\eta'} = \Gelt{H}{\xi}{\eta} \Gelt{K}{\mu^{-1/2}(x) \cdot \xi'}{\mu^{-1/2}(y) \cdot \eta'}$

        \item We have an equality $\Gelt{H}{V^*(\xi')}{\eta} = \Gelt{K}{\xi'}{V(\eta)}$, and this continues to hold when $H,K$ are Hilbert-$M$-bimodules which are covered by $\mathcal{C}$, and $V$ is a morphism which is covered by $\mathcal{C}$. (see definitions \ref{def: Category covers bimodule} and \ref{def: covering morphisms}.)

        \item Let $(s_H',t_H')$ be any (not necessarily standard) solution to the conjugate equations for $H$. Then we have
        \[
        \Gelt{H}{\xi}{\eta}^* = \Gelt{\overline{H}}{(\id \otimes \overline{\xi})t_H'(\vect{1})}{(\overline{\eta} \otimes \id)s_H'(\vect{1})}.
        \]
        In other words, we could relaxed the condition that the adjoint is defined using standard solutions to the conjugate equations.

        \item For any $\vect{x_0 \otimes \cdots \otimes x_n}, \vect{y_0 \otimes \cdots \otimes y_n} \in \mathcal{F}_n$, we get that
        \[
        \Gelt{n}{x_0 \otimes \cdots \otimes x_n}{y_0 \otimes \cdots \otimes y_n}^* = \Gelt{n}{\mu(x_n)^* \otimes \cdots \otimes \mu(x_0)^*}{y_n^* \otimes \cdots \otimes y_0^*}.
        \]

        \item For any $x_0, x_1, x_2, y_0,y_1,y_2 \in M_0$, we have
        \begin{align*}
            \Gelt{2}{x_0 \otimes 1 \otimes x_2}{y_0 \otimes y_1 \otimes y_2} =& \; \del(y_1) \Gelt{1}{x_0 \otimes x_2}{y_0 \otimes y_2} \\
            \Gelt{2}{x_0 \otimes x_1 \otimes x_2}{y_0 \otimes 1 \otimes y_2} =& \; \del(x_1^*) \Gelt{1}{x_0 \otimes x_2}{y_0 \otimes y_2}
        \end{align*}
        where we should again interpret this $1$ as taking a sum over $(1_i)_{i \in I}$, and retaining only the finitely many nonzero terms.
    \end{enumerate}
\end{lemma}
\begin{proof}
    \begin{enumerate}
        \item This follows directly from remark \ref{rem: calculation in B for reducibles} and the easy to check fact that $(\xi \cdot x) \mtimes \xi' = \xi \mtimes (\mu^{-1/2}(x) \cdot \xi')$, and similarly for $\eta,\eta'$ and $y$. 

        \item Fix for any $L \in \irr$ some $\bpi(H,L) := \{X^L_s | s \in S_L \}$ and a $\bpi(K,L) := \{Y^L_t | t \in T_L\}$ where $T_L$ and $S_L$ are finite index sets for every $L \in \irr$. Then denote by $V^L_{t,s} \in \mathbb{C}$ the scalar such that $Y^L_t V (X^L_s)^* = V^L_{t,s} \id_L$. We then get that
        \[
        V (X^L_s)^* = \sum_{t \in T_L} V^L_{t,s} \cdot (Y^L_t)^* \text{ and } Y^L_t V = \sum_{s \in S_L} V^L_{t,s} \cdot X^L_s.
        \]
        Using this, we calculate as follows.
        \begin{align*}
            \Gelt{K}{\xi'}{V(\eta)} &= \sum_{\substack{L \in \irr, t \in T_L}} \Gelt{L}{Y^L_t(\xi')}{ Y^L_t V(\eta)} = \sum_{\substack{L \in \irr \\ t \in T_L, s \in S_L}} V^L_{t,s} \Gelt{L}{Y^L_t(\xi')}{X^L_s (\eta)} \\
            &= \sum_{L \in \irr, s \in S_L} \Gelt{L}{X^L_sV^*(\xi')}{X^L_s(\eta)} = \Gelt{H}{V^*(\xi')}{\eta}
        \end{align*}

        In the more general case when $H,K$ are Hilbert-$M$-bimodules which are covered by $\mathcal{C}$, recall from notation \ref{not: G_H for reducibles} that by assumption $\xi',\eta$ are contained in some subobjects of $K$ and $H$ respectively, which are actual objects in $\mathcal{C}$. Call these objects $L_1,L_2$, and denote $P_1 \in \mor_{\mathcal{C}}(H,L_1)$ and $P_2 \in \mor_{\mathcal{C}}(K,L_2)$ the projections whose ranges are $L_1$ and $L_2$ respectively. Then $(P_2 \circ V \circ P_1) \in \mor_{\mathcal{C}}(L_1,L_2)$, and we get that
        \[
        \Gelt{K}{\xi'}{V(\eta)} = \Gelt{L_2}{P_2(\xi')}{(P_2 \circ V \circ P_1)(\eta))} = \Gelt{L_1}{(P_1 \circ V^* \circ P_2)(\xi')}{ P_1(\eta)} = \Gelt{H}{V^*(\xi')}{\eta}.
        \]

        \item Let $(s_H,t_H)$ be a standard solution to the conjugate equations for $H$. Then there is an invertible morphism $T \in \End_{\mathcal{C}}(\overline{H})$ such that 
        \begin{align*}
            s_H' = (1 \otimes T^*)s_H \text{ and } t_H' = (T^{-1} \otimes 1)t_H.
        \end{align*}
        Using the previous point, we then get that
        \begin{align*}
            \Gelt{\overline{H}}{(\id \otimes \overline{\xi})t_H'(\vect{1})}{(\overline{\eta} \otimes \id)s_H'(\vect{1})} &= \Gelt{\overline{H}}{T^{-1}(\id \otimes \overline{\xi})t_H(\vect{1})}{T^*(\overline{\eta} \otimes \id)s_H(\vect{1})} = \Gelt{\overline{H}}{(\id \otimes \overline{\xi})t_H(\vect{1})}{(\overline{\eta} \otimes \id)s_H(\vect{1})} = \Gelt{H}{\xi}{\eta}^*.
        \end{align*}
        
        \item Using again the remark \ref{rem: calculation in B for reducibles}, it suffices by induction to work with $n = 1$. Recall the definition of $\mathcal{J}$ from lemma \ref{lem: central projections in N}. By taking linear combinations, we may then assume that there is some $P \in \mathcal{J}$ such that $\vect{x_0 \otimes x_1}$ and $\vect{y_0 \otimes y_1}$ lie in the range of $P$. Recall from proposition \ref{prop: frobenius reciprocity for finite type morphisms} and the discussion below that $(S_P,T_P)$ solve the conjugate equations for the range $H_P$ of $P$ when identifying $H_P \mtimes \overline{H_P}$ and $\overline{H_P} \mtimes H_P$ with a subspace of $L^2(M^3)$ by the map $1 \otimes \mult \otimes 1$. By the previous point, we can calculate as follows, remembering this identification.
        \begin{align*}
            (1 \otimes 1 \otimes \overline{\vect{x_0 \otimes x_1}})T_P(\vect{1}) &= \sum_{s,t \in \onb(\mathcal{F}_0)}  \vect{s \otimes t} \left\langle P \vect{t^* \otimes s^*} , \vect{x_0 \otimes x_1} \right\rangle \\
            &= \sum_{s,t \in \onb(\mathcal{F}_0)} \vect{s \otimes t} \del(x_0^* t^*) \del(x_1^* s^*) \\
            &= \vect{\left( \sum_{s \in \onb(\mathcal{F}_0)} \del(s^* \mu^{-1}(x_1^*)) s \right) \otimes \left( \sum_{t \in \onb(\mathcal{F}_0)} \del(t^* \mu^{-1}(x_0^*)) t \right)} \\
            &= \vect{\mu(x_1)^* \otimes \mu(x_0)^*}
        \end{align*}
        A similar calculation yields $(\overline{\vect{y_0 \otimes y_1}} \otimes 1 \otimes 1)S_P = \vect{y_1^* \otimes y_0^*}$, proving the result.

        \item We will only show the first claim, as the second is handled analogously. Take some projection $P \in \End_{\mathcal{C}}(L^2(M^3))$ such that $P \vect{y_0 \otimes y_1 \otimes y_2} = \vect{y_0 \otimes y_1 \otimes y_2} $. Then we already get that
        \[
        \Gelt{2}{x_0 \otimes 1 \otimes x_2}{y_0 \otimes y_1 \otimes y_2} = \Gelt{2}{P \vect{x_0 \otimes 1 \otimes x_2)}}{y_0 \otimes y_1 \otimes y_2}.
        \]
        so that we need not worry about the implicit infinite sum contained in this $1$. Now, since $(1 \otimes \del \otimes 1)$ is covered by $\mathcal{C}$ (definition \ref{def: covering morphisms}), we get that $(1 \otimes \del \otimes 1) \circ P$ is a morphism in $\mathcal{C}$. We claim that $[(1 \otimes \del \otimes 1) \circ P]^* \vect{x_0 \otimes x_2} = P \vect{x_0 \otimes 1 \otimes x_2}$, which will end the proof by point 2 of this lemma. Take to this end $s,t,r \in M_0$ with $P(\vect{s \otimes t \otimes r}) = \vect{s \otimes t \otimes r}$ arbitrarily, and calculate.
        \begin{align*}
            \langle [(1 \otimes \del \otimes 1) \circ P]^* \vect{x_0 \otimes x_2} , \vect{s \otimes t \otimes r} \rangle =& \; \langle \vect{x_0 \otimes x_2}, [(1 \otimes \del \otimes 1) \circ P] \vect{s \otimes t \otimes r} \rangle \\
            = \del(s^* x_0) \del(t^*) \del(r^* x_2) =& \; \langle P \vect{x_0 \otimes 1 \otimes x_2}, \vect{s \otimes t \otimes r} \rangle
        \end{align*}
    \end{enumerate}
\end{proof}

\begin{definition} \label{def: varphi on B}
    We define a functional $\varphi: \mathcal{B} \to \mathbb{C}$ by
    \[
    \varphi \left( \Gelt{H}{\xi}{\eta} \right) := \sum_{V \in \bpi(H,L^2(M))} \overline{(\del \circ V)(\xi)} (\del \circ V)(\eta).
    \]
\end{definition}
\begin{proposition} \label{prop: varphi on B is positive faithful, modular data}
    The functional $\varphi$ from definition \ref{def: varphi on B} is well-defined. Recall from definition \ref{def: Omega} the definition of $\Omega_H$ for $H$ an object in $\mathcal{C}$. Recall from definition \ref{def: algebra B} that we fixed $\irr$, a maximal set of pairwise nonisomorphic irreducible objects in $\mathcal{C}$. Then for any $H,K \in \irr$, and $\xi,\eta \in H_0$, $\xi',\eta' \in K_0$ we get that
    \begin{equation}
    \varphi \left( \Gelt{H}{\xi}{\eta}^* \Gelt{K}{\xi'}{\eta'} \right) = \frac{\delta_{H,K}}{\norm{t_H}^2} \langle \Omega_H(\xi), \xi' \rangle \langle \eta' , \eta \rangle \label{eq: GNS inner product for varphi on B}
    \end{equation}
    and
    \begin{equation}        
        \varphi \left( \Gelt{K}{\xi'}{\eta'} \Gelt{H}{\xi}{\eta}^* \right) = \frac{\delta_{H,K}}{ \norm{s_H}^2} \langle \xi , \xi' \rangle \langle \eta' , \Omega_H^{-1}(\eta) \rangle \label{eq: skew GNS inner product for varphi on B}
    \end{equation}    
    In particular, $\varphi$ is positive and faithful. Moreover, the bijective linear map
    \[
    \varsigma : \mathcal{B} \to \mathcal{B}: \Gelt{H}{\xi}{\eta} \mapsto \Gelt{H}{\Omega_H(\xi)}{\Omega_H(\eta)}
    \]
    is well-defined, and satisfies $\varsigma(xy) = \varsigma(x)\varsigma(y)$ as well as $\varphi(xy) = \varphi(\varsigma(y) x)$ for all $x,y \in \mathcal{B}$.
\end{proposition}
\begin{proof}
    Take any object $H$ in $\mathcal{C}$, and $\xi,\eta \in H_0$ arbitrarily. To show well-definedness, we pass through the definition of our notation \ref{not: G_H for reducibles}. We calculate as follows.
    \begin{align*}
        \sum_{\substack{K \in \irr \\ V \in \mor_{\mathcal{C}}(H,K)}} \varphi \left( \Gelt{K}{V(\xi)}{V(\eta)} \right) =& \; \sum_{\substack{K \in \irr \\ V \in \bpi(H,K)}} \sum_{\substack{W \in \bpi(K,L^2(M))}} \overline{(\del \circ W \circ V)(\xi)} (\del \circ W \circ V)(\eta) \\
        =& \; \sum_{Y \in \bpi(H,L^2(M))} \overline{(\del \circ Y)(\xi)} (\del \circ Y)(\eta) = \varphi \left( \Gelt{H}{\xi}{\eta} \right)
    \end{align*}
    Hence, $\varphi$ is indeed a well-defined linear functional. Take now $H,K \in \irr$ arbitrarily, as well as $\xi, \eta \in H_0$ and $\xi',\eta' \in K_0$, and calculate.
    \begin{align*}
        \varphi & \left( \Gelt{H}{\xi}{\eta}^* \Gelt{K}{\xi'}{\eta'} \right) \\
        =& \; \varphi \left( \Gelt{\overline{H} \mtimes K}{(\id \otimes \overline{\xi})t_H(\vect{1}) \mtimes \xi'}{(\overline{\eta} \otimes \id)s_H(\vect{1}) \mtimes \eta'} \right) \\
        =& \; \sum_{V \in \bpi(\overline{H} \mtimes K, L^2(M))} \overline{(\del \circ V)[(\id \otimes \overline{\xi})t_H(\vect{1}) \mtimes \xi']} (\del \circ V)[(\overline{\eta} \otimes \id)s_H(\vect{1}) \mtimes \eta'] \\
        =& \; \frac{\delta_{H,K}}{\norm{t_H}^2} \overline{ (\del \circ t_H^*) \circ (\id \otimes \overline{\xi} \mtimes \xi') \circ t_H(\vect{1})} (\del \circ t_H) \circ ((\overline{\eta} \otimes \id)s_H(\vect{1}) \mtimes \eta') \\
        =& \; \frac{\delta_{H,K}}{\norm{t_H}^2} \langle \Omega_H(\xi) , \xi' \rangle \langle \eta' , \eta \rangle
    \end{align*}
    Here, we have used the fact that $\bpi(\overline{H} \mtimes K, L^2(M))$ is given by $\{t_H\}$ when $H = K$, and is empty otherwise. To take the final step, we have used the conjugate equation \ref{eq: conjugate equations}. Showing that
    \[
    \varphi \left( \Gelt{K}{\xi'}{\eta'} \Gelt{H}{\xi}{\eta}^* \right) = \frac{\delta_{H,K}}{ \norm{s_H}^2} \langle \xi , \xi' \rangle \langle \eta' , \Omega_H^{-1}(\eta) \rangle
    \]
    is completely analogous.

    Using the properties of the operators $\Omega_H$ given in proposition \ref{prop: properties of the operators Omega}, we show well-definedness of $\varsigma$. Fix an object $H$ in $\mathcal{C}$, and $\xi,\eta \in H_0$.
    \begin{align*}
        \sum_{\substack{K \in \irr \\ V \in \bpi(H,K)}} \varsigma \left( \Gelt{K}{V(\xi)}{V(\eta)} \right) =& \; \sum_{\substack{K \in \irr \\ V \in \bpi(H,K)}} \Gelt{K}{(\Omega_K \circ V)(\xi)}{(\Omega_K \circ V)(\eta)} \\
        =& \; \sum_{\substack{K \in \irr \\ V \in \bpi(H,K)}} \Gelt{K}{(V \circ \Omega_H)(\xi)}{(V \circ \Omega_H)(\eta)} \\
        =& \; \varsigma \left( \sum_{\substack{K \in \irr \\ V \in \bpi(H,K)}} \Gelt{K}{V(\xi)}{V(\eta)} \right) \\
        =& \; \varsigma \left( \Gelt{H}{\xi}{\eta} \right)
    \end{align*}
    Finally, we are ready to calculate as follows, using \eqref{eq: GNS inner product for varphi on B} and \eqref{eq: skew GNS inner product for varphi on B}, for any $H,K \in \irr$ and $\xi,\eta \in H_0$, $\xi',\eta' \in K_0$.
    \begin{align*}
        \varphi \left( \varsigma \left( \Gelt{K}{\xi'}{\eta'} \right) \Gelt{H}{\xi}{\eta}^* \right) =& \; \varphi \left( \Gelt{K}{\Omega_K(\xi')}{\Omega_K(\eta')} \Gelt{H}{\xi}{\eta}^* \right)  \\
        =& \; \frac{\delta_{H,K}}{\norm{s_H}^2} \langle \xi, \Omega_H(\xi') \rangle \langle \Omega_H(\eta') , \Omega_H^{-1}(\eta) \rangle \\
        =& \; \frac{\delta_{H,K}}{\norm{t_H}^2} \langle \Omega_H(\xi) , \xi' \rangle \langle \eta' , \eta \rangle \\
        =& \; \varphi \left( \Gelt{H}{\xi}{\eta}^* \Gelt{K}{\xi'}{\eta'} \right)
    \end{align*}
    Since $\varsigma$ is linear, and these elements span $\mathcal{B}$, this shows $\varphi(\varsigma(x)y) = \varphi(yx)$ for all $x,y \in \mathcal{B}$. Now, for any $x,y,z \in \mathcal{B}$, we must have
    \[
    \varphi(\varsigma(xy)z) = \varphi(zxy) = \varphi(\varsigma(y)zx) = \varphi(\varsigma(x)\varsigma(y)z).
    \]
    Hence, since $\varphi$ is faithful, we get that $\varsigma(xy) = \varsigma(x)\varsigma(y)$ for any $x,y \in \mathcal{B}$.
\end{proof}

\begin{remark} \label{rem: omega on G_n}
    Let $\varsigma$ be as in proposition \ref{prop: varphi on B is positive faithful, modular data}, and $\delta$ as in theorem \ref{thm: explicit description of gamma}. Then by proposition \ref{prop: properties of the operators Omega}.5, and an induction argument, we get that
    \[
    \varsigma \left( \Gelt{n}{x_0 \otimes \cdots \otimes x_n}{y_0 \otimes \cdots \otimes y_n} \right) = \Gelt{n}{\mu(x_0) \delta \otimes \delta^{-1} \mu(x_1) \delta \otimes \cdots \otimes \delta^{-1} \mu(x_n)}{\mu(y_0) \delta \otimes \delta^{-1} \mu(y_1) \delta \otimes \cdots \otimes \delta^{-1} \mu(y_n)}
    \]
    for any $n \in \mathbb{N}$ and $x_0,y_0, \ldots, x_n,y_n \in M_0$.
\end{remark}

\subsection{The algebraic quantum group $(\mathcal{A},\Delta)$} \label{sbsc: the algebraic quantum group}

This section will be concerned with defining the algebraic quantum group $(\mathcal{A},\Delta)$, i.e. a multiplier Hopf-$*$-algebra with invariant functionals as in \cite[Theorem 3.7]{VanDaele98}. We still keep the same category $\mathcal{C}$ fixed, and recall that it covers $L^2(M^2)$ (definition \ref{def: Category covers bimodule}), and covers $\mult$ and $(1 \otimes \del \otimes 1)$ (definition \ref{def: covering morphisms}). In order to define the algebra $\mathcal{A}$, we first need the following lemma.
\begin{lemma} \label{lem: covered objects closed under ordinary tensor product}
    Let $H,K$ be two Hilbert-$M$-bimodules as in definition \ref{def: Hilbert-M-bimodule + finite type} which are covered by $\mathcal{C}$ as in definition \ref{def: Category covers bimodule}. Then $H \otimes K$ admits a canonical covering by $\mathcal{C}$. In particular, this works for $H,K$ objects in $\mathcal{C}$
\end{lemma}
\begin{proof}
    Since $\mathcal{C}$ covers $L^2(M^2)$, we can use proposition \ref{prop: covering bimodules closed under relative tensor product} to find that also $H \mtimes L^2(M^2) \mtimes K$ is covered by $\mathcal{C}$. We can push this through the following isomorphisms.
    \begin{equation}
    H \mtimes L^2(M^2) \mtimes K \cong (H \mtimes L^2(M)) \otimes (L^2(M) \mtimes K) \cong H \otimes K \label{eq: ordinary tensor product as relative with L^2(M^2)}
    \end{equation}
\end{proof}

\begin{definition} \label{def: *-algebra A}
    Fix again a maximal set $\irr$ of pairwise nonisomorphic irreducible elements $H$ in $\mathcal{C}$. We define a $*$-algebra $\mathcal{UA}$ which as a vector space is isomorphic to the algebraic direct sum\footnote{The reader may note that this is the same vector space underlying the algebra $\mathcal{B}$, but we will endow it with a different $*$-algebra structure.}
    \[
    \bigoplus_{H \in \irr} \overline{H}_0 \otimes_{\text{alg}} H_0
    \]
    where we denote $\Aelt{H}{\xi}{\eta}$ for the element corresponding to $\overline{\xi} \otimes \eta$ with $\xi,\eta \in H_0$. The multiplication is defined as follows, for any $H,K \in \irr$, $\xi,\eta \in H_0$, $\xi',\eta' \in K_0$, and any $\bpi(H \otimes K,L)$ as in definition \ref{def: basis of partial isometries}.
    \begin{equation}
        \Aelt{H}{\xi}{\eta} \Aelt{K}{\xi'}{\eta'} := \sum_{\substack{L \in \irr \\ V \in \bpi(H \otimes K,L)}} \Aelt{L}{V(\xi \otimes \xi')}{V(\eta \otimes \eta')} \label{eq: multiplication in A}
    \end{equation}
    The adjoint is defined as
    \begin{equation}
        \Aelt{H}{\xi}{\eta}^* = \Aelt{\overline{H}}{(\id \otimes \overline{\xi})t_H(\vect{1})}{(\overline{\eta} \otimes \id)s_H(\vect{1})} \label{eq: adjoint in A}
    \end{equation}
    where we again view $\vect{1}$ as a sum over $(\vect{1_i})_{i \in I}$, and the resulting sum in \eqref{eq: adjoint in A} has finitely many nonzero terms.
\end{definition}
\begin{proof}[Proof that definition \ref{def: *-algebra A} yields a well-defined $*$-algebra.]
Everything may be done completely analogously to the definition of the $*$-algebra $\mathcal{B}$ (definition \ref{def: algebra B}), except that we should make sure the sum in \eqref{eq: multiplication in A} has only finitely many nonzero terms. Fix therefore $H,K \in \irr$, $\xi,\eta \in H_0$ and $\xi',\eta' \in K_0$. Then there exist central projections $p,q,r,s \in M_0$ such that $\xi \cdot p = \xi$, $\eta \cdot q = \eta$, $r \cdot \xi' = \xi'$ and $s \cdot \eta' = \eta$. Then recall the definition of $\mathcal{P}$ from definition \ref{def: algebra N, endomorphisms of L^2(M^2)} and the discussion below. We can find some $P \in \mathcal{P}$ such that $\vect{p \otimes r}$ and $\vect{q \otimes s}$ lie in the range $H_P$ of $P$. Now, viewing everything through the isomorphisms \ref{eq: ordinary tensor product as relative with L^2(M^2)}, it is clear that we can restrict our attention to $V \in \bpi(H \mtimes H_P \mtimes K,L)$, and this set is always finite for every $L \in \irr$ by corollary \ref{cor: finite dimensional intertwiner spaces}, and empty for all but finitely many $L \in \irr$.
\end{proof}

\begin{notation} \label{not: A_H for reducibles}
    As we did for the $*$-algebra $\mathcal{B}$ in notation \ref{not: G_H for reducibles}, we will also denote 
    \[
    \Aelt{H}{\xi}{\eta} := \sum_{\substack{K \in \irr \\ V \in \bpi(H,K)}} \Aelt{H}{V(\xi)}{V(\eta)}
    \]
    for any Hilbert-$M$-bimodule $H$ which is covered by $\mathcal{C}$ as in definition \ref{def: Category covers bimodule}, but only when $\xi,\eta$ are contained in some algebraic direct sum $\bigoplus_{a \in A} (H_a)_0$, where $(H_a)_{a \in A}$ realises the covering of $H$ by $\mathcal{C}$. In particular, we do this for reducible objects $H$ of $\mathcal{C}$.

    We will use this notation to denote for any $x_1,y_1, \ldots, x_n,y_n \in M_0$
    \[
    \Aelt{n}{x_1 \otimes \cdots \otimes x_n}{y_1 \otimes \cdots \otimes y_n} := \Aelt{L^2(M^n)}{\vect{x_1 \otimes \cdots \otimes x_n}}{\vect{y_1 \otimes \cdots \otimes y_n}} = \sum_{\substack{K \in \irr \\ V \in \bpi(L^2(M^{n+1}),K)}} \Aelt{K}{V(\vect{x_1 \otimes \cdots \otimes x_n})}{V(\vect{y_1 \otimes \cdots \otimes y_n})}.
    \]
    One should note that this is in contrast to the notation for the $*$-algebra $\mathcal{B}$, where $G_n$ was reserved for the covered object $L^2(M^{n+1})$.

    Moreover, we will use the same notation $\Aelt{n}{x_1 \otimes \cdots \otimes x_n}{y_1 \otimes \cdots \otimes y_n} \in \mathcal{M}(\mathcal{UA})$ when either $x_1, \ldots ,x_n$ may be elements of the multiplier algebra $\mathcal{M}(M_0)$. This makes sense since for any object $H$ in $\mathcal{C}$, and any $\xi,\eta \in H_0$, we can find finitely supported central projections $p,q \in M_0$ such that $p \cdot \xi = \xi$ and $q \cdot \eta = \eta$. Then take a projection $P \in \mor_{\mathcal{C}}(L^2(M^{n+1}))$ onto a finite type Hilbert-$M$-subbimodule such that $\vect{y_1 \otimes \cdots \otimes y_n \otimes q}$ lies in the range of $P$. Then we have that $P \circ \rho_{L^2(M^{n+1})}(p^{\op})$ is a finite rank operator, and it makes sense to define
    \[
    \Aelt{n}{x_1 \otimes \cdots \otimes x_n}{y_1 \otimes \cdots \otimes y_n} \Aelt{H}{\xi}{\eta} := \Aelt{L^2(M^{n+1}) \mtimes H}{P(\vect{x_1 \otimes \cdots \otimes x_n \otimes p}) \mtimes \xi}{\vect{y_1 \otimes \cdots \otimes y_n \otimes q} \mtimes \eta}.
    \]
    Similarly $\Aelt{n}{x_1 \otimes \cdots \otimes x_n}{y_1 \otimes \cdots \otimes y_n}$ is a right multiplier, and we can do something similar when $x_1, \ldots, x_n \in M_0$ and $y_1, \ldots, y_n \in \mathcal{M}(M_0)$.
\end{notation}

Now it makes sense to define the two-sided $*$-ideal
\[
\mathcal{I} := \left\langle \del(x) \Aelt{H}{\xi}{\eta} - \Aelt{1}{1}{x}\Aelt{H}{\xi}{\eta} , \del(x^*) \Aelt{H}{\xi}{\eta} - \Aelt{1}{x}{1}\Aelt{H}{\xi}{\eta} | H \in \irr , \xi,\eta \in H_0, x \in M_0 \right\rangle
\]
Then we let $\mathcal{A} := \mathcal{UA}/\mathcal{I}$, and by slight abuse of notation, we continue to denote
$\Aelt{H}{\xi}{\eta}$ when we really mean $\Aelt{H}{\xi}{\eta} + \mathcal{I} \in \mathcal{A}$.

\begin{lemma} \label{lem: calculation in A}
    The following are valid rules for calculation in the $*$-algebra $\mathcal{A}$, for any objects $H,K$ in $\mathcal{C}$, and any $\xi,\eta \in H_0$, $\xi',\eta' \in K_0$.
    \begin{enumerate}
        \item We have an equality $\Aelt{H}{V^*(\xi')}{\eta} = \Aelt{K}{\xi'}{V(\eta)}$, which continues to hold when $H,K$ are Hilbert-$M$-bimodules covered by $\mathcal{C}$, and $V$ is a linear map which is covered by $\mathcal{C}$. (see definitions \ref{def: Category covers bimodule} and \ref{def: covering morphisms}.)

        \item Let $(s_H',t_H')$ be any (not necessarily standard) solution to the conjugate equations for $H$. Then we have
        \[
        \Aelt{H}{\xi}{\eta}^* = \Aelt{\overline{H}}{(\id \otimes \overline{\xi})t_H'(\vect{1})}{(\overline{\eta} \otimes \id)s_H'(\vect{1})}.
        \]
        In other words, we could relaxed the condition that the adjoint is defined using standard solutions to the conjugate equations.

        \item For any $\vect{x_1 \otimes \cdots \otimes x_n}, \vect{y_1 \otimes \cdots \otimes y_n} \in \mathcal{F}_{n-1}$, we get that
        \[
        \Aelt{n}{x_1 \otimes \cdots \otimes x_n}{y_1 \otimes \cdots \otimes y_n}^* = \Aelt{n}{\mu(x_n)^* \otimes \cdots \otimes \mu(x_1)^*}{y_n^* \otimes \cdots \otimes y_1^*}.
        \]

        \item For any $x \in M_0$, we have
        \begin{align*}
            \Aelt{1}{1}{x} = \del(x) 1 \in \mathcal{M}(\mathcal{A}) \text{ and } \Aelt{1}{x}{1} =  \del(x^*) 1 \in \mathcal{M}(\mathcal{A}) \text{ strictly.}
        \end{align*}
    \end{enumerate}
\end{lemma}
\begin{proof}
    The proof for 1-3 is identical to the proof of the analogous properties of the $*$-algebra $\mathcal{B}$ from lemma \ref{lem: calculation in B}. Point 4 holds by definition of the ideal $\mathcal{I}$ and the $*$-algebra $\mathcal{A}$ above.
\end{proof}

\begin{definition} \label{def: hopf structure on A}
    We define the following maps on $\mathcal{A}$. The fact that they are well-defined is the content of proposition \ref{prop: hopf structure on A}.
    \begin{itemize}
        \item $\Delta: \mathcal{A} \to \mathcal{M}(\mathcal{A} \otimes \mathcal{A}): \Aelt{H}{\xi}{\eta} \mapsto \sum_{\zeta \in \onb(H_0)} \Aelt{H}{\xi}{\zeta} \otimes \Aelt{H}{\zeta}{\eta}$ for any object $H$ in $\mathcal{C}$ and $\xi,\eta \in H_0$., where $\onb(H_0)$ denotes any orthonormal basis of $H$ all of whose elements are contained in $H_0$.

        \item $S: \mathcal{A} \to \mathcal{A}: \Aelt{H}{\xi}{\eta} \mapsto \Aelt{H}{\eta}{\xi}^*$ for any object $H$ in $\mathcal{C}$ and $\xi,\eta \in H_0$.

        \item $\epsilon: \mathcal{A} \to \mathbb{C}: \Aelt{H}{\xi}{\eta} \mapsto \langle \eta, \xi \rangle$ for any object $H$ in $\mathcal{C}$ and $\xi,\eta \in H_0$.
    \end{itemize}
\end{definition}

\begin{proposition} \label{prop: hopf structure on A}
    The map $\Delta$ from definition \ref{def: hopf structure on A} is well defined and nondegenerate, and the pair $(\mathcal{A},\Delta)$ is a multiplier Hopf-$*$-algebra as in \cite[Definition 2.4]{VanDaele94}. Its antipode is given by $S$, and its co-unit by $\epsilon$.
\end{proposition}
\begin{proof}
    Let $H,K$ be irreducible objects in $\mathcal{A}$, and $\xi,\eta \in H_0$, $\xi',\eta' \in K_0$. Then let $p,q,r \in M_0$ be central projections such that $p \cdot \xi = \xi$, $\xi' \cdot q = \xi'$ and $\eta' \cdot r = \eta'$. Recall the definition of $\mathcal{P}$ from definition \ref{def: algebra N, endomorphisms of L^2(M^2)}, and the discussion below it. Let $P \in \mathcal{P}$ be such that $P \vect{q \otimes p} = \vect{q \otimes p}$. Since the range of $P$ is finite type, we can find a central projection $s \in M_0$ such that $(\lambda(r) \otimes 1) \circ P = (\lambda(r) \otimes \rho(s^{\op})) \circ P$. It follows that 
    \[
    \Aelt{K}{\xi'}{\eta'} \Aelt{H}{\xi}{\zeta} \otimes \Aelt{H}{\zeta}{\eta} = 0 \text{ for any } \zeta \in (1-s) \cdot H.
    \]
    Hence, it follows that $(\mathcal{A} \otimes 1)\Delta \left( \mathcal{A} \right) \subset \mathcal{A} \otimes \mathcal{A}$, and in particular $\Delta$ is well-defined on $\Aelt{H}{\xi}{\eta}$ for irreducible objects $H$, and $\xi,\eta \in H_0$. Similarly, one shows that $(1 \otimes \mathcal{A}) \Delta(\mathcal{A}) \subset \mathcal{A} \otimes \mathcal{A}$.

    Now, to show that $\Delta$ is a well-defined algebra homomorphism, we calculate for some $H$ which is covered by $\mathcal{C}$ and $\xi,\eta$ in some algebraic direct sum $\bigoplus_{a \in A} H_a$ where $(H_a)_{a \in A}$ realises the covering (definition \ref{def: Category covers bimodule}).
    \begin{align*}
        \sum_{\substack{K \in \irr \\ V \in \bpi(H,K)}} \Delta \left( \Aelt{K}{V(\xi)}{V(\eta)} \right) =& \; \sum_{\substack{K \in \irr \\ V \in \bpi(H,K) \\ \zeta \in \onb(K_0)}} \Aelt{K}{V(\xi)}{\zeta} \otimes \Aelt{K}{\zeta}{V(\eta)} \\
        =& \; \sum_{\substack{K \in \irr \\ V \in \bpi(H,K) \\ \zeta \in \onb(K_0)}} \Aelt{H}{\xi}{V^*(\zeta)} \otimes \Aelt{H}{V^*(\zeta)}{\eta} \\
        =& \;  \Delta \left( \Aelt{H}{\xi}{\eta} \right)
    \end{align*}
    This last step holds because $\{V^*(\zeta) | K \in \irr, V \in \bpi(H,K), \zeta \in \onb(K_0)\}$ is an $\onb(H_0)$. Now, take $H,K$ covered by $\mathcal{C}$, and note that $\{\xi \otimes \eta | \xi \in \onb(H_0), \eta \in \onb(K_0)\}$ is an $\onb((H \otimes K)_0)$. This shows multiplicativity of $\Delta$.

    To show that $\Delta$ respects adjoints, we calculate as follows for some object $H$ in $\mathcal{C}$ and $\xi,\eta \in H_0$.
    \begin{align*}
        \Delta \left( \Aelt{H}{\xi}{\eta} \right)^* =& \; \sum_{\zeta \in \onb(H_0)} \Aelt{H}{\xi}{\zeta}^* \otimes \Aelt{H}{\zeta}{\eta}^* \\
        =& \; \sum_{\zeta \in \onb(H_0)} \Aelt{\overline{H}}{(\id \otimes \overline{\xi})t_H(\vect{1})}{(\overline{\zeta} \otimes \id)s_H(\vect{1})} \otimes \Aelt{\overline{H}}{(\id \otimes \overline{\zeta})t_H(\vect{1})}{(\overline{\eta} \otimes \id)s_H(\vect{1})} \\
        =& \; \sum_{\zeta \in \onb(H_0)} \Aelt{\overline{H}}{(\id \otimes \overline{\xi})t_H(\vect{1})}{ \sum_{\theta \in \onb(\overline{H}_0)} \langle (\overline{\zeta} \otimes \id)s_H(\vect{1}) , \theta \rangle \theta} \otimes \Aelt{\overline{H}}{(\id \otimes \overline{\zeta})t_H(\vect{1})}{(\overline{\eta} \otimes \id)s_H(\vect{1})} \\
        =& \; \sum_{\theta \in \onb(\overline{H}_0)} \Aelt{\overline{H}}{ (\id \otimes \overline{\xi}) t_H(\vect{1})}{ \theta} \otimes \Aelt{\overline{H}}{ \sum_{\zeta \in \onb(H_0)} \langle \zeta , (\id \otimes \overline{\theta})s_H(\vect{1}) \rangle (\id \otimes \overline{\zeta})t_H(\vect{1})}{(\overline{\eta} \otimes \id)s_H(\vect{1})} \\
        =& \; \sum_{\theta \in \onb(\overline{H}_0)} \Aelt{\overline{H}}{ (\id \otimes \overline{\xi}) t_H(\vect{1})}{ \theta} \otimes \Aelt{\overline{H}}{ (\id \otimes \del \circ s_H^*) \circ(t_H(\vect{1}) \otimes \theta) }{(\overline{\eta} \otimes \id)s_H(\vect{1})} \\
        =& \; \sum_{\theta \in \onb(\overline{H}_0)} \Aelt{\overline{H}}{ (\id \otimes \overline{\xi}) t_H(\vect{1})}{ \theta} \otimes \Aelt{\overline{H}}{ \theta }{(\overline{\eta} \otimes \id)s_H(\vect{1})} \\
        =& \; \Delta \left( \Aelt{H}{\xi}{\eta}^* \right)
    \end{align*}
    One also easily checks that $\Delta$ maps approximate units to approximate units, and is therefore nondegenerate. Now, since $\Delta$ is coassociative by definition, it is a well-defined comultiplication on $\mathcal{A}$.

    We turn our attention to the map $\epsilon$. It is clearly multiplicative since $\langle \eta \otimes \eta', \xi \otimes \xi' \rangle = \langle \eta, \xi \rangle \langle \eta', \xi' \rangle$. It must also respect the adjoint simply by definition of the conjugate equations \eqref{eq: conjugate equations}, and therefore it is a character. One may completely straightforwardly check that
    \[
    (\id \otimes \epsilon) \circ \Delta = \id = (\epsilon \otimes \id) \circ \Delta.
    \]
    Hence, $\epsilon$ is a co-unit for the pair $(\mathcal{A}, \Delta)$.

    We finish by turning our attention to the map $S$. Firstly, it is antimultiplicative since the adjoint is antimultiplicative. Finally, we show that
    \[
    \mult_{\mathcal{A}} \circ (\id \otimes S) \circ \Delta(\cdot) = \epsilon(\cdot) 1 \in \mathcal{M}(\mathcal{A})
    \]
    where $\mult_{\mathcal{A}}: \mathcal{M}(\mathcal{A} \otimes \mathcal{A}) \to \mathcal{M}(\mathcal{A})$ denotes the multiplication. To this end, take some object $H$ in $\mathcal{C}$, and $\xi,\eta \in H_0$, and calculate as follows.
    \begin{align*}
        \mult_{\mathcal{A}} \circ (\id \otimes S) \circ \Delta \left( \Aelt{H}{\xi}{\eta} \right) =& \; \sum_{\zeta \in \onb(H_0)} \Aelt{H}{\xi}{\zeta} \Aelt{H}{\eta}{\zeta}^* \\
        =& \; \sum_{\zeta \in \onb(H_0)} \Aelt{H \otimes H}{\xi \otimes (\id \otimes \overline{\eta})t_H(\vect{1})}{\zeta \otimes (\overline{\zeta} \otimes \id)s_H(\vect{1})} \\
        =& \; \Aelt{H \otimes H}{\xi \otimes (\id \otimes \overline{\eta})t_H(\vect{1})}{s_H(\vect{1})} \\
        =& \; \Aelt{1}{(s_H^* \otimes \overline{\eta}) \circ (\xi \otimes t_H(\vect{1}))}{1} \\
        =& \; \sum_{x \in \onb(\mathcal{F}_0)} \Aelt{1}{x}{1} \langle \eta, x^* \cdot \xi \rangle \\
        =& \; \langle \eta, \xi \rangle 1 \in \mathcal{M}(\mathcal{A})
    \end{align*}
    Similarly, one shows that
    \[
    \mult_{\mathcal{A}} \circ (S \otimes \id) \circ \Delta = \epsilon
    \]
    which finishes the proof.
\end{proof}

We can explicitly describe the polar decomposition of the antipode.

\begin{proposition} \label{prop: scaling group and unitary antipode}
    Recall from definition \ref{def: Omega} the positive, invertible, possibly unbounded operators $\Omega_H$. There is a one parameter group of $*$-automorphisms $(\tau_t)_{t \in \mathbb{R}}$ of $\mathcal{A}$, which is analytic on all of $\mathcal{A}$, and given by
    \begin{equation}
        \tau_t \left( \Aelt{H}{\xi}{\eta} \right) := \Aelt{H}{\Omega_H^{-it}(\xi)}{ \Omega_H^{-it}(\eta)} \text{ for any } t \in \mathbb{R}. \label{eq: define scaling group}
    \end{equation}
    There is a $*$-anti-automorphism $R$ of $\mathcal{A}$ given by
    \begin{equation}
        R \left( \Aelt{H}{\xi}{\eta} \right) := \Aelt{H}{\Omega_H^{-1/2}(\eta)}{\Omega_H^{1/2}(\xi)}^*. \label{eq: define unitary antipode}
    \end{equation}
    Moreover, these satisfy
    \begin{enumerate}
        \item $\Delta \circ \tau_t = (\tau_t \otimes \tau_t) \circ \Delta$,

        \item $\Delta \circ R = (R \otimes R) \circ \chi \circ \Delta$, where $\chi$ denotes the flip map, 

        \item and $S = R \circ \tau_{-i/2}$.
    \end{enumerate}
\end{proposition}
\begin{proof}
    For any object $H$ in $\mathcal{C}$, and any central projections $p,q \in M_0$, $\Omega_H$ restricts to a positive invertible operator on $p \cdot H \cdot q$. It follows that $H$ admits a basis of eigenvectors of $\Omega_H$. Clearly, when $\xi,\eta$ are elements of this basis, $\tau$ is analytic on $\Aelt{H}{\xi}{\eta}$. Also multiplicativity and respecting the adjoint are immediate on such elements (when using proposition \ref{prop: properties of the operators Omega}), and since they span $\mathcal{A}$, $\tau$ is indeed an analytic one parameter group of automorphisms of $\mathcal{A}$. 

    Completely analogously, and also using that the adjoint is antimultiplicative, we get that $R$ is antimultiplicative. To show that it preserves adjoints, calculate as follows for some object $H$ in $\mathcal{C}$, and $\xi,\eta \in H_0$.
    \begin{align*}
        R \left( \Aelt{H}{\xi}{\eta}^* \right) =& \; R \left( \Aelt{\overline{H}}{(\id \otimes \overline{\xi})t_H(\vect{1})}{(\overline{\eta} \otimes \id)s_H(\vect{1})} \right) \\
        =& \; \Aelt{\overline{H}}{(\overline{\eta} \otimes \Omega_{\overline{H}}^{-1/2})s_H(\vect{1})}{(\Omega_{\overline{H}}^{1/2} \otimes \overline{\xi})t_H(\vect{1})}^* \\
        =& \; \Aelt{\overline{H}}{(\overline{\Omega_H^{1/2}(\eta)} \otimes \id)s_H(\vect{1})}{(\id \otimes \overline{\Omega_H^{-1/2}(\xi)})t_H(\vect{1})}^* \\
        =& \; \Aelt{H}{(\id \otimes (\del \circ s_H^*))(\id \otimes \Omega_H^{1/2}(\eta) \otimes \id)s_H(\vect{1})}{((\del \circ t_H^*) \otimes \id)(\id \otimes \Omega_H^{-1/2}(\xi) \otimes \id)(t_H(\vect{1}))} \\
        =& \; \Aelt{H}{\Omega_H^{-1/2}(\eta)}{\Omega_H^{1/2}(\xi)} \\
        =& \; R \left( \Aelt{H}{\xi}{\eta} \right)^*
    \end{align*}
    Now, we can prove the three final claims.
    \begin{enumerate}
        \item Since $\Omega_H$ is positive, $\Omega_H^{-it}$ is unitary for every $t \in \mathbb{R}$. Hence, $\{\Omega_H^{-it}(\zeta) | \zeta \in \onb(H_0)\}$ is an $\onb(H_0)$. The formula ensues.

        \item We start by calculating from the right hand side for some object $H$ in $\mathcal{C}$, and $\xi,\eta \in H_0$.
        \begin{align*}
            (R \otimes R)\circ \chi \circ \Delta \left( \Aelt{H}{\xi}{\eta} \right) =& \; \sum_{\zeta \in \onb(H_0)} \Aelt{H}{\Omega_H^{-1/2}(\eta)}{\Omega_H^{1/2}(\zeta)}^* \otimes \Aelt{H}{\Omega_H^{-1/2}(\zeta)}{\Omega_H^{1/2}(\xi)}^* \\
            =& \; \sum_{\zeta,\theta \in \onb(H_0)} \Aelt{H}{\Omega_H^{-1/2}(\eta)}{ \langle \Omega_H^{1/2}(\zeta), \theta \rangle \theta}^* \otimes \Aelt{H}{\Omega_H^{-1/2}(\zeta)}{\Omega_H^{1/2}(\xi)}^* \\
            =& \; \sum_{\zeta,\theta \in \onb(H_0)} \Aelt{H}{\Omega_H^{-1/2}(\eta)}{\Omega_H^{1/2}(\zeta)}^* \otimes \Aelt{H}{\langle \Omega_H^{1/2} \theta, \zeta \rangle \Omega_H^{-1/2}(\zeta)}{\Omega_H^{1/2}(\xi)}^* \\
            =& \; \sum_{\zeta,\theta \in \onb(H_0)} \Aelt{H}{\Omega_H^{-1/2}(\eta)}{ \theta}^* \otimes \Aelt{H}{\theta}{\Omega_H^{1/2}(\xi)}^* \\
            =& \; \Delta \circ R \left( \Aelt{H}{\xi}{\eta} \right)
        \end{align*}

        \item Again it suffices to prove the claim on $\Aelt{H}{\xi}{\eta}$ when $\xi,\eta$ are eigenvectors of $\Omega_H$. For such elements, one immediately gets that $\tau_{-i/2} \left( \Aelt{H}{\xi}{\eta} \right) = \Aelt{H}{\Omega_H^{-1/2}(\xi)}{\Omega_H^{1/2}(\eta)}$. Hence, indeed $R \circ \tau_{-i/2} = S$.
    \end{enumerate}
\end{proof}

In order to endow the multiplier Hopf-$*$-algebra with invariant functionals, we show that $\mathcal{A}$ is isomorphic to a corner of $\mathcal{B}$. Then we will be able to push the functional $\varphi$ from definition \ref{def: varphi on B} through this isomorphism and show that the result is indeed invariant. This is the content of the next proposition.

\begin{proposition} \label{prop: A isomorphic to corner of B}
    Fix $e \in M_0$ with $\del(e^*e) = 1$. Then there is a projection $Q_e :=\Gelt{1}{1 \otimes 1}{e \otimes e^*} \in \mathcal{M}(\mathcal{B})$. Denote by $\mathcal{B}_e$ the corner $Q_e \mathcal{B} Q_e$. There is a $*$-isomorphism
    \[
    \Theta_e: \mathcal{A} \to \mathcal{B}_e: \Aelt{H}{\xi}{\eta} \mapsto \Gelt{L^2(M) \otimes H \otimes L^2(M)}{\vect{1} \otimes \xi \otimes \vect{1}}{\vect{e} \otimes \eta \otimes \vect{e^*}} \text{ for any object } H \text{ and } \xi,\eta \in H_0
    \]
    with inverse
    \[
    \Theta_e^{-1}: \mathcal{B}_e \to \mathcal{A}: \Gelt{H}{\xi}{\eta} \mapsto \Aelt{H}{\xi}{e^* \cdot \eta \cdot \mu^{1/2}(e)} \text{ for any object } H \text{ and } \xi,\eta \in H_0 \text{ such that } \Gelt{H}{\xi}{\eta} \in \mathcal{B}_e.
    \]
\end{proposition}
\begin{proof}
    Recall $\mathcal{P}$ from definition \ref{def: algebra N, endomorphisms of L^2(M^2)} and the discussion below. Take $P \in \mathcal{P}$ such that $P (\vect{e \otimes e^*}) = \vect{e \otimes e^*}$, and denote its range by $H_P$. Take any object $H$ in $\mathcal{C}$, and $\xi,\eta \in H_0$. There exists a central projection $p \in M_0$ such that $p \cdot \xi = \xi$. Then since $P$ is of finite type, we must get that $P(\vect{1_i \otimes p}) = 0$ for all but finitely many $i \in I$, making $P(\vect{1 \otimes p})$ a well-defined vector in $H_P$. Then by definition, we get that
    \[
    Q_e \Gelt{H}{\xi}{\eta} = \Gelt{H_P}{P(\vect{1 \otimes p})}{\vect{e \otimes e^*}} \Gelt{H}{\xi}{\eta}
    \]
    is a well-defined element of $\mathcal{B}$. Similarly, $\Gelt{H}{\xi}{\eta} Q_e \in \mathcal{B}$, so $Q_e$ is a well-defined element of the multiplier algebra $\mathcal{M}(\mathcal{B})$. One readily checks that $Q_e = Q_e^2 = Q_e^*$, using the calculation rules from lemma \ref{lem: calculation in B} and the fact that $\del(e^*e) = 1$.

    Then, to show that $\Theta_e$ is indeed an isomorphism, we start by noting that it is well-defined. Indeed, using the same trick as before, $\Gelt{L^2(M) \otimes H \otimes L^2(M)}{\vect{1} \otimes \xi \otimes \vect{1}}{\vect{e} \otimes \eta \otimes \vect{e^*}}$ is always an element of $\mathcal{B}$. Clearly, multiplying on the left or right by $Q_e$, nothing happens, and hence $\Theta_e$ is a well-defined map on $\mathcal{UA}$. One sees directly that $\Theta_e(\mathcal{I}) = \{0\}$, where $\mathcal{I}$ is the $*$-ideal used to define $\mathcal{A}$. It follows that $\Theta_e$ is well-defined.
    
    To show multiplicativity, take $H,K$ objects in $\mathcal{C}$ and $\xi,\eta \in H_0$, $\xi',\eta' \in K_0$. Then take central projections $p,q,r,s \in M_0$ such that $\xi \cdot p = \xi$, $\eta \cdot q = \eta$, $r \cdot \xi' = \xi'$, and $s \cdot \eta' = \eta'$. Then we get the following.
    \begin{align*}
        \Theta_e \left( \Aelt{H}{\xi}{\eta} \right)\Theta_e \left( \Aelt{K}{\xi'}{\eta'} \right) =& \; \Gelt{L^2(M) \otimes H}{\vect{1} \otimes \xi}{\vect{e} \otimes \eta} \Gelt{2}{p \otimes 1 \otimes r}{q \otimes e^*e \otimes s} \Gelt{K \otimes L^2(M)}{\xi' \otimes \vect{1}}{\eta' \otimes \vect{e^*}} \\
        =& \; \Gelt{L^2(M) \otimes H}{\vect{1} \otimes \xi}{\vect{e} \otimes \eta} \Gelt{1}{p \otimes r}{q \otimes s} \Gelt{K \otimes L^2(M)}{\xi' \otimes \vect{1}}{\eta' \otimes \vect{e^*}} \\
        =& \Gelt{L^2(M) \otimes H \otimes K \otimes L^2(M)}{\vect{1} \otimes \xi \otimes \xi' \otimes \vect{1}}{\vect{e} \otimes \eta \otimes \eta' \otimes \vect{e^*}} \\
        =& \; \Theta_e \left( \Aelt{H}{\xi}{\eta} \Aelt{K}{\xi'}{\eta'} \right)
    \end{align*}

    To show that $\Theta_e$ preserves the adjoint, take again an object $H$ in $\mathcal{C}$, and $\xi,\eta \in H_0$. Let $p,q,r,s \in M_0$ be central projections such that $p \cdot \xi \cdot q = \xi$ and $r \cdot \eta \cdot s = \eta$. We calculate as follows.
    \begin{align*}
        \Theta_e \left( \Aelt{H}{\xi}{\eta} \right)^* =& \; \left( \Gelt{1}{1 \otimes p}{e \otimes r} \Gelt{H}{\xi}{\eta} \Gelt{1}{q \otimes 1}{s \otimes e^*} \right)^* \\
        =& \; \Gelt{1}{1 \otimes q}{e \otimes s} \Gelt{\overline{H}}{(\id \otimes \overline{\xi})t_H(\vect{1})}{(\overline{\eta} \otimes \id)s_H(\vect{1})} \Gelt{1}{p \otimes 1}{r \otimes e^*} \\
        =& \; \Theta_e \left( \Aelt{H}{\xi}{\eta}^* \right)
    \end{align*}
    
    Finally, it is clear that $\Theta_e^{-1}$ is well-defined. For any object $H$ in $\mathcal{C}$ and $\xi,\eta \in H_0$, we get that
    \[
    \Theta_e^{-1} \circ \Theta_e \left( \Aelt{H}{\xi}{\eta} \right) = \Aelt{1}{1}{e^*e} \Aelt{H}{\xi}{\eta} \Aelt{1}{1}{e^*e} = \Aelt{H}{\xi}{\eta}.
    \]
    On the other hand, we get that
    \[
    \Theta_e \circ \Theta_e^{-1} \left( \Gelt{H}{\xi}{\eta} \right) = Q_e \Gelt{H}{\xi}{\eta} Q_e.
    \]
    This ends the proof.
\end{proof}

\begin{definition} \label{def: invariant functionals on quantum group}
    Consider the functional $\varphi$ from definition \ref{def: varphi on B}, and the isomorphisms $\Theta_e$ from proposition \ref{prop: A isomorphic to corner of B}. We define the functional $\varphi_e$ on $\mathcal{A}$ by $\varphi_e := \varphi \circ \Theta_e$. We also define $\psi_e := \varphi_e \circ R$, where $R$ denotes the unitary antipode from proposition \ref{prop: scaling group and unitary antipode}.
\end{definition}

\begin{theorem} \label{thm: invariant functionals on quantum group}
    The functionals $\varphi_e$ and $\psi_e$ form definition \ref{def: invariant functionals on quantum group} are left- resp. right invariant for the multiplier Hopf-$*$-algebra $(\mathcal{A},\Delta)$. Hence, it is an algebraic quantum group in the sense of \cite{VanDaele98} and \cite[Definition 1.2]{KustermansVanDaele97}.
\end{theorem}
\begin{proof}
    By proposition \ref{prop: scaling group and unitary antipode}.2, it is sufficient to prove that $\varphi_e$ is left invariant, since right invariance of $\psi_e$ will then follow automatically. To this end, we fix some object $H$ in $\mathcal{C}$ and $\xi,\eta \in H_0$, and algebraically make the following calculation. Below, we will comment on why some of the steps are justified.
    \begin{align}
        (\id \otimes& \varphi_e) \circ \Delta \left( \Aelt{H}{\xi}{\eta} \right) \nonumber \\
        =& \; \sum_{\zeta \in \onb(H_0)} \Aelt{H}{\xi}{\zeta} \varphi_e \left( \Aelt{H}{\zeta}{\eta} \right) \nonumber \\
        =& \; \sum_{\substack{\zeta \in \onb(H_0) \\ V \in \bpi(L^2(M) \otimes H \otimes L^2(M), L^2(M))}} \Aelt{H}{\xi}{\zeta} \overline{(\del \circ V)(\vect{1} \otimes \zeta \otimes \vect{1})} (\del \circ V)(\vect{e} \otimes \eta \otimes \vect{e^*}) \label{eq: invariant functional step 1}\\
        =& \; \sum_{ V \in \bpi(L^2(M) \otimes H \otimes L^2(M), L^2(M))} \Aelt{H}{\xi}{\sum_{\zeta \in \onb(H_0)} \langle \vect{1} , V (\vect{1} \otimes \zeta \otimes \vect{1}) \rangle \zeta } (\del \circ V)(\vect{e} \otimes \eta \otimes \vect{e^*}) \nonumber \\
        =& \; \sum_{ V \in \bpi(L^2(M) \otimes H \otimes L^2(M), L^2(M))} \Aelt{H}{\xi}{\sum_{\zeta \in \onb(H_0)} (\del \otimes 1 \otimes \del) V^*(\vect{1}) } (\del \circ V)(\vect{e} \otimes \eta \otimes \vect{e^*}) \nonumber \\
        =& \; \sum_{ V \in \bpi(L^2(M) \otimes H \otimes L^2(M), L^2(M))} \Aelt{1}{V(\vect{1} \otimes \xi \otimes \vect{1})}{\vect{1} } (\del \circ V)(\vect{e} \otimes \eta \otimes \vect{e^*}) \label{eq: invariant functional step 2}\\
        =& \; \sum_{ V \in \bpi(L^2(M) \otimes H \otimes L^2(M), L^2(M))} \overline{(\del \circ V)(\vect{1} \otimes \xi \otimes \vect{1})} (\del \circ V)(\vect{e} \otimes \eta \otimes \vect{e^*}) \nonumber \\
        =& \; \varphi_e \left( \Aelt{H}{\xi}{\eta} \right) \nonumber
    \end{align}
    Now some comments on why this is all allowed. Note that $\eta \in H_0$ is left- and right supported on some central projection in $M_0$. This makes it so that $V(\vect{e} \otimes \eta \otimes \vect{e^*}) = 0$ for all but finitely many $V \in \bpi(L^2(M) \otimes H \otimes L^2(M), L^2(M))$. Moreover, recall the definition of $\mathcal{P}$ from definition \ref{def: algebra N, endomorphisms of L^2(M^2)} and the discussion below. Identifying $L^2(M) \otimes H \otimes L^2(M)$ with $L^2(M^2) \mtimes H \mtimes L^2(M^2)$, we can find $P_V,Q_V \in \mathcal{P}$ for each of these $V$, such that $V$ factors through the range of $P_V \mtimes \id_H \mtimes Q_V$. Now keep $\zeta$ fixed. Since the projections $P_V,Q_V$ are of finite type, and $\zeta$ is left and right supported on some finite central projection in $M_0$, we get that $V(\vect{1_i} \otimes \zeta \otimes \vect{1_j}) = 0$ for all but finitely many $i,j \in I$. Hence, by bimodularity of $V$, also $\langle \vect{1_k}, V(\vect{1} \otimes \zeta \otimes \vect{1}) \rangle = 0$ for all but finitely many $k \in I$. Now note that
    \[
    \langle \vect{1_k} , V(\vect{1_i} \otimes \zeta \otimes \vect{1_j}) \rangle = \langle (\del(1_i \cdot) \otimes \id \otimes \del(1_j \cdot)) \circ V^*(\vect{1_k}) , \zeta \rangle \text{ for any }i,j,k \in I.
    \]
    Hence, summing all of this over $i,j,k \in I$, we get that $\zeta \mapsto V(\vect{1} \otimes \zeta \otimes \vect{1})$ is a well-defined (unbounded) linear operator from $H_0$ to $\mathcal{F}_0$. Moreover, this operator is covered by $\mathcal{C}$ as in definition \ref{def: covering morphisms} and hence by lemma \ref{lem: calculation in A}, we can make the steps from \eqref{eq: invariant functional step 1} to \eqref{eq: invariant functional step 2}.
\end{proof}

We are now ready to describe the modular theory of the quantum group $(\mathcal{A},\Delta)$.

\begin{proposition} \label{prop: ratio of invariant functionals}
    Let $\delta$ be as in theorem \ref{thm: explicit description of gamma}. For any $e,f \in M_0$ with $\del(e^*e) = 1 = \del(f^*f)$, we have
    \[
    \del(e \delta^{-2} e^*) \varphi_f = \del(f \delta^{-2} f^*) \varphi_e.
    \]
\end{proposition}
\begin{proof}
    Take any object $H$ in $\mathcal{C}$, and take $\xi,\eta \in H_0$ with central projections $p,q,r,s \in M_0$ such that $p \cdot \xi \cdot q = \xi$ and $r \cdot \eta \cdot s = \eta$. Then we can calculate as follows, using the map $\varsigma$ from proposition \ref{prop: varphi on B is positive faithful, modular data}.
    \begin{align*}
        \varphi_e \left( \Aelt{H}{\xi}{\eta} \right) =& \; \varphi \left( \Theta_e \left( \Aelt{H}{\xi}{\eta} \right) \right) = \varphi \left( \Gelt{1}{1 \otimes p}{e \otimes r} \Gelt{H}{\xi}{\eta} \Gelt{1}{q \otimes 1}{s \otimes e^*} \right) = \varphi \left( \varsigma \left( \Gelt{1}{q \otimes 1}{s \otimes e^*} \right) \Gelt{1}{1 \otimes p}{e \otimes r} \Gelt{H}{\xi}{\eta} \right) \\
        =& \; \varphi \left( \Gelt{2}{q \delta \otimes \delta^{-1} \otimes p}{s \delta \otimes \delta^{-1} \mu(e^*) e \otimes r} \Gelt{H}{\xi}{\eta} \right) = \varphi \left( \Gelt{2}{q \otimes 1 \otimes p}{s \delta^2 \otimes \delta^{-2} \mu(e^*)e \otimes r} \Gelt{H}{\xi}{\eta} \right) \\
        =& \; \del(\delta^{-2} \mu(e^*)e) \varphi \left( \Gelt{1}{q \otimes p}{s \delta^2 \otimes r} \Gelt{H}{\xi}{\eta} \right) = \del(e \delta^{-2}e^*) \varphi \left( \Gelt{1}{q \otimes p}{s \delta^2 \otimes r} \Gelt{H}{\xi}{\eta} \right) \\
    \end{align*}
    Similarly, we find
    \begin{equation}
    \varphi_f \left( \Aelt{H}{\xi}{\eta} \right) = \del(f \delta^{-2} f^*) \varphi \left( \Gelt{1}{q \otimes p}{s \delta^{-2} \otimes r} \Gelt{H}{\xi}{\eta} \right) \label{eq: explicit formula varphi_f}
    \end{equation}
    which proves the proposition.
\end{proof}

\begin{lemma} \label{lem: varphi on A_1}
    Recall from definition \ref{def: concrete unitary 2-category of M-bimodules} the set $\mathcal{E}$ of $0$-cells of the category $\mathcal{C}$, and the central projections $\{z_a | a \in \mathcal{E}\} \subset M$. For any $x,y \in M_0$, we have
    \[
    \varphi_e \left( \Aelt{1}{x}{y} \right) = \del(e \delta^{-2} e^*) \sum_{a \in \mathcal{E}} \del(z_a x^*) \del(z_a \delta^2 y).
    \]
\end{lemma}
\begin{proof}
    Recall the definition of $\mathcal{J} \subset \mathcal{P} \subset \mathcal{N}$ from definition \ref{def: algebra N, endomorphisms of L^2(M^2)} and the discussion below it. Consider the projection $\mult^* \circ \mult \in \mathcal{N}$, and note that by centrality, for any $P \in \mathcal{J}$, we get 
    $\mult \circ P \neq 0$ if and only if $P \leq \mult^* \circ \mult$. In this case, $\mult \circ P \circ \mult^*$ is a minimal projection in $\End_{\mathcal{C}}(L^2(M))$, and hence, $\mult \circ P \circ \mult^* = z_a$ for some $a \in \mathcal{E}$, the $0$-cells of $\mathcal{C}$. Now suppose there is some (nonzero) minimal projection $Q \in \mathcal{P}$ with $Q \leq P$. Then also $Q \leq \mult^* \circ \mult$, and hence $\mult \circ Q \circ \mult^* \leq z_a$. Now, as $z_a$ is minimal and $\mult \circ Q \circ \mult^*$ is positive and nonzero, we must get $\mult \circ Q \circ \mult^* = z_a$. Hence, $P = \mult^* \circ z_a \circ \mult = Q$. All of this goes to show that $\{\lambda(z_a) \circ \mult | a \in \mathcal{E}\}$ is a $\bpi(L^2(M^2),L^2(M))$. Now we are ready to calculate, using the formula \eqref{eq: explicit formula varphi_f}.
    \begin{align*}
        \varphi_e \left( \Aelt{1}{x}{y} \right) =& \; \del(e \delta^{-2} e^*) \varphi \left( \Gelt{1}{1 \otimes x}{\delta^2 \otimes y} \right) = \del(e \delta^{-2} e^*) \sum_{V \in \bpi(L^2(M^2), L^2(M))} \overline{(\del \circ V) (\vect{1 \otimes x})} (\del \circ V)(\vect{\delta^2 \otimes y}) \\
        =& \; \del(e \delta^{-2} e^*) \sum_{a \in \mathcal{E}} \overline{\del(z_a x)} \del(z_a \delta^2 y) = \del(e \delta^{-2} e^*) \sum_{a \in \mathcal{E}} \del(z_a x^*) \del(z_a \delta^2 y)
    \end{align*}
\end{proof}

\begin{proposition} \label{prop: modular element of A}
    Take $a \in \mathcal{E}$ arbitrarily, and some positive nonzero $y \in M_0$ with $z_a y = y$. The element
    \begin{equation}
        \nu := \del(\delta^2 y)^{-1} \Aelt{1}{\delta^2}{y} \in \mathcal{M}(\mathcal{A}) 
    \end{equation}
    does not depend on the choice of $a \in \mathcal{E}$ or $y \in M_0 \cap M_a$. Moreover, $\nu$ is the modular element of the algebraic quantum group $(\mathcal{A},\Delta)$, as defined in \cite[Proposition 3.8]{VanDaele98}.
\end{proposition}
\begin{proof}
    By \cite[Proposition 3.8]{VanDaele98}, there is a multiplier $D \in \mathcal{M}(\mathcal{A})$ such that
    \[
    (\varphi_e \otimes \id) \Delta(x) = \varphi_e(x) D \text{ for any } x \in \mathcal{A},
    \]
    and this is the modular element of $(\mathcal{A},\Delta)$.
    Now take some $a \in \mathcal{E}$ and positive nonzero $x,y \in M_0 \cap M_a$ arbitrarily. We calculate, using lemma \ref{lem: varphi on A_1}.
    \begin{align*}
        D &= \varphi_e \left( \Aelt{1}{x}{y} \right)^{-1}  (\varphi_e \otimes \id)\Delta \left( \Aelt{1}{x}{y} \right) = \del(e \delta^{-2} e^*)^{-1} \del(x)^{-1} \del(\delta^2 y)^{-1} \sum_{t \in \onb(\mathcal{F}_0)} \varphi_e \left( \Aelt{1}{x}{t} \right) \Aelt{1}{t}{y} \\
        &= \del(\delta^2 y)^{-1} \sum_{t \in \onb(\mathcal{F}_0)} \del(z_a \delta^2 t) \Aelt{1}{t}{y} = \del(\delta^2 y)^{-1} \Aelt{1}{z_a \delta^2}{y} = \del(\delta^2 y)^{-1} \Aelt{1}{\delta^2}{y} = \nu
    \end{align*}
\end{proof}

\begin{proposition} \label{prop: modular automorphism on A}
    There is an automorphism\footnote{though not necessarily a $*$-automorphism.} $\kappa: \mathcal{A} \to \mathcal{A}$ given by 
    \begin{equation}
        \kappa \left( \Aelt{H}{\xi}{\eta} \right) = \Aelt{H}{\Omega_H(\xi)}{\delta^{-2} \cdot \Omega_K(\eta) \cdot \delta^2} \label{eq: defining automorphism kappa}
    \end{equation}
    which satisfies
    \[
    \varphi_e(xy) = \varphi_e(\kappa(y) x) \text{ for all } x,y \in \mathcal{A}.
    \]
\end{proposition}
\begin{proof}
    Let $H,K$ be objects in $\mathcal{C}$, and $\xi,\eta \in H_0$, $\xi',\eta' \in K_0$ with central projections $p,q,r,s,p',q',r',s' \in M_0$ such that $p \cdot \xi \cdot q = \xi$, $r \cdot \eta \cdot s = \eta$, $p' \cdot \xi' \cdot q' = \xi'$ and $r' \cdot \eta' \cdot s' = \eta'$. We calculate as follows, using again the automorphism $\varsigma$ from proposition \ref{prop: varphi on B is positive faithful, modular data}, as well as the formula \eqref{eq: explicit formula varphi_f}.
    \begin{align*}
        \varphi_e \left( \Aelt{H}{\xi}{\eta} \Aelt{K}{\xi'}{\eta'} \right) =& \; \varphi \left( \Theta_e \left( \Aelt{H}{\xi}{\eta} \right) \Theta_e \left( \Aelt{K}{\xi'}{\eta'} \right) \right) \\
        =& \; \varphi \left( \Gelt{1}{1 \otimes p}{e \otimes r} \Gelt{H}{\xi}{\eta} \Gelt{1}{q \otimes p'}{s \otimes r'} \Gelt{K}{\xi'}{\eta'} \Gelt{1}{q' \otimes 1}{s' \otimes e^*} \right) \\
        =& \; \varphi \left( \Gelt{1}{q \delta \otimes \delta^{-1}p'}{s \delta \otimes \delta^{-1}r'} \Gelt{K}{\Omega_K(\xi')}{\Omega_K(\eta')} \Gelt{2}{q' \delta \otimes \delta^{-1} \otimes p}{s' \delta \otimes \delta^{-1} \mu(e^*)e \otimes r} \Gelt{H}{\xi}{\eta} \right) \\
        =& \; \del(e \delta^{-2} e^*) \varphi \left( \Gelt{1}{q \otimes p'}{s \delta^2 \otimes r'} \Gelt{K}{\Omega_K(\xi')}{\delta^{-2} \cdot \Omega_K(\eta') \cdot \delta^2} \Gelt{1}{q' \otimes p}{s' \otimes r} \Gelt{H}{\xi}{\eta} \right) \\
        =& \; \varphi_e \left( \Aelt{K}{\Omega_K(\xi')}{\delta^{-2} \cdot \Omega_K(\eta') \cdot \delta^2} \Aelt{H}{\xi}{\eta} \right) \\
        =& \; \varphi_e \left( \kappa \left( \Aelt{K}{\xi'}{\eta'} \right) \Aelt{H}{\xi}{\eta} \right) 
    \end{align*}
    Using faithfulness of $\varphi_e$, one reasons similarly to the proof of proposition \ref{prop: varphi on B is positive faithful, modular data} to find that $\kappa$ is an automorphism of $\mathcal{A}$. 
\end{proof}

To finish this section, we show that the algebraic quantum group $(\mathcal{A},\Delta)$ admits a canonical $\del$-preserving right coaction on $M$.

\begin{proposition} \label{prop: action of quantum group on M}
    There is a right coaction $\alpha: M \curvearrowleft (\mathcal{A},\Delta)$ given by
    \begin{equation}
        \alpha: M_0 \to \mathcal{M}(M_0 \otimes \mathcal{A}): x \mapsto \sum_{y \in \onb(\mathcal{F}_0)} y \otimes \Aelt{1}{y}{x}.
    \end{equation}
    Moreover, this coaction preserves $\del$, i.e.
    \[
    (\del \otimes \id) \alpha(x) = \del(x) 1 \text{ for any } x \in M_0.
    \]
\end{proposition}
\begin{proof}
    We start by showing that $\alpha$ is multiplicative. Take $x,y \in M_0$ arbitrarily.
    \begin{align*}
        \alpha(x) \alpha(y) =& \; \sum_{s,t \in \onb(\mathcal{F}_0)} st \otimes \Aelt{1}{s}{x} \Aelt{1}{t}{y} = \sum_{s,t,r \in \onb(\mathcal{F}_0)} \del(r^*st) r \otimes \Aelt{2}{s \otimes t}{x \otimes y} \\
        =& \; \sum_{s,t,r \in \onb(\mathcal{F}_0)} r \otimes \Aelt{2}{s \otimes \del(t^*s^*r)t}{x \otimes y} = \sum_{s,r \in \onb(\mathcal{F}_0)} r \otimes \Aelt{2}{s \otimes s^*r}{x \otimes y} \\
        =& \; \sum_{r \in \onb(\mathcal{F}_0)} r \otimes \Aelt{2}{\mult^*(r)}{x \otimes y} = \sum_{r \in \onb(\mathcal{F}_0)} r \otimes \Aelt{1}{r}{xy} = \alpha(xy)
    \end{align*}
    Next, we show it preserves the $*$-structure. Take $x \in M_0$ arbitrarily.
    \begin{align*}
        \alpha(x)^* &= \sum_{y \in \onb(\mathcal{F}_0)} y^* \otimes \Aelt{1}{y}{x}^* = \sum_{y,z \in \onb(\mathcal{F}_0)} \del(z^*y^*)z \otimes \Aelt{1}{\mu(y)^*}{x^*} \\
        &= \sum_{z \in \onb(\mathcal{F}_0)} z \otimes \Aelt{1}{ \sum_{y \in \onb(\mathcal{F}_0)} \mu(\del(y^* \mu^{-1}(z^*)) y)^*}{x^*} = \sum_{z \in \onb(\mathcal{F}_0)} z \otimes \Aelt{1}{z}{x^*} = \alpha(x^*)
    \end{align*}
    It is straightforward to see that $(\alpha \otimes \id) \circ \alpha = (\id \otimes \Delta) \circ \alpha$. Moreover, one sees that $\alpha$ maps the approximate unit $\sum_{i \in I} 1_i$ to an approximate unit, and is therefore nondegenerate. Finally, the following calculation shows that $\alpha$ preserves $\del$. Take again $x \in M_0$ arbitrarily.
    \begin{align*}
        (\del \otimes \id) \alpha(x) = \sum_{y \in \onb(\mathcal{F}_0)} \del(y) \Aelt{1}{y}{x} = \Aelt{1}{ \sum_{y \in \onb(\mathcal{F}_0)} \del(y^*)y}{x} = \Aelt{1}{1}{x} = \del(x) 1
    \end{align*}
\end{proof}

Henceforth, we will denote the algebraic quantum group $\mathbb{G} := (\mathcal{A},\Delta)$.

\subsection{Equivalence of the unitary 2-categories $\mathcal{C}$ and $\mathcal{C}( M \curvearrowleft \mathbb{G} )$} \label{sbsc: equivalence of categories}

This section will conclude the proof of theorem \ref{thm: unique quantum group action for equivariant category} by showing that the categories $\mathcal{C}$ and $\mathcal{C}(M \curvearrowleft \mathbb{G})$ are equivalent to each other. We still keep the category $\mathcal{C}$ fixed, and will use the notation $\mathbb{G} = (\mathcal{A},\Delta)$ as before.

\begin{proposition} \label{prop: functor from C to equivariant reps }
    Recall definition \ref{def: equivariant bimodule corepresentation}. For every object $H$ in $\mathcal{C}$, there is an equivariant corepresentation $U_H \in \mathcal{M}(\mathcal{K}(H) \otimes \mathcal{A})$ given by
    \[
    U_H = \sum_{\xi, \eta \in \onb(H_0)} E_{\xi,\eta} \otimes \Aelt{H}{\xi}{\eta},
    \]
    where $E_{\xi,\eta}$ denotes the rank one operator $\zeta \mapsto \langle \zeta,\eta \rangle \xi$. Moreover, for any $T \in \mor_{\mathcal{C}}(H,K)$, we get that
    \[
    (T \otimes 1)U_H = U_K(T \otimes 1).
    \]
    In other words, there is a faithful functor $\mathcal{C} \to \mathcal{C}(M \curvearrowleft \mathbb{G})$.
\end{proposition}
\begin{proof}
    It is immediate that $U_H$ is a corepresentation of $\mathbb{G}$ by the definition of $\Delta$, and it is unitary by definition of $S$, see definition \ref{def: hopf structure on A} and proposition \ref{prop: hopf structure on A}. It remains to show that $U_H$ is equivariant. To this end, take $x \in M$ arbitrarily, and calculate as follows.
    \begin{align*}
        U_H^* (\rho_H(x^{\op}) \otimes 1) U_H =& \; \sum_{\xi,\eta,\zeta,\theta \in \onb(H_0)} E_{\xi,\eta} \rho_H(x^{\op}) E_{\zeta,\theta} \otimes \Aelt{H}{\eta}{\xi}^* \Aelt{H}{\zeta}{\theta} \\
        =& \; \sum_{\xi, \eta, \zeta, \theta \in \onb(H_0)} \langle \zeta \cdot x, \eta \rangle E_{\xi,\theta} \otimes \Aelt{H}{\eta}{\xi}^* \Aelt{H}{\zeta}{\theta} \\
        =& \; \sum_{\xi, \eta, \theta \in \onb(H_0)} E_{\xi,\theta} \otimes \Aelt{H}{\eta}{\xi}^* \Aelt{H}{\eta \cdot x^*}{\theta} \\
        =& \; \sum_{\xi,\eta,\theta \in \onb(H_0)} E_{\xi,\theta} \otimes \Aelt{\overline{H} \otimes H}{(\id \otimes \overline{\eta})t_H(\vect{1}) \otimes \eta \cdot x^*}{(\overline{\xi} \otimes \id)s_H(\vect{1}) \otimes \theta} \\
        =& \; \sum_{\xi,\theta \in \onb(H_0)} E_{\xi,\theta} \otimes \Aelt{\overline{H} \otimes H}{t_H(\vect{\mu^{1/2}(x)^*})}{(\overline{\xi} \otimes \id)s_H(\vect{1}) \otimes \theta} \\
        =& \sum_{\xi,\theta \in \onb(H_0)} E_{\xi,\theta} \otimes \Aelt{1}{\mu^{1/2}(x)^*}{(\overline{\xi} \otimes t_H^*)(s_H(\vect{1}) \otimes \theta)} \\
        =& \; \sum_{\substack{\xi,\theta \in \onb(H_0) \\ y \in \onb(\mathcal{F}_0)}} E_{\xi,\theta} \otimes \Aelt{1}{\mu^{1/2}(x)^*}{\langle \theta \cdot y, \xi \rangle \mu^{1/2}(y)^*} \\
        =& \; \sum_{y \in \onb(\mathcal{F}_0)} \rho_H(y^{\op}) \otimes R \left( \Aelt{1}{y}{x} \right) \\
        =& \; (\rho \otimes \id) \alpha^{\op}(x^{\op})
    \end{align*}
    Similarly, one finds that $U_H (\lambda_H(x) \otimes 1)U_H^* = (\lambda_H \otimes \id)\alpha(x)$.

    Now take objects $H,K$ in $\mathcal{C}$, and $T \in \mor_{\mathcal{C}}(H,K)$. Using lemma \ref{lem: calculation in A}.1, we find that $T$ intertwines $U_H$ and $U_K$.
\end{proof}

\begin{lemma} \label{lem: integrating intertwiners}
    Let $\sigma_i := 1_i \sigma$ and $\delta_i := 1_i \delta$. Consider the following multiplier in $\mathcal{M}(M_0)$. Define $\upsilon_i := \del(\delta_i^2) d_i^{-2} \sigma_i^{-2} \delta_i^{-2}$, and then let
    \[
    \upsilon := \sum_{i \in I} \upsilon_i \in \mathcal{M}(M_0)
    \]
    Then let $X,Y$ be two unitary $\alpha$-equivariant corepresentations of $\mathbb{G}$ on finite type Hilbert-$M$-bimodules $H,K$ respectively. Let $T: K \to H$ be an $M$-bimodular linear map such that $T \lambda_K(p) = T$ for some central projection $p \in M_0$. Then the operator
    \[
    \sum_{y \in \onb(1_i \cdot \mathcal{F}_0)} (\lambda_H(y) \otimes 1)X(T \otimes 1)Y^*(\lambda_K(y^*) \otimes 1)
    \]
    lies in the domain of $(\id \otimes \varphi_e)$. Applying $(\id \otimes \varphi_e)$, and summing over all $i \in I$ yields 
    \[
    \Phi_e(T) := (\id \otimes \varphi_e) [X (\lambda_H(\upsilon) T \otimes 1)Y^*],
    \]
    a well-defined bounded $M$-bimodular intertwiner of the corepresentations $X$ and $Y$.
\end{lemma}
\begin{proof}
    The proof consists of three parts
    \begin{enumerate}
        \item $(\id \otimes \varphi_e) \left[\sum_{y \in \onb(\mathcal{F}_0)}(\lambda_H(y) \otimes 1)X(T \otimes 1)Y^*(\lambda_K(y^*) \otimes 1) \right]$ is bounded and $M$-bimodular.
        \item $(\id \otimes \varphi_e) \left[\sum_{y \in \onb(\mathcal{F}_0)}(\lambda_H(y) \otimes 1)X(T \otimes 1)Y^*(\lambda_K(y^*) \otimes 1) \right]$ intertwines the corepresentations.
        \item $(\id \otimes \varphi_e) \left[\sum_{y \in \onb(\mathcal{F}_0)}(\lambda_H(y) \otimes 1)X(T \otimes 1)Y^*(\lambda_K(y^*) \otimes 1) \right]$ equals $\Phi_e(T)$.
    \end{enumerate}
    We prove these three parts as follows.
    \begin{enumerate}
        \item $M$-bimodularity is immediate by a simple calculation. We show boundedness. By linearity, it suffices to prove the claim for positive $T$, by an operator-valued version of the Cauchy-Schwarz inequality, it suffices to prove the claim when $X = Y$. Taking an appropriate multiple, we have that $T$ is bounded above by a scalar multiple of $\lambda_H(p)$ for some central projection $p \in M_0$, so by linearity, it suffices to assume $T = \lambda_H(1_i)$ for some $i \in I_a$, $a \in \mathcal{E}$. 

        Now, we calculate, using lemma \ref{lem: varphi on A_1}.
        \begin{align*}
            (\id &\otimes \varphi_e) \left[ \sum_{y \onb(\mathcal{F}_0)} (\lambda_H(y) \otimes 1) X (\lambda_H(1_i) \otimes 1)X^*(\lambda_H(y^*) \otimes 1) \right] \\
            =& \; (\id \otimes \varphi_e) \left[ \sum_{y \in \onb(\mathcal{F}_0)} (\lambda_H(y) \otimes 1) [(\lambda_H \otimes \id)\alpha(1_i)] XX^* [(\lambda_H \otimes \id)\alpha(1_i)] (\lambda_H(y^*) \otimes 1) \right] \\
            =& \; \sum_{y,t \in \onb(\mathcal{F}_0)} \lambda_H(yty^*) \varphi_e \left( \Aelt{1}{t}{1_i} \right) \\
            =& \; \del(e \delta^{-2} e^*) \del(1_i \delta^2) \sum_{y,t \in \onb(\mathcal{F}_0)} \lambda_H(y t y^*) \del(t^* z_a) \\
            =& \; \del(e \delta^{-2} e^*) \del(1_i \delta^2) \lambda_H(z_a)
        \end{align*}
        Clearly, this is bounded.

        \item Define for any $y,z \in M_0$ the operator
        \[
        S_{y,z} := (\id \otimes \varphi_e)[ (\lambda_H(y) \otimes 1)X(T \otimes 1)Y^*(\lambda_K(z^*) \otimes 1)].
        \]
        Using left invariance of $\varphi_e$, we calculate that for any $x \in M_0$, we have the following.
        \begin{align*}
            S_{y,z}& \otimes \Aelt{1}{y}{x} \Aelt{1}{z}{x}^* \\
            =& \; (1 \otimes \Aelt{1}{y}{x})(S_{y,z} \otimes 1)(1 \otimes \Aelt{1}{z}{x}^*) \\
            =& \; (1 \otimes \Aelt{1}{y}{x})[(\id \otimes \id \otimes \varphi_e)((\lambda_H(y) \otimes 1 \otimes 1)(\id \otimes \Delta)[X(T \otimes 1)Y^*](\lambda_K(z^*) \otimes 1 \otimes 1))](1 \otimes \Aelt{1}{z}{x}^*) \\
            =& \; (\id \otimes \id \otimes \varphi_e) \left[ (\lambda_H(y) \otimes \Aelt{1}{y}{x} \otimes 1)X_{(12)}X_{(13)}(T \otimes 1 \otimes 1)Y^*_{(13)}Y^*_{(12)}(\lambda_K(z^*) \otimes \Aelt{1}{z}{x}^* \otimes 1) \right]
        \end{align*}
        Summing the left hand side over $x,y,z \in \onb(\mathcal{F}_0)$ yields $T \otimes 1$. Summing the right hand side first over $y,z \in \onb(\mathcal{F}_0)$ yields
        \[
        X(S_{x,x} \otimes 1)Y^*
        \]
        by equivariance of the corepresentations $X,Y$. Then summing over $x \in \onb(\mathcal{F}_0)$ yields $X(T \otimes 1)Y^*$.

        \item We start by noting that  $\upsilon = \sum_{x \in \onb(\mathcal{F}_0)} x^* \delta^2 \mu^{-1}(x) \delta^{-2}$, which may be checked by direct computation. Using this, we calculate as follows.
        \begin{align*}
            X(\lambda_H(\upsilon) T \otimes 1)Y^* =& \; \sum_{x \in \onb(\mathcal{F}_0)}X(\lambda_H(x^*) T \lambda_K(\delta^2 \mu^{-1}(x) \delta^{-2}) \otimes 1)Y^* \\
            =& \; \sum_{x \in \onb(\mathcal{F}_0)}[(\lambda_H \otimes \id)\alpha(x^*)] X(T \otimes 1)Y^* [(\lambda_K \otimes \id)\alpha(\delta^2 \mu^{-1}(x) \delta^{-2})] \\
            =& \; \sum_{x,y,z \in \onb(\mathcal{F}_0)} (\lambda_H(y) \otimes \Aelt{1}{y}{x^*}) X(T \otimes 1)Y^* (\lambda(z) \otimes \Aelt{1}{z}{\delta^2 \mu^{-1}(x) \delta^{-2}})
        \end{align*}
        Now, applying $(\id \otimes \varphi_e)$, and using the modular automorphism $\kappa$ from proposition \ref{prop: modular automorphism on A}, we find the following.
        \begin{align*}
            (\id& \otimes \varphi_e)  \left[ \sum_{x,y,z \in \onb(\mathcal{F}_0)} \left( \lambda_H(y) \otimes \kappa \left( \Aelt{1}{z}{\delta^2 \mu^{-1}(x) \delta^{-2}} \right) \Aelt{1}{y}{x^*} \right) X(T \otimes 1)Y^*(\lambda_K(z) \otimes 1) \right] \\
            =& \; (\id \otimes \varphi_e)  \left[ \sum_{x,y,z \in \onb(\mathcal{F}_0)} (\lambda_H(y) \otimes \Aelt{2}{\mu(z) \otimes y}{x \otimes x^*}) X(T \otimes 1)Y^* (\lambda_K(z) \otimes 1) \right] \\
            =& \; (\id \otimes \varphi_e)  \left[ \sum_{y,z \in \onb(\mathcal{F}_0)} (\lambda_H(y) \otimes \Aelt{1}{\mu(z)y}{1}) X(T \otimes 1)Y^* (\lambda_K(z) \otimes 1) \right] \\
            =& \; (\id \otimes \varphi_e)  \left[ \sum_{y,z \in \onb(\mathcal{F}_0)} (\lambda_H(y) \otimes 1) X(T \otimes 1)Y^* (\del(z^*y^*)\lambda_K(z) \otimes 1) \right] \\
            =& \; (\id \otimes \varphi_e)  \left[ \sum_{y \in \onb(\mathcal{F}_0)} (\lambda_H(y) \otimes 1) X(T \otimes 1)Y^* (\lambda_K(y^*) \otimes 1) \right] 
        \end{align*}
        This proves the third claim.
    \end{enumerate}
    
\end{proof}

\begin{proposition} \label{prop: functor maps irreducibles to irreducibles}
    For any object $H,K$ in $\mathcal{C}$, the $M$-bimodular intertwiners of the equivariant corepresentations $U_H,U_K$ from proposition \ref{prop: functor from C to equivariant reps } are precisely given by $\mor_{\mathcal{C}}(H,K)$.
\end{proposition}
\begin{proof}
    We first show the proposition when $H=K$.
    
    Take $\xi,\eta \in H_0$ arbitrarily, and consider
    \begin{equation}
    T_{\xi,\eta} := \sum_{x,y \in \onb(\mathcal{F}_0)} E_{x \cdot \xi \cdot y^*, x \cdot \eta \cdot y^*}. \label{eq: minimal bimodular map}
    \end{equation}
    Note that this is an $M$-bimodular operator in $B(H)$, and moreover, the linear span of such operators is SOT-dense in the space of all $M$-bimodular operators of $B(H)$. Let us calculate, using the notation from lemma \ref{lem: integrating intertwiners}.
    \begin{align}
        \Phi_e(T_{\xi,\eta}) =& \; (\id \otimes \varphi_e)[U_H(\lambda_H(\upsilon)T_{\xi,\eta} \otimes 1)U_H^*] \nonumber \\
        =& \; (\id \otimes \varphi_e) \left[ \sum_{z \in \onb(\mathcal{F}_0)} (\lambda_H(z) \otimes 1)U_H(T_{\xi,\eta} \otimes 1)U_H^*(\lambda_H(z^*) \otimes 1) \right] \\
        =& \; \sum_{\substack{x,y,z \in \onb(\mathcal{F}_0) \\ \zeta,\theta \in \onb(H_0)}} E_{\zeta,\theta} \varphi_e \left( \Aelt{H}{z^* \cdot \zeta}{x \cdot \xi \cdot y^*} \Aelt{H}{z^* \cdot \theta}{x \cdot \eta \cdot y^*}^* \right) \nonumber \\
        =& \; \sum_{\substack{x,y,z \in \onb(\mathcal{F}_0) \\ \zeta,\theta \in \onb(H_0)}} E_{\zeta,\theta} \varphi_e \left( \Aelt{H \otimes \overline{H}}{z^* \cdot \zeta \otimes (\id \otimes \overline{z^* \cdot \theta})t_H(\vect{1})}{ x \cdot \xi \cdot y^* \otimes (\overline{x \cdot \eta \cdot y^*} \otimes \id)s_H(\vect{1}) } \right) \nonumber \\
        =& \; \sum_{\substack{x,y,z \in \onb(\mathcal{F}_0) \\ \zeta,\theta \in \onb(H_0) \\ V \in \bpi(L^2(M) \otimes H \otimes \overline{H} \otimes L^2(M), L^2(M))}} E_{\zeta,\theta} \overline{(\del \circ V)(\vect{1} \otimes z^* \cdot \zeta \otimes (\id \otimes \overline{z^* \cdot \theta})t_H(\vect{1}) \otimes \vect{1})} \nonumber \\
        & \cdot (\del \circ V)(\vect{e} \otimes x \cdot \xi \cdot y^* \otimes (\overline{x \cdot \eta \cdot y^*} \otimes \id)s_H(\vect{1}) \otimes \vect{e^*} ) \label{eq: intertwiners come from category}
    \end{align}
    Now, note that
    \[
    \sum_{y \in \onb(\mathcal{F}_0)} ( x \cdot \xi \cdot y^* \otimes (\overline{x \cdot \eta \cdot y^*} \otimes \id)s_H(\vect{1})) = P_{H \mtimes \overline{H}} ( x \cdot \xi \otimes (\overline{x \cdot \eta} \otimes \id)s_H(\vect{1})),
    \]
    where $P_{H \mtimes \overline{H}}$ is the projection $H \otimes \overline{H} \to H \mtimes \overline{H}$ from definition \ref{def: relative tensor product}. Hence, in \eqref{eq: intertwiners come from category}, we need only consider finitely many $V \in \bpi(L^2(M) \otimes H \mtimes \overline{H} \otimes L^2(M),L^2(M))$. For any such $V$, we define the Frobenius reciprocal
    \[
    \widetilde{V} := ((\del \circ V) \otimes \id_{L^2(M) \otimes H}) \circ (\id_{L^2(M) \otimes H} \otimes \id_{\overline{H}} \otimes \mult^*(\vect{1}) \otimes \id_H) \circ (\id_{L^2(M) \otimes H} \otimes t_H(\vect{1})),
    \]
    which is an element of $\End_{\mathcal{C}}(L^2(M) \otimes H)$.
    Then one checks by direct calculation that for fixed $\zeta,\theta \in H_0$
    \[
    \langle (\tr_{\mathcal{C}} \otimes \id)\widetilde{V} (\zeta) , \theta \rangle = \sum_{z \in \onb(\mathcal{F}_0)} (\del \circ V)(\vect{1} \otimes z^* \cdot \zeta \otimes (\id \otimes \overline{z^* \cdot \theta})t_H(\vect{1}) \otimes \vect{1}),
    \]
    where $\tr_{\mathcal{C}}$ denotes the categorical trace.
    It follows that the expression in \eqref{eq: intertwiners come from category} equals
    \begin{align*}
        \sum_{\substack{x \in \onb(\mathcal{F}_0) \\ V \in \bpi(L^2(M) \otimes H \mtimes \overline{H} \otimes L^2(M),L^2(M))}} [(\tr_{\mathcal{C}} \otimes \id)\widetilde{V}]^* (\del \circ V)(\vect{e} \otimes x \cdot \xi \mtimes (\overline{x \cdot \eta} \otimes \id)s_H(\vect{1}) \otimes \vect{e^*}).
    \end{align*}
    Since now $[(\tr_{\mathcal{C}} \otimes \id)\widetilde{V}]^* \in \End_{\mathcal{C}}(H)$ for every $V \in \mor_{\mathcal{C}}(L^2(M) \otimes H \mtimes \overline{H} \otimes L^2(M),L^2(M))$, we may conclude.

    Now, for arbitrary objects $H,K$ in $\mathcal{C}$, let $T: H \to K$ intertwine the corepresentations. Then the matrix
    \[
    [T] := \begin{pmatrix}
        0 & T \\
        0 & 0
    \end{pmatrix}
    \]
    intertwines $U_K \oplus U_H$, and hence this matrix is an element of $\End_{\mathcal{C}}(K \oplus H, K \oplus H)$. Now, $T$ equals the composition
    \begin{figure}[h!]
        \centering
        \begin{tikzcd}
            H \arrow[r] & K \oplus H \arrow[r, "{[T]}"] & K \oplus H \arrow[r] & K
        \end{tikzcd}
    \end{figure}

    which finishes the proof.
\end{proof}

\begin{proposition} \label{prop: functor is essentially surjective}
    For any unitary equivariant corepresentation $U$ of $\mathbb{G}$ on some finite type Hilbert-$M$-bimodule $H$, there exists an object $K$ in $\mathcal{C}$, and an $M$-bimodular unitary $T: K \to H$ which intertwines $U$ and $U_K$, as defined in proposition \ref{prop: functor from C to equivariant reps }.
\end{proposition}
\begin{proof}
    By Zorn's lemma, it suffices to find some object $K$ in $\mathcal{C}$ and some nonzero intertwiner from $U$ to $U_K$. Suppose by contradiction that these don't exist. Denote $\Uelt{\xi}{\eta} := (\overline{\xi} \otimes \id) U (\eta \otimes 1)$ for $\xi,\eta \in H$. By applying lemma \ref{lem: integrating intertwiners} to the bimodular map $T_{\eta,\eta'}$ as in \eqref{eq: minimal bimodular map}, we get that for any object $K$ in $\mathcal{C}$, any $\xi,\eta \in H$, any $\xi',\eta' \in K_0$
    \[
    \sum_{x,y \in \onb(\mathcal{F}_0)} \varphi_e \left( \Uelt{\xi}{\upsilon x \cdot \eta \cdot y^*} \Aelt{K}{\xi'}{x \cdot \eta' \cdot y^*}^* \right) = 0.
    \]
    In particular, by passing to an appropriate subobject, we can apply this to $K' = L^2(M) \otimes K \otimes L^2(M)$, whence we find that for any $p,q,r,s \in M_0$
    \[
    \sum_{x,y \in \onb(\mathcal{F}_0)} \varphi_e \left( \Uelt{\xi}{\upsilon x \cdot \eta \cdot y^*} \left( \Aelt{1}{p}{xq} \Aelt{K}{\xi'}{\eta'} \Aelt{1}{r}{s\mu^{-1/2}(y^*)} \right)^* \right) = 0.
    \]
    Letting $p,q,r,s$ be central projections in $M_0$ which increase to the identity, we find
    \[
    \sum_{x,y \in \onb(\mathcal{F}_0)} \overline{\del(x) \del(\mu^{-1/2}(y^*))} \varphi_e \left( \Uelt{\xi}{\upsilon x \cdot \eta \cdot y^*} \Aelt{K}{\xi'}{\eta'}^* \right) = 0,
    \]
    and hence
    \[
    \varphi_e \left( \Uelt{\xi}{\upsilon \cdot \eta} \Aelt{K}{\xi'}{\eta'}^* \right) = 0.
    \]
    Since the elements $\Aelt{K}{\xi'}{\eta'}$ span a dense subspace of $L^{\infty}(\mathbb{G})$, and $\varphi_e$ is faithful, we have found that $\Uelt{\xi}{\upsilon \cdot \eta} = 0$. As this holds for every $\xi,\eta \in H$, it follows that $U = 0$, which is absurd.
\end{proof}

Together, propositions \ref{prop: functor from C to equivariant reps }, \ref{prop: functor maps irreducibles to irreducibles}, and \ref{prop: functor is essentially surjective} show an equivalence between the categories $\mathcal{C}$ and $\mathcal{C}(M \curvearrowleft \mathbb{G})$. This finally shows theorem \ref{thm: unique quantum group action for equivariant category}.

\section{Quantum automorphism groups of connected locally finite discrete structures} \label{sc: qaut of CLF discrete structures}

In the previous section \ref{sc: reconstruction theorem}, we constructed an algebraic quantum group $\mathbb{G} = (\mathcal{A},\Delta)$ for any fixed $M$, a direct sum of matrix algebras, from any concrete unitary $2$-category of finite type Hilbert-$M$-bimodules $\mathcal{C}$ under reasonable assumptions, with an action $\alpha: M_0 \to \mathcal{M}(M_0 \otimes \mathcal{A})$. Moreover, we showed that $\mathcal{C}$ is equivalent to the category of unitary $\alpha$-equivariant corepresentations of $\mathbb{G}$ on finite type Hilbert-$M$-bimodules. In \cite{RollierVaes2024}, a special case of this approach was already used to define the quantum automorphism group of connected locally finite graph. This section takes it further, by constructing a unitary $2$-category of finite type Hilbert-$M$-bimodules out of any connected locally finite discrete structure on $M$. We will make precise what we mean by this below in subsection \ref{sbsc: building categories from discrete structures}, but one important example given in subsection \ref{sbsc: example quantum cayley graph} is that of quantum Cayley graphs as introduced in \cite{Wasilewski2023}.

\subsection{Building unitary 2-categories from discrete structures} \label{sbsc: building categories from discrete structures}

Keep again a fixed von Neumann algebra $M = \ell^{\infty} \bigoplus_{i \in I} M_i$ for some matrix algebras $M_i$ indexed by a set $I$.

\begin{definition} \label{def: locally finite map}
    Recall from \ref{def: discrete quantum space} the definition of $\mathcal{F}_n$, the finitely supported vectors in $L^2(M^{n+1})$. Let $T: \mathcal{F}_m \to \mathcal{F}_n$ be an $M$-bimodular linear map. We say $T$ is locally finite if for every $0 \leq k \leq n$ and $0 \leq l \leq m$, we have that for every $i \in I$,
    \[
    (1^{\otimes k} \otimes 1_i \otimes 1^{\otimes (n-k)}) \circ T \text{ and } T \circ (1^{\otimes l} \otimes 1_i \otimes 1^{\otimes (m-l)})
    \]
    are finite rank operators.
\end{definition}

\begin{definition} \label{def: reciprocals and shrinkages}
    Let $T: \mathcal{F}_n \to \mathcal{F}_m$ be an $M$-bimodular locally finite map. We denote
    \begin{itemize}
        \item $T^{>} := (T \otimes \id) \circ (\id^{\otimes n} \otimes \vect{1} \otimes \id) \circ (\id^{\otimes (n-1)} \otimes \mult^*)$ when $n \geq 1$,
        \item $T_{>} := (\id \otimes T) \circ (\id \otimes \vect{1} \otimes \id^{\otimes n}) \circ (\mult^* \otimes \id^{\otimes (n-1)})$ when $n \geq 1$,
        \item $_{<}T :=  (\mult \otimes \id^{\otimes (m-1)}) \circ (\id \otimes \del \otimes \id^{\otimes m}) \circ (\id \otimes T)$ when $m \geq 1$,
        \item $^{<}T := (\id^{\otimes (m-1)} \otimes \mult) \circ (\id^{\otimes m} \otimes \del \otimes \id) \circ (T \otimes \id)$ when $m \geq 1$,
    \end{itemize}
    and we call these maps the reciprocals of $T$. It is easy to see that they are well-defined and locally finite. For any $1 \leq i \leq m$ and $1 \leq j \leq n$, we call the maps
    \[
    (\id^{\otimes i} \otimes \del \otimes \id^{\otimes (m-i)}) \circ T \text{ and } T \circ (\id^{\otimes j} \otimes \vect{1} \otimes \id^{\otimes (n-j)})
    \]
    shrinkages of $T$, and note that these are also well-defined and locally finite.    
\end{definition}

\begin{definition} \label{def: connected set of maps}
    Let $\mathcal{T}_{n,m}$ be a set of $M$-bimodular linear maps $T: \mathcal{F}_m \to \mathcal{F}_n$ for every $n,m \in \mathbb{N}$. We denote by $\overline{\mathcal{T}} = \bigcup_{n,m \geq 0} \overline{\mathcal{T}}_{n,m}$ the smallest set of $M$-bimodular linear maps $T: \mathcal{F}_m \to \mathcal{F}_n$ which contains $\mathcal{T} \cup \{ \mult \}$, and is closed under linear combinations, composition, adjoints, the relative tensor product $\mtimes$, reciprocals, and shrinkages. We say the set $\mathcal{T} = \bigcup_{n,m \geq 0} \mathcal{T}_{n,m}$ is connected if $\mathcal{F}_1$ is contained in the union of the ranges of all maps in $\overline{\mathcal{T}}_{1,1}$.
\end{definition}

We are now ready to prove theorem \ref{thm: CLF discrete structures give good categories}, which we repeat below.

\begin{theorem} \label{thm: CLF discrete structures give good categories}
    Let $\mathcal{T}_{n,m}$ be a set of $M$-bimodular linear maps from $\mathcal{F}_m$ to $\mathcal{F}_n$ for every $n,m \in \mathbb{N}$ which are all locally finite in the sense of definition \ref{def: locally finite map}, and such that $\mathcal{T} = \bigcup_{n,m \geq 0} \mathcal{T}_{n,m}$ is connected in the sense of definition \ref{def: connected set of maps}. Then there exists a smallest concrete unitary $2$-category of Hilbert-$M$-bimodules $\mathcal{C}$, as in definition \ref{def: concrete unitary 2-category of M-bimodules}, which covers $L^2(M^n)$ for every $n \in \mathbb{N}$ in the sense of definition \ref{def: Category covers bimodule}, and which covers every $T \in \mathcal{T}$, $\mult$, and $(\id \otimes \del \otimes \id)$ in the sense of definition \ref{def: covering morphisms}. In particular, the category $\mathcal{C}$ satisfies the conditions of theorem \ref{thm: unique quantum group action for equivariant category}. the resulting algebraic quantum group acts faithfully on $M$.
\end{theorem}
\begin{proof}
    We denote as in definition \ref{def: connected set of maps} by $\overline{\mathcal{T}} = \bigcup_{n,m \geq 0} \overline{\mathcal{T}}_{n,m}$ the smallest set of $M$-bimodular linear maps $T: \mathcal{F}_m \to \mathcal{F}_n$ which contains $\mathcal{T} \cup \{ \mult \}$, and is closed under linear combinations, composition, adjoints, the relative tensor product $\mtimes$, reciprocals, and shrinkages. Recall that by $M$-bimodularity, $\overline{\mathcal{T}}_{0,0}$ acts by scalars on all matrix blocks $\vect{M_i} \subset \mathcal{F}_0$ for any $i \in I$. Denote by $T_i$ these scalars. We define an equivalence relation $\approx$ on $I$ given by
    \[
    i \approx j \text{ if and only if } T_i = T_j \text{ for every } T \in \overline{\mathcal{T}}_{0,0}.
    \]
    Then we let the set of equivalence classes of $\approx$ be indexed by a set $\mathcal{E}$, and we denote by $\{I_a | a \in \mathcal{E}\}$ the individual equivalence classes such that $I = \bigcup_{a \in \mathcal{E}} I_a$. We also denote $M_a := \ell^{\infty} \bigoplus_{i \in I_a} M_i$, and by $z_a$ the unit of $M_a$. We define the following for any $a,b \in \mathcal{E}$ and $n,m \in \mathbb{N}$.
    \begin{itemize}
        \item $\overline{\mathcal{T}}^{a-}_{n,m} := \{T \circ \lambda(z_a) | T \in \overline{\mathcal{T}}_{n,m} \}$

        \item $\overline{\mathcal{T}}^{-b}_{n,m} := \{T \circ \rho(z_b^{\op}) | T \in \overline{\mathcal{T}}_{n,m} \}$

        \item $\overline{\mathcal{T}}^{a-b}_{n,m} := \{T \circ \lambda(z_a) \circ \rho(z_b^{\op}) | T \in \overline{\mathcal{T}}_{n,m} \}$
    \end{itemize}
    Now, for any $x = x_0 \otimes \cdots \otimes x_n \in M_0^{n+1}$, we denote $\overline{x} := x_n^* \otimes \cdots \otimes x_0^*$, and for any $T: \mathcal{F}_n \to \mathcal{F}_m$, we denote by $\widetilde{T}: \mathcal{F}_m \to \mathcal{F}_n$ the morphism such that $\widetilde{T} (\vect{x}) = \overline{T^*(\vect{\overline{x}})}$. We will prove the following
    \begin{enumerate}
    
        \item For each $T \in \overline{\mathcal{T}}_{n,n}$, the maps
        \begin{align*}
            (\id \otimes \tr_n)[(\id \otimes (\mu^{-1})^{\otimes n}) \circ T] :=& \; \mult \circ (\id \otimes (\del \circ \mult) \otimes \id) \circ \cdots \circ (\id^{\otimes n} \otimes \del \otimes \id^{\otimes n}) \circ T^{>^n} \\
            (\tr_n \otimes \id)[(\mu^{\otimes n} \otimes \id) \circ T] :=& \; \mult \circ (\id \otimes (\del \circ \mult) \otimes \id) \circ \cdots \circ (\id^{\otimes n} \otimes \del \otimes \id^{\otimes n}) \circ T_{>^n}
        \end{align*}
        are elements of $\overline{\mathcal{T}}_{0,0}$.
        
        \item The elements of $\overline{\mathcal{T}}^{a-}_{n,m}$ and $\overline{\mathcal{T}}^{-b}_{n,m}$ define bounded $M$-bimodular linear operators from $\mathcal{F}_n$ to $\mathcal{F}_m$.
        
        \item For every $T \in \overline{\mathcal{T}}^{a-}_{n,m}$, the range projection $P$ of $T$ belongs to $\overline{\mathcal{T}}^{a-}_{n,n}$, and $P \overline{\mathcal{T}}^{a-}(n,n)P$ is a finite-dimensional $C^*$-algebra. A similar statement holds for $T \in \overline{\mathcal{T}}^{-b}_{n,m}$.
        
        \item We have that
        \[
        \overline{\mathcal{T}}^{a-b}_{n,m} \subset ( \overline{\mathcal{T}}^{a-}_{n,m} \cap \overline{\mathcal{T}}^{-b}_{n,m} )
        \]
        
        \item We have that $\left( \overline{\mathcal{T}}^{a-b}_{n,m} \right)^* = \overline{\mathcal{T}}^{a-b}_{m,n}$, as well as $\overline{\mathcal{T}}^{a-b}_{n,k} \circ \overline{\mathcal{T}}^{a-b}_{k,m} \subset \overline{\mathcal{T}}^{a-b}_{n,m}$, and
        \[
        \overline{\mathcal{T}}^{a-b}_{n,m} \otimes_{(M,\del)} \overline{\mathcal{T}}^{b-c}_{n',m'} \subset \overline{\mathcal{T}}^{a-c}_{n+n',m+m'}
        \]
        
        \item For any $T \in \overline{\mathcal{T}}^{a-b}_{n,m}$, we have $\Tilde{T} \in \overline{\mathcal{T}}^{b-a}_{m,n}$.
        
        \item Let $P \in \overline{\mathcal{T}}^{a-b}_{n,n}$ be a self-adjoint projection. Then let $R \in \overline{\mathcal{T}}^{b-a}_{n,n}$ be the range projection of $\widetilde{P}$, and define
        \begin{align*}
            s_P :=& \; (P \otimes_{(M,\del)} \widetilde{P})(\id^{\otimes n} \otimes \vect{1} \otimes \id^{\otimes n}) \cdots (\id \otimes \mult^*(\vect{1}) \otimes \id) \mult^* \\
            t_P :=& \; (R \otimes_{(M,\del)} P)(\id^{\otimes n} \otimes \vect{1} \otimes \id^{\otimes n}) \cdots (\id \otimes \mult^*(\vect{1}) \otimes \id) \mult^*
        \end{align*}
        we get that $(\mult (s_P^* \otimes \id) \otimes \id^{\otimes })(\id^{\otimes n} \otimes (\id \otimes t_P)\mult^*) = P$, $(\mult(t_P^* \otimes \id) \otimes \id^{\otimes n})(\id^{\otimes n} \otimes (\id \otimes s_P)\mult^*) = R$, i.e. $(s_P,t_P)$ satisfy the conjugate equations for $P,R$, and moreover $s_P s_P^* \leq P \otimes_{(M,\del)} R$, $t_P t_P^* \leq R \otimes_{(M,\del)} P$.
    \end{enumerate}
    The respective proofs of these statements are as follows.
    \begin{enumerate}
        \item We will show the statement for $(\id \otimes \tr_n)[(\id \otimes (\mu^{-1})^{\otimes n}) \circ T]$. The other is handled analogously. Take $T \in \overline{\mathcal{T}}_{n,n}$ arbitrarily, and consider $T_{(0)} := _{<^n}T \in \overline{\mathcal{T}}_{2n,0}$. Then inductively define $T_{(k+1)} := \mult \circ (\id \otimes \del \otimes \id) \circ (T_{(k)})_{>}^{>} \in \overline{\mathcal{T}}_{2(n-k)}$. Then it is clear that $T_{(n)} = (\id \otimes \tr_n)[(\id \otimes (\mu^{-1})^{\otimes n}) \circ T]$.
        
        \item Take $T \in \overline{\mathcal{T}}_{n,m}$ arbitrarily, and denote $T_a := T \circ \lambda(z_a)$ and $T^b := T \circ \rho(z_b)$. We will only show the claim for $T_a$, as the other is handled analogously. Take $\xi := \xi_0 \otimes \cdots \otimes \xi_n \in M_0^{n+1}$ and $\eta := \eta_0 \otimes \cdots \otimes \eta_m \in M_0^{m+1}$ simple tensors in $\mathcal{F}_n$ and $\mathcal{F}_m$ respectively. We note the following, where we will use a sort of Sweedler notation $\mult^*(\zeta) := \zeta^{(1)} \otimes \zeta^{(2)}$.
        \begin{align*}
            &\left\langle T^{>}(\vect{\xi_0 \otimes \cdots \otimes \xi_{n-1}}), \vect{\eta_0 \otimes \cdots \otimes \eta_m \otimes \mu(\xi_n)^*} \right\rangle \\
            =& \; \left\langle T \left(\vect{\xi_0 \otimes \cdots \otimes \xi_{n-1}^{(1)} \otimes 1} \right) \otimes \vect{\xi_{n-1}^{(2)}}, \vect{\eta_0 \otimes \cdots \otimes \eta_m \otimes \mu(\xi_n)^*} \right\rangle \\
            =& \; \del(\mu(\xi_n) \xi_{n-1}^{(2)}) \left\langle T \left(\vect{\xi_0 \otimes \cdots \otimes \xi^{(1)}_{n-1} \otimes 1} \right), \vect{\eta_0 \otimes \cdots \otimes \eta_m} \right\rangle \\
            =& \; \left\langle T(\vect{\xi_0 \otimes \cdots \otimes \xi_{n-1}\xi_n \otimes 1}), \vect{\eta_0 \otimes \cdots \otimes \eta_m} \right\rangle
        \end{align*}
        Inductively, this means
        \begin{align*}
            &\left\langle T^{>^n} \left( \vect{\xi_0^{(1)}} \right), \vect{\eta_0 \otimes \cdots \otimes \eta_m \mu(\xi_n)^* \otimes \mu(\xi_{n-1})^* \otimes \cdots \otimes \mu(\xi_1)^* \otimes (\mu(\xi_0^{(2)}))^*} \right\rangle \\
            =& \; \left\langle T^{>^{(n-1)}}(\vect{\xi_0 \otimes 1}), \vect{\eta_0 \otimes \cdots \otimes \eta_m \mu(\xi_{n})^* \otimes \cdots \otimes \mu(\xi_1)^*} \right\rangle \\
             \; \cdots & \text{ repeat previous step another } n-1 \text{ times } \cdots \\
            =& \; \left\langle T(\vect{\xi_0 \otimes \cdots \otimes \xi_{n-1} \otimes 1}), \vect{\eta_0 \otimes \cdots \otimes \eta_m \mu(\xi_n)^*} \right\rangle \\
            =& \; \left\langle T(\vect{\xi_0 \otimes \cdots \otimes \xi_{n-1} \otimes 1}), \rho(\mu^{1/2}(\xi_n)^{\op})^* \vect{\eta_0 \otimes \cdots \otimes \eta_m} \right\rangle \\
            =& \; \left\langle \rho(\mu^{1/2}(\xi_n)^{\op})T( \vect{\xi_0 \otimes \cdots \otimes \xi_{n-1} \otimes 1}), \vect{\eta_0 \otimes \cdots \otimes \eta_m} \right\rangle \\
            =& \; \left\langle T(\vect{\xi_0 \otimes \cdots \otimes \xi_n}), \vect{\eta_0 \otimes \cdots \otimes \eta_m} \right\rangle
        \end{align*}
        It then follows that $\norm{T_a} = \norm{T_a^{>^n}} = \norm{\left( T_a^{>^n} \right)^* T_a^{>^n}}^{1/2}$.
        However, $\left(T^{>^n}\right)^*T^{>^n} \in \overline{\mathcal{T}}(0,0)$, and hence it acts as a scalar on $M_a \subset \mathcal{F}_0$, from which it follows that $T_a$ is bounded.
    
        \item First we show that for any $i \in I_a$ the map
        \begin{align*}
            \overline{\mathcal{T}}^{a-}_{m,n} \to& B(1_i \cdot \mathcal{F}_n, 1_i \cdot \mathcal{F}_m) \\
            T \mapsto& (\lambda(1_i) \otimes \id^{\otimes m})T
        \end{align*}
        is injective. Suppose $(\lambda(1_i) \otimes \id^{\otimes m})T = 0$. Then, since $(\id \otimes \tr_n)[(\id \otimes (\mu^{-1})^{\otimes n})TT^*]$ is constant on $a$, it must be zero, from which it follows that 
        \begin{align*}
            0 =& \; (\id \otimes \tr_n)[(\id \otimes (\mu^{-1})^{\otimes n})TT^*] \\
            =& \; (^{<^m}T)(^{<^m}T)^*
        \end{align*}
        And hence $^{<^m}T = 0$, from which it follows that
        \begin{align*}
            T =& \; (^{<^m}T)^{>^m} \\
            =& \; 0^{>^m} \\
            =& \; 0
        \end{align*}
        Now, fix $T \in \overline{\mathcal{T}}^{a-}_{m,n}$, and $i \in I_a$. Let $P$ denote the range projection of $T$, and take $H_0$ to be the range of $(\lambda(1_i) \otimes \id^{\otimes m})T$. Since $T$ is locally finite, $H_0$ is a finite-dimensional Hilbert space. Then by the reasoning above, the map
        \begin{align*}
            T \overline{\mathcal{T}}_{n,n} T^* \to B(H_0): X \mapsto (\lambda(1_i) \otimes \id^{\otimes m})X
        \end{align*}
        is an injective $*$-homomorphism, from which it follows that  $T \overline{\mathcal{T}}_{n,n} T^*$ is a finite dimensional $C^*$-algebra. Therefore, $P$ must belong to this $C^*$-algebra, and in particular, $P \in \overline{\mathcal{T}}^{a-}_{m,m}$. By the same reasoning, 
        $P \overline{\mathcal{T}}_{m,m} P$ is a finite-dimensional $C^*$-algebra.

        The statement for $T \in \overline{\mathcal{T}}^{-b}_{m,n}$ is proven analogously.
        
        \item We will show $\overline{\mathcal{T}}^{a-b}_{m,n} \subset \overline{\mathcal{T}}^{a-}_{m,n}$. The claim $\overline{\mathcal{T}}^{a-b}_{m,n} \subset \overline{\mathcal{T}}^{-b}_{m,n}$ is handled analogously.

        Take $T \in \overline{\mathcal{T}}^{a-}_{m,n}$ arbitrarily. We will show
        \[ (\id^{\otimes m} \otimes \rho(z_b^{\op})) \circ T \in \overline{\mathcal{T}}^{a-}_{m,n} \]
        Pick $i \in I_a$ arbitrarily, and note that since $T$ is locally finite, we can find some finite set $\mathcal{E}_0$ such that for all $b' \in \mathcal{E} \backslash \mathcal{E}_0$, we have
        \[ (\lambda(1_i) \otimes 1^{\otimes (m-1)} \otimes \rho(1_{b'}^{\op}))T = 0 \]
        As $\mathcal{E}_0$ is finite, we can find a morphism $S \in \overline{\mathcal{T}}_{0,0}$ such that $S(\vect{1_j}) = \vect{1_j}$ for any $j \in I_b$, and $S(\vect{1_{j'}}) = 0$ for all $j' \in I_{b'}$ when $b' \in \mathcal{E}_0 \backslash \{b\}$. Now, note that
        \begin{align*}
            R := (\id^{\otimes m} \otimes S) \circ T =& \; \left( S \circ (_{<^m}T)\right)_{>^m} \in \overline{\mathcal{T}}^{a-}(m,n)
        \end{align*}
        But now, we claim $R = (\id^{\otimes m} \otimes \rho(z_b^{\op}))T$. To this end, first note that
        \[
        (\id^{\otimes m} \otimes \rho(z_b^{\op}))R = (\id^{\otimes m} \otimes \rho(z_b^{\op})S)T = (\id^{\otimes m} \otimes \rho(z_b^{\op}))T.
        \]
        On the other hand, for any $c \in \mathcal{E} \backslash \{b\}$, we can find an $Q \in \overline{\mathcal{T}}_{0,0}$ with $Q(\vect{1_{k}}) = \vect{1_{k}}$ for all $k \in I_{c}$ and $Q(\vect{1_j}) = 0$ for all $j \in I_b$. Then
        \begin{align*}
            (\id^{\otimes m} \otimes \rho(z_{c}^{\op}))R =& \; (\id^{\otimes m} \otimes \rho(z_{c}^{\op})) \circ (\id^{\otimes m} \otimes Q)R \\
            =& \; 0
        \end{align*}
        since $(\id^{\otimes m} \otimes Q)R = 0$. Indeed, we can fix $i \in I_a$ arbitrarily, and find that
        \begin{align*}
            (\lambda(1_i) \otimes \id^{\otimes m})(\id^{\otimes m} \otimes Q)R =& \; (\id^{\otimes m} \otimes QS)(\lambda(1_i) \otimes \id^{\otimes m})T \\
            =& \; \sum_{d \in \mathcal{E}_0} (\id^{\otimes m} \otimes (QS \circ \rho(z_d^{\op}))) (\lambda(1_i) \otimes \id^{\otimes m})T \\
            =& \; (\id^{\otimes m} \otimes (Q \circ \rho(z_b^{\op}))) (\lambda(1_i) \otimes \id^{\otimes m})T \\
            =& \; 0
        \end{align*}
        Then since $N \mapsto (\lambda(1_i) \otimes \id^{\otimes m})N$ is injective on $\overline{\mathcal{T}}^{a-}_{m,n}$, we get that $(\id^{\otimes m} \otimes Q)R = 0$. Hence, it follows that $(\id^{\otimes m} \otimes \rho(z_b^{\op}))T \in \overline{\mathcal{T}}^{a-}_{m,n}$.
        
        \item The first two claims are immediate. Using 4, we get that
        \begin{align*}
            \overline{\mathcal{T}}^{a-b}_{m,n} \otimes_{(M,\del)} \overline{\mathcal{T}}^{b-c}_{m',n'} \subset& \overline{\mathcal{T}}^{a-}_{m,n} \otimes_{(M,\del)} \overline{\mathcal{T}}^{-c}_{m',n'} \\
            =& \; \left( \overline{\mathcal{T}}_{m,n} \otimes_{(M,\del)} \overline{\mathcal{T}}_{m',n'} \right) \circ \left(\lambda(z_a) \otimes \id^{\otimes (n+n'-1)} \otimes \rho(z_c^{\op}) \right) \\
            \subset& \overline{\mathcal{T}}^{a-c}_{m+m',n+n'}
        \end{align*}
        which also proves the last claim.
        
        \item For starters, we will show that for any $T \in \overline{\mathcal{T}}_{m,n}$, $\widetilde{T} = {}_{<^m}T^{>^n}$. To this end, take simple tensors $\xi := \xi_0 \otimes \cdots \otimes \xi_n \in M_0^{\otimes (n+1)}$ and $\eta := \eta_0 \otimes \cdots \otimes \eta_m \in M_0^{\otimes (m+1)}$, then similarly to the   argument in $2$, we have the following.
        \begin{align*}
            \langle T_{>}(\vect{\xi_1 \otimes \cdots \otimes \xi_{n}}), \vect{\xi_0^* \otimes \eta_0 \otimes \cdots \otimes \eta_m} \rangle = \langle T(\vect{1 \otimes \xi_0\xi_1 \otimes \cdots \otimes \xi_n}), \vect{\eta} \rangle
        \end{align*}
        Combining this with what we already had in $2$, we find the following.
        \begin{align*}
            &\left\langle _{<^m}T^{>^n}(\vect{\overline{\eta}}), \vect{\overline{\xi}} \right\rangle \\
            &= \left\langle _{<^m}T \left( \vect{\eta_m^* \otimes \cdots \otimes \eta_0^*\mu^{-1}(\xi_0) \otimes \cdots \otimes \mu^{-1}(\xi_{n-1}) \otimes 1 } \right), \vect{\xi_n^*} \right\rangle \\
            &= \left\langle \vect{\eta_m^* \otimes \cdots \otimes \eta_0^* \mu^{-1}(\xi_0) \otimes \cdots \otimes \mu^{-1}(\xi_{n-1}) \otimes 1}, (T^*)_{>^m}(\vect{\xi_n^*}) \right\rangle \\
            &= \left\langle \vect{\eta_{m-1}^* \otimes \cdots \otimes \eta_0^* \mu^{-1}(\xi_0) \otimes \cdots \otimes \mu^{-1}(\xi_{n-1}) \otimes 1}, (T^*)_{>^{m-1}}(\vect{1 \otimes \eta_m\xi_n^*}) \right\rangle \\
             \; \cdots & \text{ repeat previous step another } m-1 \text{ times } \cdots \\
            &= \left\langle \vect{\eta_0^*\mu^{-1}(\xi_0) \otimes \mu^{-1}(\xi_1) \otimes \cdots \otimes \mu^{-1}(\xi_n) \otimes 1 }, T^*( \vect{ 1 \otimes \eta_1 \otimes \cdots \otimes \eta_m \xi_n^* })  \right\rangle \\
            &= \left\langle T \left( \vect{ (\mu^{-1})^{\otimes (n+1)}\xi} \right), \vect{\eta} \right\rangle
        \end{align*}
        Using this, we calculate as follows.
        \begin{align*}
            {}_{<^m} T^{>^n}(\xi) =& \; \sum_{\eta \in \onb(\mathcal{F}_n)} \left\langle {}_{<^m} T^{>^n}(\xi) , \eta \right\rangle \eta \\
            =& \; \sum_{\eta \in \onb(\mathcal{F}_n)} \left\langle T(\overline{\mu^{\otimes n+1} \eta}) , \overline{\xi} \right\rangle \eta \\
            =& \; \sum_{\eta \in \onb(\mathcal{F}_n)} \left\langle \overline{\mu^{\otimes n+1} \eta} , T^*(\overline{\xi}) \right\rangle \eta \\
            =& \; \sum_{\eta \in \onb(\mathcal{F}_n)} \left\langle \overline{T^*(\overline{\xi})} , \eta \right\rangle \eta \\
            =& \; \overline{T^*(\overline{\xi})}
        \end{align*}
        Hence, indeed ${}_{<^m} T^{>^n} = \widetilde{T}$.

        Now, for $T \in \overline{\mathcal{T}}^{a-b}_{m,n}$, write $T = (\lambda(z_a) \otimes \id^{\otimes (m-1)} \otimes \rho(z_b^{\op})) \circ R$ for some $R \in \overline{\mathcal{T}}_{m,n}$. Then
        \begin{align*}
            \widetilde{T} = \widetilde{R} \circ (\lambda(z_b) \otimes \id^{\otimes (m-1)} \otimes \rho(z_a^{\op})) \in \overline{\mathcal{T}}^{b-a}_{n,m}
        \end{align*}
        
        \item Since $\widetilde{P}$ is idempotent, i.e. identity on its range, and $R$ is its range projection, we get that $\widetilde{P}R = R$. Now note that 
        \[
        (\mult(t_P^* \otimes \id) \otimes \id^{\otimes n})(\id^{\otimes n} \otimes (\id \otimes s_P)\mult^*) = \widetilde{P}^2 R = R
        \]
        On the other hand,
        \[
        (\mult (s_P^* \otimes \id) \otimes \id^{\otimes })(\id^{\otimes n} \otimes (\id \otimes t_P)\mult^*) = P \widetilde{(\widetilde{P})^*} P = P^3 = P
        \]

        The other claims are immediate.
        
    \end{enumerate}
    Now, we are ready to define $\mathcal{C}$. As mentioned, we take $\mathcal{E}$ to be the set of $0$-cells of $\mathcal{C}$.  We take the objects of our category to be the ranges of projections in $\bigoplus_{a,b \in \mathcal{C}_0} \overline{\mathcal{T}}^{a-b}$, and morphisms from the range of $P$ to the range of $Q$ are given by $Q \overline{\mathcal{T}} P$. All remaining axioms are readily checked. Since the ranges of morphisms in $\overline{\mathcal{T}}^{a-b}_{1,1}$ span $\mathcal{F}_1$, we have that $L^2(M^2)$ is covered by $\mathcal{C}$. The map $(\id \otimes \del \otimes \id): M_0^{\otimes 3} \to M_0^{\otimes 2}$ is covered because $\overline{\mathcal{T}}$ is closed under shrinkages. all necessary axioms are now satisfied to do our quantum group reconstruction. We retrieve an algebraic quantum group $(\mathcal{A},\Delta)$, with all the structure introduced in section \ref{sc: reconstruction theorem}.

    Since the category $\mathcal{C}$ is generated by subobjects of $L^2(M^2)$, we get that the algebra $\mathcal{B}$ from the reconstruction (see subsection \ref{sbsc: *-algebra B obtained from category}) is generated by $\Gelt{1}{x \otimes y}{z \otimes t}$ for $x,y,z,t \in M_0$. Pushing everything through the isomorphism $\Theta_e^{-1}$ from proposition \ref{prop: A isomorphic to corner of B}, we find that $\mathcal{A}$ is generated by $\Aelt{1}{x}{y}$ for $x,y \in M_0$. It follows that the action $\alpha: M_0 \to \mathcal{M}(M_0 \otimes \mathcal{A})$ as defined in proposition \ref{prop: action of quantum group on M} is faithful.
\end{proof}

\subsection{Example: quantum Cayley graphs} \label{sbsc: example quantum cayley graph}

Introduced by Wasilewski in \cite{Wasilewski2023}, quantum Cayley graphs are quantum graphs associated to any pair $(\Gamma,P)$ consisting of a discrete quantum group $\Gamma$ and a generating projection $P \in c_{00}(\Gamma)$ in the sense of \cite[Definition 5.1]{Wasilewski2023}. For simplicity, we restrict our attention to central generating projections invariant under the unitary antipode, or equivalently, symmetric generating sets of irreducible corepresentations of $\widehat{\Gamma}$. We will show that the projection onto the edge space of a quantum Cayley graph is connected and locally finite in the sense of definitions \ref{def: connected set of maps}, \ref{def: locally finite map}. Hence, we can apply theorems \ref{thm: CLF discrete structures give good categories} and \ref{thm: unique quantum group action for equivariant category} to define their quantum automorphism groups. The reader will note that the construction in this subsection from definition \ref{def: qaut of quantum cayley graph} onwards applies to any connected locally finite quantum graph, and in particular, this generalises the results from \cite{RollierVaes2024}.

\begin{definition} \label{def: CLF quantum graph}
    Given any quantum graph $\Pi$ with $M$-bimodular `edge space projection' given by $P: \mathcal{F}_1 \to \mathcal{F}_1$, we say $\Pi$ is connected and locally finite if $P$ is locally finite in the sense of definition \ref{def: locally finite map}, and the singleton $\{P\}$ is connected in the sense of definition \ref{def: connected set of maps}.
\end{definition}

Let us recall theorem \ref{thm: quantum cayley graphs give quantum automorphism groups}, which we will prove in this subsection.

\begin{theorem} \label{thm: quantum cayley graphs give quantum automorphism groups}
    Let $\Pi$ be any connected locally finite quantum graph in the sense of definition \ref{def: CLF quantum graph} with adjacency matrix $\adj$. There is a universal quantum group acting on $M$ by $\alpha$ in such a way that
    \[
    \alpha \circ \adj = (\adj \otimes \id) \circ \alpha.
    \]
    In particular, one may take $\Pi$ to be the quantum Cayley graph associated by \cite[Definition 5.1]{Wasilewski2023} to any $\Gamma$, discrete quantum group, and $S \subset \irr(\widehat{\Gamma})$ a finite symmetric generating set of irreducible corepresentations of $\widehat{\Gamma}$.
\end{theorem}
This quantum group, which will be defined rigorously in definition \ref{def: qaut of quantum cayley graph}, will be called $\qaut(\Pi)$.

We start by recalling some notions from \cite{Wasilewski2023}.

Fix a discrete quantum group $\Gamma = (M,\Delta_{\Gamma})$ with invariant weights $\varphi_{\Gamma}$ and $\psi_{\Gamma}$, both normalised to be delta forms\footnote{This is possible for any discrete quantum group.}. Denote again $M = \ell^{\infty} \bigoplus_{i \in I} M_i$, where $I$ is this time a maximal set of pairwise nonequivalent irreducible corepresentations of $\widehat{\Gamma}$ on Hilbert spaces $H_i$ of dimension $d_i$. Then $(M_i)_{i \in I}$ are the matrix algebras $B(H_i)$, with $M_i \cong \mathbb{C}^{d_i \times d_i}$. We again denote $M_0 \subset M$ the finitely supported part of $M$. We let $\sigma \in \mathcal{M}(M_0)$ be the multiplier such that $\psi_{\gamma} = \tr_M( \cdot \sigma)$, where $\tr_M$ denotes again the Markov trace from definition \ref{def: delta-form}. We will again denote $\mu^t(x) = \sigma^{-t}x\sigma^t$ for any $x \in \mathcal{M}(M_0)$. We will also denote again by $E^i_{k,l}$ the $k,l$-entry of $M_i$, but throughout this section we will work in a basis such that $\sigma$ is diagonal.

Fix a finite symmetric generating set $S \subset I$ of irreducible corepresentations of $\widehat{\Gamma}$. We denote $p := \sum_{i \in S} 1_i \in M_0$. Through \cite[Definition 5.1]{Wasilewski2023}, we associate to the projection $p$ the quantum adjacency matrix
\begin{equation}
    \adj: M \to M: x \mapsto (\psi_{\Gamma}(p \cdot ) \otimes \id ) \circ \Delta(x). \label{eq: defining adjacency matrix}
\end{equation}
By \cite[Lemma 5.8]{Wasilewski2023}, we can describe $\adj$ more explicitly as
\begin{equation}
    \adj(E^i_{k,l}) = \sum_{s \in S}\sum_{j \in I} \frac{\dim_q(i) \dim_q(s)}{\dim_q(j)} \sum_{n=1}^{d_s} \sum_{m=1}^{m(j,i \otimes s)} V_{n,m} E^i_{k,l} V_{n,m}^* \label{eq: explicit description of adjacency matrix}
\end{equation}
where $\dim_q$ denotes the quantum dimension, $m(j,i \otimes s)$ is the multiplicity of $j$ in $i \otimes s$, $V_m$ are elements of a $\bpi(H_j, H_i \otimes H_s)$, and $V_{n,m}(v) := V_m^*(v \otimes e_n)$ where $e_n$ is the $n$-th vector of the basis of $H_s$. Finally, we also get the following $M$-bimodular `edge space projection' $\mathcal{F}_1 \to \mathcal{F}_1$.
\begin{equation}
    P := \sum_{x \in \onb(\mathcal{F}_0)} \rho(x^{\op}) \otimes \lambda \left( \adj \left( \mu^{1/2}(x)^* \right) \right) \label{eq: edge projection quantum cayley graph}
\end{equation}
This is the operator $\Psi^{KMS}(\adj)$ from \cite[Lemma 3.3]{Wasilewski2023} but with its legs flipped. By \cite[Proposition 3.7]{Wasilewski2023} it is a self-adjoint idempotent. We call the resulting quantum Cayley graph $\Pi$.

\begin{proposition}
    For any quantum Cayley graph $\Pi$, its edge space projection $P$ is locally finite as in definition \ref{def: locally finite map}, and $\{P\}$ is connected as in definition \ref{def: connected set of maps}, i.e. any quantum Cayley graph $\Pi$ is connected and locally finite in the sense of definition \ref{def: CLF quantum graph}.
\end{proposition}
\begin{proof}
    Clearly for any $i \in I$ we get that $P \circ (1_i \otimes 1)$ is finite rank. Since $\adj$ restricts to a map $M_0 \to M_0$ by $\eqref{eq: explicit description of adjacency matrix}$, we also get that $P \circ (1 \otimes 1_i)$ is finite rank for any $i \in I$. Hence, $P$ is indeed locally finite.

    Consider $P_1 = P$, and define inductively
    $P_n = (\id \otimes \del \otimes \id) \circ (P \mtimes P_{n-1}) \circ (\id \otimes \vect{1} \otimes \id)$. To show that $\{P\}$ is connected, it suffices to show that for any $y,z \in M_0$ there is some $n \in \mathbb{N}$ such that $P_n \vect{y \otimes z} \neq 0$. By linearity and $M$-bimodularity of $P_n$, it suffices to check this for $y = E^i_{k,k}$ and $z = E^j_{m,m}$ for some $i,j \in I$, $1 \leq n \leq d_i$ and $1 \leq m \leq d_j$. A simple calculation shows that
    \[
    P_n = \sum_{x \in \onb(\mathcal{F}_0)} \rho(x^{\op}) \otimes \lambda \left( \adj^n \left( \mu^{1/2}(x)^* \right) \right).
    \]
    Now, using the fact that $\sigma$ is diagonal, we get that
    \begin{align*}
        (\del \otimes \id) P_n \vect{E^i_{k,k} \otimes E^j_{p,p}} &= \sum_{x \in \onb(\mathcal{F}_0)} \del \left( E^i_{k,k} \mu^{-1/2}(x) \right) \vect{\adj^n \left( \mu^{1/2}(x)^* \right) E^j_{m,m}} \\
        &= \adj^n(E^i_{k,k}) E^j_{m,m}.
    \end{align*}
    So it suffices to find some $n \in \mathbb{N}$ such that $E^j_{m,m}\adj^n(E^i_{k,k}) E^j_{m,m} \neq 0$. Denote by $H_S := \bigoplus_{s \in S} H_s$ and by $\gamma$ the corepresentation of $\widehat{\Gamma}$ on $H_S$. By a successive application of the formula \eqref{eq: explicit description of adjacency matrix}, we find that $E^j_{m,m}\adj^n(E^i_{k,k}) E^j_{m,m}$ equals
    \[
    \frac{\dim_q(i) \dim_q(\gamma)^n}{\dim_q(j)} \sum_{r \in \onb(H_S^{\otimes n})} \sum_{l=1}^{m(j,i \otimes \gamma^{\otimes n})} E^j_{m,m} V_{r,l} E^i_{k,k} V_{r,l}^* E^j_{m,m}.
    \]
    Where $m(j,i \otimes \gamma^{\otimes n})$ denotes the multiplicity of $j$ in $i \otimes \gamma^{\otimes n}$, $V_l$ are elements of some $\bpi(H_j,H_i \otimes H_S^{\otimes n})$, and $V_{r,l}(v) = V_l^*(v \otimes r)$. Suppose now this is zero for every $n \in \mathbb{N}$. Then by positivity, we get that $E^j_{m,m}V_{r,l}E^i_{k,k} = 0$ for every $V_{r,l}$, meaning $\langle V_{r,l} (e_k), e_m \rangle = 0$ and hence $\langle e_k \otimes r, V(e_m) \rangle = 0$ for every $r \in H_S^{\otimes n}$ and every $V: H_j \to H_i \otimes H_S^{\otimes n}$. This contradicts the fact that $S$ is a generating set.
\end{proof}

\begin{definition} \label{def: qaut of quantum cayley graph}
    Let $\Pi$ be any connected locally finite quantum graph (e.g. a quantum Cayley graph), and let $\mathcal{C}(\Pi)$ be the smallest concrete unitary $2$-category of finite type Hilbert-$M$-bimodules which covers $L^2(M^2)$ as well as the maps $(\id \otimes \del \otimes \id), \mult$ and $P$, the edge space projection of $\Pi$. This exists by theorem \ref{thm: CLF discrete structures give good categories}. Then apply theorem \ref{thm: unique quantum group action for equivariant category} to the category $\mathcal{C}(\Pi)$ and find an algebraic quantum group $(\mathcal{A},\Delta)$, which we will call $\qaut(\Pi)$.
\end{definition}

To prove theorem \ref{thm: quantum cayley graphs give quantum automorphism groups}, we will show that $\qaut(\Pi)$ is the universal quantum group admitting an action on $M$ which preserves the quantum graph structure.

Recall from the preceding reconstruction that $\qaut(\Pi) = (\mathcal{A},\Delta)$ for the algebra $\mathcal{A}$ generated by $\Aelt{n}{x}{y}$, $x,y \in M_0^{\otimes n}$ under the relations
\begin{itemize}
    \item $\Aelt{n}{x}{y}\Aelt{m}{z}{t} = \Aelt{n+m}{x \otimes z}{y \otimes t}$,
    
    \item $\Aelt{1}{x}{y}^* = \Aelt{1}{\mu(x)^*}{y^*}$,

    \item $\Aelt{1}{1}{x} = \Aelt{1}{x^*}{1} = \del(x)1$ strictly,
    
    \item and for any $T: \mathcal{F}_n \to \mathcal{F}_m$ which is covered by $\mathcal{C}(\Pi)$ we have $\Aelt{m+1}{T(x)}{y} = \Aelt{n+1}{x}{T^*(y)}$ for any $x \in M_0^{n+1}$ and $y \in M_0^{m+1}$
\end{itemize}

\begin{proposition} \label{prop: qaut preserves graph structure}
    The action $\alpha$ of $\qaut(\Pi)$ on $M$ as defined in proposition \ref{prop: action of quantum group on M} satisfies
    \[
    \alpha \circ \adj = (\adj \otimes \id) \circ \alpha.
    \]
\end{proposition}
\begin{proof}
    As $P$ is a morphism in $\mathcal{C}(\Pi)$, we find that for any $x,y,z,t \in M_0$ we have
    \[
    \sum_{r \in \onb(\mathcal{F}_0)} \Aelt{2}{z \mu^{-1/2}(r) \otimes \adj \left(\mu^{1/2}(r)^* \right)x}{y \otimes t} = \sum_{r \in \onb(\mathcal{F}_0)} \Aelt{2}{z \otimes x}{y \mu^{-1/2}(r) \otimes \adj \left(\mu^{1/2}(r)^* \right)t}.
    \]
Letting $z,t$ increase to $1$, we get that the right hand side becomes $\Aelt{1}{x}{\adj(y)}$, while the left hand side gives $\Aelt{1}{\adj^*(x)}{y}$.\footnote{Contrary to the convention in \cite{Wasilewski2023}, we are still working with the GNS, not the KMS, inner product, and $\adj^*$ is the adjoint with respect to this inner product.} It follows that the action $\alpha$ satisfies
\[
\alpha \circ \adj = (\adj \otimes \id) \circ \alpha.
\]
\end{proof}

\begin{proposition} \label{prop: qaut of cayley is universal}
    Let $\mathbb{G}$ be any locally compact quantum group acting by $\beta$ on $M$ in such a way that 
    \[
    \beta \circ \adj = (\adj \otimes \id) \circ \beta,
    \]
    then there exists a homomorphism of quantum groups $\pi: \mathcal{A} \to L^{\infty}(\mathbb{G})$ such that
    \[
    \beta  = (\id \otimes \pi) \circ \alpha.
    \]
\end{proposition}
\begin{proof}
    Consider $U \in B(L^2(M)) \otimes L^{\infty}(\mathbb{G})$, the unitary implementation of the action $\beta$. For reference on this, see \cite{Vaes01}. Denote $u_{x,y}$ the $\vect{x},\vect{y}$-entry of $U$. Denote by $\mathcal{T}_{n,m}$ the set of all $M$-bimodular intertwiners from $U^{\otimes (m+1)}$ to $U^{\otimes (n+1)}$, and by $\mathcal{T} := \bigcup_{n,m \in \mathbb{N}} \mathcal{T}_{n,m}$. Now, seeing as 
    \[
    P = (\id \otimes \mult) \circ (\id \otimes \adj \otimes \id) \circ (\mult^* \otimes \id)
    \]
    and $\adj$ intertwines $U$, and $\mult \in \mathcal{T}_{0,1}$, we get that $P \in \mathcal{T}_{1,1}$. It is an easy exercise to show that $\mathcal{T}$ is closed under linear combinations, composition, adjoints, the relative tensor product $\mtimes$, reciprocals and shrinkages (see definition \ref{def: reciprocals and shrinkages}). It follows that every projection $T \in B(L^2(M^n))$ which is covered by $\mathcal{C}(\Gamma,\Pi)$ is also an intertwiner of $U^{\otimes n}$, and hence by the universal property of $\mathcal{A}$, we find the required homomorphism $\pi$ with $\pi \left( \Aelt{1}{x}{y} \right) = u_{x,y}$.
\end{proof}
Together, propositions \ref{prop: qaut preserves graph structure} and \ref{prop: qaut of cayley is universal} prove theorem \ref{thm: quantum cayley graphs give quantum automorphism groups}.

\printbibliography[title=Bibliography]

\end{document}